\documentclass[11pt]{article}
\usepackage{amsfonts}
\usepackage{amssymb}
\usepackage{graphicx}
\usepackage{amsmath}
\usepackage{amsthm}
\usepackage{lscape}
\usepackage[usenames]{color}
\usepackage{tikz}

\usepackage{multirow}
\usepackage{makecell}
\usepackage{prettyref,soul}
\newrefformat{eq}{(\ref{#1})}
\newrefformat{chap}{Chapter~\ref{#1}}
\newrefformat{sec}{Section~\ref{#1}}
\newrefformat{algo}{Algoritheorem~\ref{#1}}
\newrefformat{fig}{Figure~\ref{#1}}
\newrefformat{tab}{Table~\ref{#1}}
\newrefformat{rmk}{Remark~\ref{#1}}
\newrefformat{clm}{Claim~\ref{#1}}
\newrefformat{def}{Definition~\ref{#1}}
\newrefformat{cor}{Corollary~\ref{#1}}
\newrefformat{lmm}{Lemma~\ref{#1}}
\newrefformat{prop}{Proposition~\ref{#1}}
\newrefformat{app}{Appendix~\ref{#1}}
\newrefformat{ex}{Example~\ref{#1}}
\newrefformat{exer}{Exercise~\ref{#1}}
\newrefformat{soln}{Solution~\ref{#1}}
\newrefformat{cond}{Condition~\ref{#1}}

\newtheorem{theorem}{Theorem}
\newtheorem{corollary}{Corollary}
\newtheorem{proposition}{Proposition}
\newtheorem{lemma}{Lemma}
\newtheorem{example}{Example}
\newtheorem{remark}{Remark}
\newtheorem{definition}{Definition}
\newcommand{\beq}{\begin{equation}}
\newcommand{\eeq}{\end{equation}}
\newcommand{\beas}{\begin{eqnarray*}}
	\newcommand{\eeas}{\end{eqnarray*}}
\newcommand{\bea}{\begin{eqnarray}}
\newcommand{\eea}{\end{eqnarray}}
\newcommand{\bei}{\begin{itemize}}
	\newcommand{\eei}{\end{itemize}}
\newcommand{\ben}{\begin{enumerate}}
	\newcommand{\een}{\end{enumerate}}
\newcommand{\bet}{\begin{theorem}}
	\newcommand{\eet}{\end{theorem}}
\newcommand{\bel}{\begin{lemma}}
	\newcommand{\eel}{\end{lemma}}
\newcommand{\bep}{\begin{proposition}}
	\newcommand{\eep}{\end{proposition}}
\newcommand{\bed}{\begin{definition}}
	\newcommand{\eed}{\end{definition}}
\newcommand{\bec}{\begin{corollary}}
	\newcommand{\eec}{\end{corollary}}
\newcommand{\bex}{\begin{example}}
	\newcommand{\eex}{\end{example}}
\newcommand{\RR}{\mathbb{R}}

\newcommand{\hf}{{1\over 2}}
\newcommand{\row}{\textrm{row}}
\newcommand{\col}{\textrm{col}}

\usepackage{natbib}
\newcommand{\diag}{\textrm{diag}}
\newcommand{\op}{\rm{op}}
\newcommand{\f}{\rm{F}}

\newcommand{\ceil}[1]{\lceil #1 \rceil}
\newcommand{\floor}[1]{\lfloor #1 \rfloor}
\newcommand{\paraspp}{\mathcal{P}_\alpha(\eigenbd,M)}
\newcommand{\paraspq}{\mathcal{Q}_\alpha(\eigenbd,M)}
\newcommand{\nparaspp}{\mathcal{P}'_\alpha(\eigenbd,M)}
\newcommand{\nparaspq}{\mathcal{Q}'_\alpha(\eigenbd,M)}
\newcommand{\calp}{\mathcal{P}}
\newcommand{\calq}{\mathcal{Q}}

\newcommand{\lopratefull}{\frac{\tau^2}{16} n^{-1} \min\{n^{\frac{1}{2\alpha}},\frac{p}{2}\}}
\newcommand{\lopelem}{\tau n^{-\frac{1}{2}}}
\newcommand{\lopband}{n^{\frac{1}{2\alpha}}}

\newcommand{\eopratefull}{\frac{\tau^2}{32}(n^{-\frac{2\alpha}{2\alpha+1}} + \frac{\log p}{n})}
\newcommand{\eopband}{\ceil{n^{\frac{1}{2\alpha+1}}}}
\newcommand{\eopelem}{\tau (nk)^{-\frac{1}{2}}}

\newcommand{\efrateup}{n^{-\frac{2\alpha+1}{2\alpha+2}}}
\newcommand{\efratefull}{\frac{\tau^{2}}{32}n^{-1}\min\left\{n^{\frac{1}{2\alpha+2}},\frac{p}{2}\right\}}
\newcommand{\efelem}{\tau n^{-\frac{1}{2}}}
\newcommand{\taubd}{\min\{M, \frac{1}{4}\eigenbd^{-1},\eigenbd^{\frac{1}{2}}-1\}}

\newcommand{\indc}[1]{\mathbf{1}\left({#1}\right)}
\newcommand{\abs}[1]{|{#1}|}
\newcommand{\norm}[1]{\|{#1} \|}
\newcommand{\fnorm}[1]{\|#1\|_{\rm F}}
\newcommand{\opnorm}[1]{\|#1\|_{\rm op}}
\newcommand{\vecnorm}[1]{\|#1\|_{\rm 2}}
\newcommand{\vecnormsq}[1]{\|#1\|_{\rm 2}^2}
\newcommand{\rownorm}[1]{\|#1\|_{\rm \infty}}
\newcommand{\colnorm}[1]{\|#1\|_{\rm 1}}

\newcommand{\fnormsq}[1]{\|#1\|_{\rm F}^2}
\newcommand{\opnormsq}[1]{\|#1\|_{\rm op}^2}

\newcommand{\pb}{\mathbb{P}}
\newcommand{\ep}{\mathbb{E}}
 
\newcommand{\var}{\textsf{Var}} 

\newcommand{\sgn}{\textsf{sgn}}
\newcommand{\mz}{\mathbf{Z}}
\newcommand{\mx}{\mathbf{X}}

\newcommand{\eigenbd}{\eta}

\newcommand{\projj}{\mathbf{P}_\eigenbd}
\newcommand{\cut}[2]{\mathbf{C}^{#1}_{#2}}
\newcommand{\extend}[2]{\mathbf{E}^{#1}_{#2}}
\begin{document}

\title{Minimax  Estimation of Large Precision Matrices with Bandable Cholesky Factor}
\author{Yu Liu and Zhao Ren\\
	University of Pittsburgh$^{1}$}
\date{}
\maketitle

\begin{abstract}
Last decade witnesses significant methodological and theoretical advances in estimating large precision matrices. In particular, there are scientific applications such as longitudinal data, meteorology and spectroscopy in which the ordering of the variables can be interpreted through a bandable structure on the Cholesky factor of the precision matrix. However, the minimax theory has still been largely unknown, as opposed to the well established minimax results over the corresponding bandable covariance matrices. In this paper, we focus on two commonly used types of parameter spaces, and develop the optimal rates of convergence under both the operator norm and the Frobenius norm. A striking phenomenon is found: two types of parameter spaces are fundamentally different under the operator norm but enjoy the same rate optimality under the Frobenius norm, which is in sharp contrast to the equivalence of corresponding two types of bandable covariance matrices under both norms. This fundamental difference is established by carefully constructing the corresponding minimax lower bounds. Two new estimation procedures are developed: for the operator norm, our optimal procedure is based on a novel local cropping estimator targeting on all principle submatrices of the precision matrix while for the Frobenius norm, our optimal procedure relies on a delicate regression-based thresholding rule. Lepski's method is considered to achieve optimal adaptation. We further establish rate optimality in the nonparanormal model. 
Numerical studies are carried out to confirm our theoretical findings.
\end{abstract}

\footnotetext[1]{%
Department of Statistics, University of Pittsburgh, Pittsburgh, PA 15260. e-mail: yul125@pitt.edu, zren@pitt.edu} 



\newpage

\section{Introduction}\label{sec: introduction}

Covariance matrix plays a fundamental role in many important multivariate
statistical problems. They include the principal component analysis, linear and quadratic discriminant analysis,
clustering analysis, regression analysis and conditional dependence relationship studies in graphical models. During the last two decades, with the advances of technology, it is very common that the datasets are high-dimensional (the dimension $p$ can be much larger than the sample size $n$) in many applications such as genomics, fMRI data, astrophysics, spectroscopic imaging, risk management, portfolio allocation and numerical weather forecasting \citep{heyer1997application,eisen1998cluster,hamill2001distance,ledoit2003improved,schafer2005shrinkage,padmanabhan2016estimating}. It has been well-known that the sample covariance matrix performs poorly and can yield to invalid conclusions in the high-dimensional settings. For example, see \cite{wachter1976probability,wachter1978strong,johnstone2001distribution,karoui2003largest,paul2007asymptotics,johnstone2009consistency} for details on the limiting behaviors of the spectra of sample covariance matrices when both $n$ and $p$ increase.

To avoid the curse of dimensionality, certain structural assumptions are almost necessary in order to estimate the covariance matrix or its inverse, the precision matrix, consistently. In this paper, we consider large precision matrix estimation with bandable Cholesky factor. Its connection with other structures are discussed at the end of introduction. Both the operator norm loss ($\opnorm{S}=\sup_{\norm{x}_2=1}\norm{Sx}_2$) and the Frobenius norm loss ($\fnorm{S}=(\sum_{i,j} s_{ij}^2)^\hf$) are investigated. 

We begin with introducing the bandable Cholesky factor of the precision matrix.  Assume that $\mx=(X_1, \dots X_p)^T$ is a centered $p$-variate random vector with covariance matrix $\Sigma$. Let $\mathbf{a}_i=(a_{i1},\ldots,a_{i(i-1)})^T$ be the coefficients of the population regression of $X_i$ on its previous variables $\mx_{1,i-1}=(X_1, X_2 \dots X_{i-1})^T$. In other words, $\hat{X_i}=\sum_{t=1}^{i-1} a_{it}X_t=\mx_{1,i-1}^T\mathbf{a}_i$ is the linear projection of $X_i$ on $\mx_{1,i-1}$ in population (Define $\hat{X_1}=0$). Set $A$ as the lower triangular matrix with zeros on the diagonal and zero-padded coefficients $(\mathbf{a}_i^T,\textbf{0})$ arranged in the rows. Denote
the residual $\boldsymbol{\epsilon}=\mx - \hat{\mx}=(I-A)\mx$ and $D=\var(\boldsymbol{\epsilon})$. The regression theory implies the residuals are uncorrelated, and thus the matrix $D$ is diagonal. The modified Cholesky decomposition of $\Omega$ is
\begin{equation}\label{eq: def cd of omega}
\Omega=\Sigma^{-1}=(I-A)^TD^{-1}(I-A),
\end{equation}
where $I-A$ is the Cholesky factor of $\Omega$.  There is a natural order on the variables based on the above Cholesky decomposition. Indeed, the well-known AR($k$) model can be characterized by the $k$-banded Cholesky factor $A \equiv [a_{ij}]_{p \times p}$ of the precision matrix in which $a_{ij}=0$ if $i-j>k$. Inspired by the auto-regression model, we consider the bandable structures imposed on the Cholesky factor. More specifically,  for $M > 0$, $\eigenbd > 1$ we define the parameter space $\paraspp$ of precision matrices by
\begin{equation}\label{eq: def paraspp}
\begin{split}
\paraspp=\Big\{ \Omega : \quad  &\eigenbd^{-1} \leq \lambda_{\rm{min}}(\Omega) \leq \lambda_{\rm{max}}(\Omega) < \eigenbd, \\
&\max_{i}\sum_{j<i-k}|a_{ij}|<Mk^{-\alpha}, \quad  k \in [p]
\Big\}.
\end{split}
\end{equation}
Here, $\lambda_{\rm{max}}(\Omega)$, $\lambda_{\rm{min}}(\Omega)$ are the maximum and minimum eigenvalues of $\Omega$ and the index set $[p]=\{1, 2, \dots, p\}$. We follow the convention that the sum over an empty set of indices is equal to zero when $i-k\leq 1$. This parameter space was first proposed in \citep{bickel2008regularized}. The parameter $\alpha $ specifies how fast the sequence $a_{ij}$ decays to zero as $j$ goes away from $i$. The covariance matrix estimation problem has been extensively studied when a similar bandable structure is imposed on the covariance matrix (e.g., \citep{bickel2008regularized,cai2010optimal}). Unlike the order in these bandable covariance matrices, in which large distance $|i-j|$ implies nearly independence, the order in bandable Cholesky factor encodes a natural auto-regression interpretation in the sense that the coefficients $a_{ij}$ is close to zero when $i-j>0$ is large. 

Although several approaches have been developed to estimate the precision matrix with bandable Cholesky factor, the optimality question remains mostly open, partially due to the following two reasons. (i) Intuitively, one would expect the minimax rate of convergence over $\paraspp$ under the operator norm to be the same as that over the class of bandable covariance matrices with the same decay parameter $\alpha$. Under sub-Gaussian assumptions, \cite{cai2010optimal} established the optimal rate of convergence $\mathbb{E}\opnormsq{\tilde{\Omega}-\Omega} \asymp n^{\frac{-2\alpha+1}{2\alpha}}+ \frac{\log p}{n}$ uniformly for all bandable covariance matrices $\Sigma=\Omega^{-1}=[\sigma_{ij}]_{p\times p}$ with bounded spectra such that $\max_{i}\sum_{|j-i|>k}|\sigma_{ij}|<Mk^{-\alpha}$, $k \in [p]$. To establish such a rate of convergence for $\paraspp$, \cite{lee2017estimating} provided a lower bound with the matching rate. However, we show a surprising result in this paper that estimation over $\paraspp$ is a much harder task than that over bandable covariance matrices. Therefore, the lower bound in \cite{lee2017estimating} is sub-optimal, and all attempts on showing the same rate of convergence $n^{\frac{-2\alpha+1}{2\alpha}}+ \frac{\log p}{n}$ intrinsically cannot succeed. (ii) From the methodological aspect, due to the regression interpretation of the Cholesky decomposition (\ref{eq: def cd of omega}), almost all existing methods rely on an intermediate estimator of $A$ obtained by running regularized regression of each variable against its previous variables $X_{i}\sim\sum_{j=1 }^{i-1}a_{ij}X_{j}$. For instance, \cite{bickel2008regularized} estimated each row of $A$ by fitting the
banded regression model $X_{i}\sim\sum_{j=\max \{ 1,i-k\} }^{i-1}a_{ij}X_{j}$ with some bandwidth $k$. \cite{wu2003nonparametric}
used an AIC or BIC penalty to pick the best bandwidth $k$. In addition, \cite{huang2006covariance} proposed adding a Lasso or Ridge
penalty while \cite{levina2008sparse} proposed using a nested
Lasso penalty to the regression. See, for instance,  \cite{banerjee2014posterior,lee2017estimating} for Bayesian approaches following the similar idea. The typical analysis for those estimation procedures in a row-wise fashion is to bound the operator norm by its matrix $\ell_1$/$\ell_{\infty}$ norm. Although this analysis may provide optimal rates of convergence under the operator norm loss for some sparsity structure (see, i.e., \cite{cai2012optimal,cai2016estimating2} for sparse covariance and precision matrices estimation), it might be sub-optimal for the bandable structure as seen in bandable covariance matrix estimation \cite{cai2010optimal,bickel2008regularized}. Therefore, in order to obtain rate-optimality over $\paraspp$, a novel analysis or even a new estimation approach is expected.



{\textbf{Main Results.}} With regard to the above two issues, we provide satisfactory solutions in this paper. We at the first time show that the rate of convergence under the operator norm over $\paraspp$ is intrinsically slower than that over the counterpart class of bandable covariance matrices. This is achieved via a novel minimax lower bound construction. Moreover, in order to obtain a rate-optimal estimator, we propose a novel local cropping estimator which does not rely on any estimator of $A$, and thus requires a new analysis. Our local cropping approach targets on accurate estimation of principal submatrices of the precision matrix under the operator norm, which results in a tradeoff between one variance term and two bias terms. The name comes after the idea of estimating each principal submatrix of the precision matrix, which is to crop the center $k$ by $k$ submatrix of the inverse of $3k$ by $3k$ sample covariance matrix using their neighbors in two directions of the same size. (During the finalizing process of this paper, we realized that a similar estimator is independently proposed to estimate precision matrices with a different structure \citep{hu2017minimax}.) Since our procedure does not directly explore the structure on each row of $A$, the analysis of bias terms is much more involved, requiring a block-wise partition strategy. More details are discussed in Sections \ref{sec: est op} and \ref{sec: up op l}. In the end, besides $\paraspp$, a similar type of classes of parameter spaces with bandable Cholesky factor is considered as well,
\begin{equation}\label{eq: def paraspq}
\begin{split}
\paraspq=\Big\{\Omega : \quad &\eigenbd^{-1} \leq \lambda_{\rm{min}}(\Omega) \leq \lambda_{\rm{max}}(\Omega) < \eigenbd, \\
&\abs{a_{ij}}  <M(i-j)^{-\alpha-1}, \quad j \in [i-1]
\Big\}.
\end{split}
\end{equation} 
We further establish another surprising result: the optimal rates of convergence of two spaces, namely $\paraspp$ and $\paraspq$, are different under the operator norm. This remarkable distinction is different from the comparison of two similar types of parameter spaces for bandable covariance matrices in \cite{cai2010optimal} and bandable Toeplitz covariance matrices in \cite{cai2013optimal}. The contrast of minimax results on $\paraspp$ and $\paraspq$ is summarized in Theorem \ref{thm: main theorem op} below. We mainly focus on the high-dimensional setting, assuming that $\log p=O(n)$ and $n=O(p)$. Theorem \ref{thm: main theorem op} implies inconsistency when $\log p=O(n)$ is violated. In addition, one can easily obtain that when $n=O(p)$ is violated, the minimax rate becomes the smaller value between  $p/n$ and the one shown in Theorem \ref{thm: main theorem op} for each space.
\begin{theorem}\label{thm: main theorem op}
	Under normality assumption, the minimax risk of estimating the precision matrix $\Omega$ over the parameter space $\paraspp$ with $\alpha > \hf$ given in \prettyref{eq: def paraspp} under the operator norm satisfies 
	\begin{equation} \label{eq: lop rate}
	\inf_{\tilde{\Omega}}\sup_{\paraspp} \mathbb{E}\opnormsq{\tilde{\Omega}-\Omega} \asymp 
	n^{-\frac{2\alpha-1}{2\alpha}}+ \frac{\log p}{n}  .
	\end{equation}
	The minimax risk of estimating the precision matrix $\Omega$ over the parameter space $\paraspq$ with $\alpha>0$ given in \prettyref{eq: def paraspq} under the operator norm satisfies
	\begin{equation} \label{eq: eop rate}
	\inf_{\tilde{\Omega}}\sup_{\paraspq} \ep\opnormsq{\tilde{\Omega}-\Omega} \asymp  n^{-\frac{2\alpha}{2\alpha+1}}+ \frac{\log p}{n} .
	\end{equation}
\end{theorem}

Moreover, we also consider the minimax rates of convergence of precision matrix estimation under the Frobenius norm loss over $\paraspp$ and $\paraspq$. This time, we prove that two types of spaces enjoy the same optimal rate of convergence. Together with the different rates of convergence under the operator norm loss, we demonstrate the intrinsic difference between operator norm and Frobenius norm. The Frobenius norm of a $p$ by $p$ matrix is defined as the $\ell_2$ vector norm of all entries. Driven by this fact, our estimation approach is naturally obtained by optimally estimating $A$ and $D$ in (\ref{eq: def cd of omega}) separately. Due to the decay structure in $\paraspp$, which is defined in terms of nested $\ell_1$ norm of each row of $A$, our estimator is based on regression with a delicate thresholding rule. The minimax procedure is motivated by wavelet nonparametric function estimation, although the space $\paraspp$ cannot be exactly described by any Besov ball (\cite{cai2012minimax,delyon1996minimax}). We summarize the optimality result under the Frobenius norm in Theorem \ref{thm: main theorem f} below.

\begin{theorem}\label{thm: main theorem f}
	Under normality assumption, the minimax risk of estimating the precision matrix $\Omega$ over $\paraspp$ and $\paraspq$ given in \prettyref{eq: def paraspp} and \prettyref{eq: def paraspq} satisfies
	\begin{equation} \label{eq: lef rate}
	\inf_{\tilde{\Omega}}\sup_{\paraspp} \frac{1}{p} \ep\fnormsq{\tilde{\Omega}-\Omega} \asymp \inf_{\tilde{\Omega}}\sup_{\paraspq} \frac{1}{p} \ep\fnormsq{\tilde{\Omega}-\Omega} \asymp \efrateup .
	\end{equation}
\end{theorem}

{\textbf{Related Literature.}} During the last decade, various structural assumptions are imposed in literature of high-dimensional statistics in order to estimate the covariance/precision matrix consistently under various loss functions. While mostly driven by the specific scientific applications, popular structures include \textit{ordered sparsity} (bandable covariance matrices, precision matrices with bandable Cholesky factor), \textit{unordered sparsity} (sparse covariance matrices, sparse precision matrices) and other more complicated ones such as certain combination of sparsity and low-rankness (spike covariance matrices, covariance with tensor product, latent graphical models). Many estimation procedures have been proposed accordingly to estimate high-dimensional covariance/precision matrices via taking advantages of these specific structures. For example, banding (\cite{wu2009banding,bickel2008regularized,xiao2014theoretic,bien2016convex}) and tapering methods (\cite{furrer2007estimation,cai2010optimal}) were developed to estimate bandable covariance matrices or precision matrices with bandable Cholesky factor; thresholding procedures were used in \cite{bickel2008covariance,karoui2008operator,cai2011adaptive} to estimate sparse covariance matrices; penalized likelihood estimation (\cite{huang2006covariance,yuan2007model,d2008first,banerjee2008model,rothman2008sparse,lam2009sparsistency,ravikumar2011high}) and penalized regression methods (\cite{meinshausen2006high,yuan2010high,cai2011constrained,sun2013sparse,ren2015asymptotic}) are designed for sparse precision matrix estimation. 

The fundamental difficulty of various covariance/precision matrices estimation problems have been carefully investigated in terms of the minimax risks under the operator norm loss among other losses, especially for those \textit{ordered and unordered sparsity structures}. Specifically, for unordered structures, \cite{cai2012optimal} considered the problems of optimal estimation of sparse covariance while \cite{cai2016estimating2} (see \cite{ren2015asymptotic} as well) established the optimality results for estimating sparse precision matrices. For ordered structures, \cite{cai2010optimal} established the optimal rates of convergence over two types of bandable covariance matrices. In addition, with an extra Toeplitz structure, \cite{cai2013optimal} studied optimal estimation of two types of  bandable Toeplitz covariance matrices. However, it was still largely unknown about the optimality results on estimating precision matrices with bandable Cholesky factor. See an exposure paper with discussion \cite{cai2016estimating} and references therein on minimax results of covariance/precision matrix estimation under some other losses. In this paper, we provide a solution to this open problem by establishing the optimal rates of convergence over two types of precision matrices with bandable Cholesky factor. \textit{Thus, this paper completes the minimaxity results of all four sparsity structures commonly considered in literature}. 

{\textbf{Organization of the Paper.}} The rest of the paper is organized as follows. First, we propose our estimation procedures for precision matrix estimation in Section \ref{sec: methodology}. In particular, local cropping estimators and regression-based thresholding estimators are designed for estimating precision matrices under the operator norm and the Frobenius norm respectively. Section \ref{sec: risk op} establishes the optimal rates of convergence under the operator norm for two commonly used types of parameter spaces $\paraspp$ and $\paraspq$. A striking difference between two spaces are revealed when considering operator norm loss. 
Section \ref{sec: risk f} considers rate-optimal estimation under the Frobenius norm. 
The results reveal that the fundamental difficulty of estimation for two parameter spaces are the same when considering Frobenius norm loss. Section \ref{sec: Adaptivity} considers the adaptive estimation through a variation of Lepski's method under the operator norm. In Section \ref{sec: rank}, we extend the results to nonparanormal models for inverse correlation matrix estimation by applying local cropping procedure to rank-based estimators. Section \ref{sec: simulation} presents the numerical performance of our local cropping procedure to illustrate the difference between two parameter spaces by simulation studies. We also demonstrate the sub-optimality of banding estimators. 
Discussion and all technical lemmas used in proofs of main results are relegated to the supplement.

\textbf{Notation.} we introduce some basic notations that will be used in the rest of the paper. $\indc\cdot$ indicates the indicator function while $\mathbf{1}$ indicates the all-ones vector. $\sgn(\cdot)$ indicates the sign function. $\floor{s}$ represents the largest integer which is no more than $s$. $\ceil{s}$ represents the smallest integer which is no less than $s$. Define $a_n \asymp b_n$ if there is a constant $C > 0$ independent of $n$ such that $C^{-1} \leq a_n/b_n \leq C$. For any vector $x$, $\norm{x}_p$ indicates its $\ell_p$ norm. For any $p$ by $q$ matrix $S = [s_{ij}]_{p \times q} \in \RR^{p \times q}$, we use $S^T$ to denote its transpose. The $\ell_p$ matrix norm is define as $\norm{S}_p=\sup_{\norm{x}_p=1}\norm{Sx}_p$. The $\ell_2$ matrix norm is also called the the operator norm or the spectral norm, and denoted as $\opnorm{S}$. The Frobenius norm is defined as $\fnorm{S}=(\sum_{i,j} s_{ij}^2)^\hf$. $\lambda_{\max}(S)$ and $\lambda_{\min}(S)$ are the largest and smallest singular values of $S$ when $S$ is not symmetric. When $S$ is a real symmetric matrix, $\lambda_{\max}(S)$ and $\lambda_{\min}(S)$ denote its largest and smallest eigenvalues.  $\textrm{row}_i(S)$ and $\textrm{col}_i(S)$ indicate the $i$-th row and column of matrix $S$. $a:b$ denotes the index set $\{a, a+1, \dots, b\}$. $[p]$ is short for the set $1:p$. For the random vector $\mx \in \RR^{p\times 1}$ and the data matrix $\mz \in \RR^{n \times p}$, $\mx_{a:b}$ and $\mz_{a:b}$ indicates the $(a:b)$-th columns of $\mx^T$ and $\mz$. For any square matrix $S$, $\diag(S)$ denotes the diagonal matrix with diagonal entries being those on the main diagonal of $S$ while for any vector $v$, $\diag(v)$ denotes the diagonal matrix with diagonal entries being $v$. In the estimation procedure under the operator norm, we use the matrix notation in the form of $S_m^{(k)}$ to facilitate the proof, where $S$ is always a square matrix, $m$ indicates the location information, and $(k)$ indicates that the size of $S_m^{(k)}$ is $k$. Throughout the paper we denote by $C$ a generic positive constant which may vary from place to place but only depends on $\alpha$, $\eigenbd$, $M$ and possibly some sub-Gaussian distribution constant $\rho$ in (\ref{eq: def sub gaussian}).

\section{Methodologies}\label{sec: methodology}

In this section, we introduce our methodologies for estimating precision matrices over $\paraspp$ and $\paraspq$ under both the operator norm and the Frobenius norm. Assume that $\mx=(X_1, \dots X_p)^T$, a $p$-variate random vector with mean zero and precision matrix $\Omega_p$.
Our estimation procedures are based on its $n$ i.i.d. copies $\mz \in \RR^{n \times p}$. We write $ \mz = (\mz_1, \dots, \mz_p) $, where each $\mz_i$ consists of $n$ i.i.d. copies of $X_i$. Our estimation procedures are different under the operator norm and the Frobenius norm. 		

\subsection{Estimation Procedure under the Operator Norm}\label{sec: est op}

We focus on the estimation problem under the operator norm first. As we discussed in the introduction, almost all existing methodologies (\cite{wu2003nonparametric,huang2006covariance,bickel2008regularized}) directly appeal the Cholesky decomposition of the precision matrix. They first estimate the Cholesky factor $A$ and $D$ by auto-regression and then estimate the precision matrix according to $\Omega = (I-A)^TD^{-1}(I-A)$. The corresponding analysis in the row-wise fashion may not suitable for the operator norm loss. In this paper, we propose a novel local cropping estimator, which focuses on the estimation of $\Omega$ directly.

To facilitate the illustration of the estimation procedure, we define two matrix operators. The \textbf{cropping operator} is designed to crop the center block out of the matrix. For a $p$ by $p$ matrix $E \equiv [e_{ij}]_{p \times p}$, we define the $k \times k$ matrix $\cut{k}{m}(E) \equiv [c_{ij}]_{k \times k}$, where $1 \leq m \leq p-k+1$, with
\begin{equation}\label{eq: def operator cut}
c_{ij} = e_{i+m-1,j+m-1},\text{ when } 1\leq i,j \leq k.
\end{equation}
The parameter $m$ indicates the location and $k$ indicates the dimension. It is clear that $\cut{k}{m}(E)$ is a principal submatrix of $E$. The \textbf{expanding operator} is designed to put a small matrix onto a large zero matrix.
For a $k$ by $k$ matrix, $C \equiv [c_{ij}]_{k \times k}$, define the $p \times p$ matrix $\extend{p}{m}(C) \equiv [e_{ij}]_{p \times p}$, where $1 \leq m \leq p-k+1$, with
\begin{equation}\label{eq: def operator extend}
e_{ij} = c_{i-m+1,j-m+1},\text{ when } m\leq i,j \leq m+k-1, \text{ otherwise } e_{ij} =0.
\end{equation}
The parameter $m$ indicates the location and $p$ indicates the dimension. Note that for a $k$ by $k$ matrix $C$, we have $\cut{k}{m}(\extend{p}{m}(C))=C$. An illustration of two operators is provided in Figure \ref{fig:operators}.

\begin{figure}
	
	\centering
	\begin{tikzpicture}[scale=4.5]
	\draw (-0.5, 0.5) -- (0.5, -0.5);
	\draw[ultra thick, black] (-0.5, -0.5) rectangle (0.5, 0.5);
	\fill[ultra thick, lightgray] (-0.2, -0.2) rectangle (0.2, 0.2);
	\draw (-0.2, 0.5) -- (-0.2, 0.2);
	\node at (-0.2, 0.55) {m};
	\node at (0, 0.25) {k};
	\node at (0, -0.45) {p};
	\draw[ultra thick, ->] (0.6,0.05) -- (1,0.05);
	\draw[ultra thick, <-] (0.6,-0.05) -- (1,-0.05);
	\fill[ultra thick, lightgray] (1.15, -0.2) rectangle (1.55, 0.2);
	\node at (1.35, 0.25) {k};
	\node at (0.8, 0.15) {$\cut{k}{m}$};
	\node at (0.8, -0.15) {$\extend{p}{m}$};
	\end{tikzpicture}
	\caption{An illustration of the cropping operator and the expanding operator.}
	\label{fig:operators}
\end{figure}
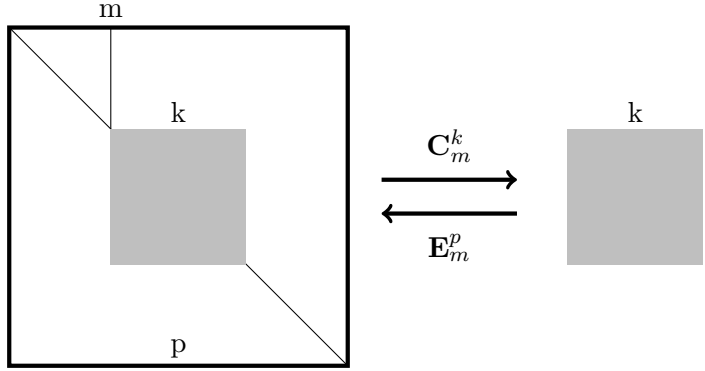

In addition, for technical reasons (of obtaining rates of convergence in expectation rather than in probability), we introduce a \textbf{projection operator}. For a real square matrix $S$, let the singular value decomposition of $S$ be $S = U\Lambda V^T$ with $UU^T=I$, $VV^T=I$ and $\Lambda = \diag(\lambda_i)$. Let $\Lambda^* = \diag(\lambda_i^*)$, where $\lambda_i^* = \min\{\max\{\lambda_i, \eigenbd^{-1}\}, \eigenbd\}$, then define
\begin{equation}\label{eq: projection}
\mathbf{P}_\eigenbd(S) = U\Lambda^* V^T.
\end{equation}
For a symmetric matrix $S$, we modify $\mathbf{P}_\eigenbd(\cdot)$ a little bit and define $\mathbf{P}_\eigenbd(S) = U\Lambda^* U^T$, where $S = U\Lambda U^T$ is its eigen-decomposition. Since all eigenvalues of $\mathbf{P}_\eigenbd(\cdot)$ are in the interval $[\eigenbd^{-1}, \eigenbd]$, $\mathbf{P}_\eigenbd(S)$ is always invertible and positive definite.

We are ready to construct the local cropping estimator $\tilde{\Omega}_k^{\op}$ with bandwidth $k<p$.  At a high level, we first propose an estimator of each principal submatrix of size $k$ and $2k$ in $\Omega$ using cropping and expanding operators. Then we arrange over those local estimators to estimate $\Omega$. Since the core idea of estimating those local estimators in our procedure is to crop the inverse of sample covariance matrix with a relatively larger size, we call $\tilde{\Omega}_k^{\op}$ in (\ref{eq: def est op}) \textbf{the local cropping estimator}.

Specifically, we first define an estimator $\tilde{\Omega}_{m}^{(k)}$ of the principal submatrix $\cut{k}{m}(\Omega)$ at each location $m$. To this end, we select the sample covariance matrix with a relative larger size, in this case, $3k$. Let the modified local sample covariance matrix be
\begin{equation}\label{eq: 1 leop est}
\tilde{\Sigma}_{m-k}^{(3k)} = \mathbf{P}_\eigenbd\big(\cut{{3k}}{m-k} (\frac{1}{n}\mz^T\mz)\big),\qquad 2-k\leq m\leq p. 
\end{equation}
We refer to Remark \ref{rem:boundary} for the treatment when the index is beyond the index set $[p]$. Note that the operator $\projj(\cdot)$ guarantees $\tilde{\Sigma}_{m-k}^{(3k)}$ to be invertible. Then we use the center part of its inverse  to estimate  $\mathbf{C}_{m}^{k}(\Omega)$, i.e.,
\begin{equation}\label{eq: local estimator}
\tilde{\Omega}_{m}^{(k)} = \cut{k}{k+1}\big((\tilde{\Sigma}_{m-k}^{(3k)})^{-1}\big).
\end{equation}
Similarly, we can define local estimators of $\tilde{\Omega}_{m}^{(2k)}$ via replacing $k$ by $2k$. Arranging over these estimators in the form of weighted sum, we obtain the estimator of $\Omega$, that is,
\begin{equation} \label{eq: def est op}
\tilde{\Omega}_{k}^{\op} = \mathbf{P}_\eigenbd\Big( \frac{1}{k}\big(\sum_{m=2-2k}^p \extend{p}{m}(\tilde{\Omega}_{m}^{(2k)}) - \sum_{m=2-k}^p \extend{p}{m}(\tilde{\Omega}_{m}^{(k)})\big)\Big).
\end{equation}
The operator $\extend{p}{m}(\cdot)$ makes these local estimators in the correct places. The final step \prettyref{eq: def est op} is motivated by the analysis of optimal bandable covariance matrix estimation procedure proposed in \cite{cai2010optimal}. Indeed, the optimal tapering estimator in \cite{cai2010optimal} can be rewritten as a sum of many principal submatrices of the sample covariance matrix in a similar way as \prettyref{eq: def est op}. In contrast, our estimator is not in a form of tapering the sample covariance matrix. However, in the analysis of our local cropping estimator in Section \ref{sec: risk op}, the direct target of $\tilde{\Omega}_{k}^{\op}$ is a certain tapered population precision matrix with bandwidth $k$. There are natural bias and variance terms involved in the distance of $\tilde{\Omega}_{k}^{\op}$ and its direct target. Together with the bias of the tapered population precision matrix, our analysis involves two bias terms and one variance term, which critically determine the optimal choice of bandwidth. It is worthwhile to mention that although the local estimators in (\ref{eq: def est op}) overlap with each other significantly, the variance term is not influenced too much by the overlap due to a technique of rearranging all local estimators in the analysis. Please refer to the proof of Theorem \ref{thm: up op 1} for further details.

In Section \ref{sec: risk op}, we show that the local cropping estimator with an optimal choice of bandwidth would achieve the minimax risk under the operator norm over parameter spaces $\paraspp$ in \prettyref{eq: def paraspp} and $\paraspq$ in \prettyref{eq: def paraspq}. However, the optimal choices of bandwidth are fundamentally distinct between $\paraspp$ and $\paraspq$. Specifically, we show that the optimal bandwidth over $\paraspp$ is $k \asymp  \lopband$ while that one over $\paraspq$ is $k \asymp  n^{\frac{1}{2\alpha+1}}$. 

\begin{remark}\label{rem:boundary}
	Of note, the estimator $\tilde{\Omega}_{k}^{\op}$ depends on $\mz_{2-4k}, \dots, \mz_{p+4k-1}$. The index of variable is clear most of the time, while we need to be careful when it is close to the boundary. When the index is beyond the index set $[p]$, we shrink the size of the corresponding block by discarding the data with meaningless indexes. It can be shown that this shrinking operation would not change the theoretical properties of the final estimator.
\end{remark}

\subsection{Estimation Procedure under the Frobenius Norm}\label{sec: est f}
Under the Frobenius norm, our estimation procedure is based on the Cholesky decomposition of the precision matrix \prettyref{eq: def cd of omega}. More specifically, we estimate the matrix $A$ and $D$ respectively by auto-regression, and then combine them together to construct the estimator of $\Omega$. The following estimation procedure applies to both the parameter space $\paraspp$ and $\paraspq$ as we will show that they enjoy the same optimal rate of convergence in Section \ref{sec: risk f}.  

Our estimator of the $i$-th row of $A$ is based on the regression of $X_i$ against its previous variables. Unlike those existing methods (\cite{wu2003nonparametric,huang2006covariance,bickel2008regularized}) which rely on certain banding or penalized approaches for such a regression problem, we apply a thresholding procedure due to the decay structure in $\paraspp$ which is defined in terms of nested $\ell_1$ norm. To this end, we first regress $X_i$ against $\mx_{i-k_1:i-1}=(X_{i-k_1}, \dots, X_{i-1})^T$ with bandwidth $k_1=\ceil{\frac{n}{c}}$ with some sufficiently large $c>0$. Recall that the $n\times 1$ matrix $\mz_i$ consists of $n$ observations of $X_i$, and the $n \times k_1$ matrix $\mz_{i-k_1:i-1}$ represents $n$ observations of $\mx_{i-k_1:i-1}$.
The empirical regression coefficients are
\begin{equation}\label{eq: def thresholding coeff}
(\hat{a}_{i(i-k_1)},\dots \hat{a}_{i(i-1)})^T=(\mz_{i-k_1:i-1}^T\mz_{i-k_1:i-1})^{-1}\mz_{i-k_1:i-1}^T\mz_i.
\end{equation}
We then further threshold the coefficients by taking advantages of the bandable structure of the Cholesky factor $A$. Specifically, we define $\hat{\textbf{a}}_i^* \in \RR^{i-1}$ with coordinate $\hat{a}_{ij}^{*}$ as follows,
\begin{equation}\label{eq: def thresholding}
\hat{a}_{ij}^{*} = 
\begin{cases}
\hat{a}_{ij}, & \text{if }i-k_0 < j \leq i-1,\\
\hat{a}_{ij} \indc{|\hat{a}_{ij}| > \lambda_j}, & \text{if } i-k_1< j \leq i-k_0,\\
0, & \text{if } 1\leq j \leq i-k_1,
\end{cases}
\end{equation}
where $k_0=\ceil{n^{\frac{1}{2\alpha + 2}}}$, $R= 8\eigenbd\opnorm{(\mz_{i-k_1:i-1}^T\mz_{i-k_1:i-1})^{-1}}$, and the threshold level $\lambda_{j} = (\ceil{\log_2^{i-j} - \log_2^{k_0}} R)^{\hf}$. Note that we keep the last $k_0$ coefficients and apply an entry-wise thresholding rule for which the thresholding level remains the same in each block and the size of each block doubles backwards sequentially for the remaining coefficients in (\ref{eq: def thresholding}). Our procedure is inspired by the optimal estimation procedure over Besov balls for many nonparametric function estimation problems, or equivalently, the corresponding Gaussian sequence models (See \cite{cai2012minimax} the reference therein). We emphasize that any linear estimator of the coefficients $(\hat{a}_{i(i-k_1)},\dots \hat{a}_{i(i-1)})^T$ cannot yield to the optimal rates of convergence in our setting under the Frobenius norm.

Our estimator of $I-A$ can be constructed by arranging zero-padded $\hat{\mathbf{a}}_i^{*T}$, $i \in [p]$ accordingly with an identity matrix. Specifically, set the $ij$-th entry of $\hat{A}^*$ as $\hat{a}_{ij}^{*}$ when $ i \in [p]$, $j \in [i-1]$, otherwise as zero. We also need to bound the singular values of $(I-\hat{A}^*)$. To this end, we define
\begin{equation*} 
\widetilde{I - A}=\projj(I-\hat{A}^*),
\end{equation*}
as our estimator of $(I-A)$, where $\projj(\cdot)$ is defined in \prettyref{eq: projection}.

The estimation of $D$ is based on the sample variances of those empirical residuals in the regression of $X_i$ against $\mx_{i-k_1:i-1}=(X_{i-k_1}, \dots, X_{i-1})^T$. For each $i$, the sample variance of the empirical residual is 
\begin{equation}\label{eq: def e}
\hat{d}_i=\frac{1}{n-k_1}\mz_i^T(I-\mathbf{M}_i)^T(I-\mathbf{M}_i)\mz_i,
\end{equation}
where $\mathbf{M}_i=\mz_{i-k_1:i-1}(\mz_{i-k_1:i-1}^T\mz_{i-k_1:i-1})^{-1}\mz_{i-k_1:i-1}^T $.
Let $\hat{D}=\diag(\hat{d})$, where $\hat{d}=(\hat{d}_1,\dots ,\hat{d}_p)^T$. We define $\tilde{D} = \projj(\hat{D})$ as our estimator of $D$.

Finally, define our estimator of $\Omega$ as  
\begin{equation} \label{eq: lf 1 est}
\tilde{\Omega}_k^{\f} = (\widetilde{I - A})^T\tilde{D}^{-1}(\widetilde{I - A}).
\end{equation}

\begin{remark}
	For the parameter space $\paraspq$, a much simpler banding estimation scheme on the empirical regression coefficients is able to achieve the minimax risk. Set $k=\ceil{n^{\frac{1}{2\alpha+2}}}$. We use the empirical residuals and coefficients obtained by regressing each $X_i$ against $\mx_{i-k:i-1}$  to directly construct the estimators of $A$ and $D$. It can be proved that this estimator achieves the minimax risk over the parameter space $\paraspq$.
\end{remark} 

\section{Rate Optimality under the Operator Norm}\label{sec: risk op}
In this section, we establish the optimal rates of convergence for estimating the precision matrix over the parameter spaces $\paraspp$ and $\paraspq$ given in \prettyref{eq: def paraspp} and \prettyref{eq: def paraspq} under the operator norm. We first derive the risk upper bound of the local cropping estimator in \prettyref{sec: up op l} over parameter space $\paraspp$. We provide a matching risk lower bound by applying the Assouad's lemma and the Le Cam's method in \prettyref{sec: low op l}  over $\paraspp$. The establishment of the rate optimality over the parameter space $\paraspq$ is similar to the one over $\paraspp$, which is provided in \prettyref{sec: risk op e}. 

Throughout this section, we assume that $\mx=(X_1, \dots X_p)^T$ follows certain sub-Gaussian distribution with constant $\rho > 0$, that is,
\begin{equation} \label{eq: def sub gaussian}
\pb \{ |v^T(\mx-\ep \mx)| > t\} \leq 2 \exp(-t^2/(2\rho)),
\end{equation}
for all $t > 0 $ and all unit vectors $\| v\|_2=1$.
\begin{remark}
	The sub-Gaussian distribution is often assumed in high-dimensional statistical problems. In our settings, this assumption is critical to derive the exponential-type concentration inequality for the quadratic terms of $\mx$. When only certain moment conditions are posed, one can replace each local estimator in  (\ref{eq: local estimator}) by an Huber-type estimator proposed in \cite{minsker2018sub,ke2018user}. We leave the theoretical investigation in future works.   
\end{remark} 

\subsection{Minimax Upper Bound under the Operator Norm over $\paraspp$}\label{sec: up op l}
In this section, we develop the following upper bound of our estimation procedure proposed in \prettyref{sec: est op}. 
\begin{theorem} \label{thm: up op 1}
	When $\ceil{\lopband} \leq p$,  the local cropping estimator defined in \prettyref{eq: def est op} of the precision matrix $\Omega$ over $\paraspp$ with $\alpha>\frac{1}{2}$ given in \prettyref{eq: def paraspp} satisfies 
	\begin{equation*} \label{eq: up op 1}
	\sup_{\paraspp} \ep\opnormsq{\tilde{\Omega}_k^{\rm{op}}-\Omega} \leq Ck^{-2\alpha+1}+C\frac{\log p + k}{n}. 
	\end{equation*}
	When $k=\ceil{\lopband}$, we have
	\begin{equation*} \label{eq: up op l k}
	\sup_{\paraspp} \ep\opnormsq{\tilde{\Omega}_k^{\rm{op}}-\Omega} \leq Cn^{-\frac{2\alpha-1}{2\alpha}}+ C\frac{\log p}{n} .
	\end{equation*}
\end{theorem}
The optimal choice of $k \asymp n^{\frac{1}{2\alpha}}$ is due to the bias-variance tradeoff. Combining Theorem \ref{thm: up op 1} with the minimax lower bound derived in \prettyref{sec: low op l}, we immediately obtain that the local cropping estimator is rate optimal.

\begin{proof}
	As we discussed in \prettyref{sec: est op}, the direct target of our local cropping estimator is certain tapered population precision matrix with bandwidth $k$, which can be written as a weighted sum of many principal submatrices of the population precision matrix. We construct this corresponding tapered population precision matrix $\Omega^*_k$ as follows. Denote the precision matrix $\Omega \equiv [\omega_{ij}]_{p \times p}$. We define $\Omega^*_k \equiv [\omega^*_{ij}]_{p \times p}$ such that for $i,j \in [p]$, 
	\begin{equation} \label{eq: tp def}
	\omega^*_{ij} = m_{ij} \omega_{ij}, \text{ where }m_{ij}=\max\{0,2-\frac{1}{k} |i-j|\}- \max\{0,1-\frac{1}{k} |i-j|\}.
	\end{equation}
	The following lemma elucidates the decomposition of this tapered precision matrix $\Omega^*_k$. 
	\begin{lemma} \label{lmm: tp omega decomp}
		The $\Omega_k^*$ defined in \prettyref{eq: tp def} can be written as
		\begin{equation*} 
		\Omega^*_k = \frac{1}{k}\Big(\sum_{m=2}^{2k+1} \big(\sum^{\floor{p/2k}}_{j=-1} \extend{p}{m+2kj}\big(\cut{2k}{m+2kj}(\Omega)\big)\big)  -\sum_{m=2}^{k+1} \big(\sum^{\floor{p/k}}_{j=-1} \extend{p}{m+kj}\big(\cut{k}{m+2kj}(\Omega)\big) \big) \Big).
		\end{equation*}
	\end{lemma}
	The proof of Lemma \ref{lmm: tp omega decomp} can be found in \citep{cai2010optimal} (refer to the proof of Lemma 1 with covariance matrix therein replaced by the precision matrix), and thus omitted.
	Define 
	\begin{equation*}
	\tilde{\Omega}^*_k = \frac{1}{k}\Big(\sum_{m=2}^{2k+1} \big(\sum^{\floor{p/2k}}_{j=-1} \extend{p}{m+2kj}(\tilde{\Omega}_{m+2kj}^{(2k)})\big)  -\sum_{m=2}^{k+1} \big(\sum^{\floor{p/k}}_{j=-1} \extend{p}{m+kj}(\tilde{\Omega}_{m+kj}^{(k)})\big)  \Big).
	\end{equation*}
	It is easy to check $\tilde\Omega_k^{\op} = \projj(\tilde{\Omega}^*_k)$. Since the eigenvalues of $\Omega$ are in the interval $[\eigenbd^{-1}, \eigenbd]$, the operator $\projj(\cdot)$ would not increase the risk much. Indeed, according to \prettyref{eq: proj 1} in \prettyref{lmm: projection}, we have 
	\begin{align}
	\ep\opnormsq{\tilde{\Omega}_k^{\op}-\Omega} &\leq 4\ep\opnormsq{\tilde{\Omega}^*_k-\Omega} \nonumber \\
	&\leq 8 \ep\opnormsq{\tilde{\Omega}^*_k-\Omega^*_k} + 8\opnormsq{\Omega^*_k - \Omega}. \label{eq: up lop 0}
	\end{align}
	The following lemma bounds the bias between our direct target $\Omega^*_k$ and the population precision matrix.
	\begin{lemma} \label{lmm: tp omega close}
		For $\Omega$ in the parameter space $\paraspp$ defined in \prettyref{eq: def paraspp} with $\alpha > \hf$, $\Omega^*_k$ defined in \prettyref{eq: tp def}, we have 
		\begin{equation*}
		\opnormsq{\Omega^*_k - \Omega} \leq Ck^{-2\alpha + 1}.
		\end{equation*}
	\end{lemma}
	\begin{remark}
		Unlike existing methods, our procedure does not directly utilize the decay structure of Cholesky factor. Consequently, the proof of Lemma \ref{lmm: tp omega close} is involved and requires a block-wise partition strategy.
	\end{remark}
	Then we turn to the analysis of $\ep\opnormsq{\tilde{\Omega}^*_k-\Omega^*_k}$. 
	\begin{equation} \label{eq: up lop 1}
	\begin{split}
	&\ep \opnormsq{\tilde{\Omega}^*_k - \Omega^*_k} \\
	\leq
	& 2\ep \big(\frac{1}{k}
	\sum_{m=2}^{2k+1} 
	\opnorm{\sum^{\floor{p/2k}}_{j=-1} \extend{p}{m+2kj}(\tilde{ \Omega}_{m+2kj}^{(2k)}) 
		- \sum^{\floor{p/2k}}_{j=-1} \extend{p}{m+2kj}\big(\cut{2k}{m+2kj}(\Omega)\big)}
	\big)^2  \\
	&+  2\ep \big(\frac{1}{k}
	\sum_{m=2}^{k+1} 
	\opnorm{\sum^{\floor{p/k}}_{j=-1} \extend{p}{m+kj}(\tilde{ \Omega}_{m+kj}^{(k)}) 
		- \sum^{\floor{p/k}}_{j=-1} \extend{p}{m+kj}\big(\cut{k}{m+kj}(\Omega)\big)}
	\big)^2 .
	\end{split}
	\end{equation}
	These two terms can be bounded in the same way, we only focus on the second term. 
	\begin{align}  
	&\ep \big(\frac{1}{k}
	\sum_{m=2}^{k+1} 
	\opnorm{\sum^{\floor{p/k}}_{j=-1} \extend{p}{m+kj}(\tilde{ \Omega}_{m+kj}^{(k)}) 
		- \sum^{\floor{p/k}}_{j=-1} \extend{p}{m+kj}\big(\cut{k}{m+kj}(\Omega)\big)}
	\big)^2  \nonumber \\
	\leq & \ep \big(\max_m \opnorm{\sum^{\floor{p/k}}_{j=-1} (\extend{p}{m+kj}(\tilde{ \Omega}_{m+kj}^{(k)})  - \extend{p}{m+kj}\big(\cut{k}{m+kj}(\Omega)\big)\big)} ^2 \big) \nonumber  \\
	\leq & \ep \big(\max_{m,j} \opnorm{\tilde{ \Omega}_{m+kj}^{(k)} - \cut{k}{m+kj}(\Omega) }^2 \big) \nonumber  \\
	\begin{split} 
	\leq & 2\ep \big(\max_{m \in [p]} \opnorm{\tilde{ \Omega}_{m}^{(k)} -
		\cut{k}{k+1}\big( (\cut{{3k}}{m-k}( \Omega^{-1} ))^{-1} \big)}^2\big) \\
	&+ 2\big(\max_{m \in [p]} \opnorm{\cut{k}{k+1}\big( (\cut{{3k}}{m-k}( \Omega^{-1} ))^{-1} \big)-\cut{k}{m}(\Omega)}^2\big) ,
	\end{split} \label{eq: up lop 2}
	\end{align}
	where we further have variance term and bias term of local estimators.
	For the variance term in \prettyref{eq: up lop 2}, we further have
	\begin{align}
	&\opnorm{\tilde{ \Omega}_{m}^{(k)} - \cut{k}{k+1}\big( (\cut{{3k}}{m-k}( \Omega^{-1} ))^{-1} \big)} \nonumber\\
	\leq & \opnorm{ (\tilde{ \Sigma}_{m-k}^{(3k)} )^{-1}- (\cut{{3k}}{m-k}( \Omega^{-1} ))^{-1} } \nonumber\\
	\leq & \eigenbd^2 \opnorm{\tilde{ \Sigma}_{m-k}^{(3k)} - \cut{{3k}}{m-k}( \Omega^{-1} ) }\nonumber\\
	\leq & 2\eigenbd^2 \opnorm{\cut{{3k}}{m-k}( \frac{1}{n}\mz\mz^T ) - \cut{{3k}}{m-k}( \Omega^{-1} ) }.\label{eq: up lop 3}
	\end{align}
	The last two inequalities hold because of the fact that the eigenvalues of $\tilde{ \Sigma}_{m-k}^{(3k)}$ and $\cut{{3k}}{m-k}( \Omega^{-1} )$ are in the interval $[\eigenbd^{-1}, \eigenbd]$, and \prettyref{lmm: projection}. The following concentration inequality of sample covariance matrix facilitates our proof.
	\begin{lemma} \label{lmm: sample cov max}
		For the observations $\mz$ following certain sub-Gaussian distribution with constant $\rho$ and precision matrix $\Omega$, we have 
		\begin{equation*}
		\ep \big( \max_{m \in [p]} \opnormsq{\cut{{3k}}{m-k}( \frac{1}{n}\mz\mz^T ) - \cut{{3k}}{m-k}( \Omega^{-1} ) } \big) 
		\leq C\frac{\log p + k}{n}.
		\end{equation*}
	\end{lemma}
	
	Lemma \ref{lmm: sample cov max} is an extension of the result in Chapter 2 of \citep{saulis2012limit}. Its proof can be found in Lemma 3 of \citep{cai2010optimal}.
	
	Combining \prettyref{lmm: sample cov max}, \prettyref{eq: up lop 2} and \prettyref{eq: up lop 3}, we have
	\begin{equation}\label{eq: up lop 4}
	\ep \big(\max_{m \in [p]} 
	\opnorm{\tilde{ \Omega}_{m}^{(k)} - \cut{k}{k+1}\big( (\cut{{3k}}{m-k}( \Omega^{-1} ))^{-1} \big)}^2\big)
	\leq C\frac{\log p + k}{n}.
	\end{equation}
	We turn to bounding the bias term of local estimator in (\ref{eq: up lop 2}).
	\begin{lemma} \label{lmm: bias in block up lop}
		Assume that $\Omega \in \paraspp$ defined in \prettyref{eq: def paraspp} with $\alpha > \hf$. Then we have
		\begin{equation*}
		\opnormsq{\cut{{k}}{k+1}\big( (\cut{{3k}}{m-k}( \Omega^{-1} ))^{-1} \big) - \cut{k}{m}(\Omega) }\leq Ck^{-2\alpha +1}.
		\end{equation*}
	\end{lemma}
	\prettyref{lmm: bias in block up lop}, together with \prettyref{eq: up lop 4}, \prettyref{eq: up lop 2} and \prettyref{eq: up lop 1}, implies that
	\begin{equation}\label{eq: up lop 6}
	\ep\opnormsq{\tilde{\Omega}^*_k-\Omega^*_k} \leq C\frac{\log p + k}{n} + Ck^{-2\alpha +1}.
	\end{equation}
	Plugging \prettyref{lmm: tp omega close} and \prettyref{eq: up lop 6} into \prettyref{eq: up lop 0}, we finish the proof of \prettyref{thm: up op 1}.
\end{proof}

\subsection{Minimax Lower Bound under the Operator Norm over $\paraspp$}\label{sec: low op l}
\prettyref{thm: up op 1} in \prettyref{sec: up op l} proves that the local cropping estimator defined in \prettyref{eq: def est op} attains the convergence rate of $n^{\frac{-2\alpha+1}{2\alpha}}+ \frac{\log p}{n}$. In this section, we establish the following matching lower bound, which proves the rate optimality of the local cropping estimator. 

\begin{theorem} \label{thm: 1 lower bound in op}
	The minimax risk of estimating the precision matrix $\Omega$ over $\paraspp$ defined in \prettyref{eq: def paraspp} under the operator norm with $\alpha\geq \frac{1}{2}$ satisfies
	\begin{equation}
	\inf_{\tilde{\Omega}}\sup_{\paraspp}\mathbb{E}\opnormsq{\tilde{\Omega}-\Omega} \geq
	{\frac{\tau^2}{32} \Big( n^{-\frac{2\alpha-1}{2\alpha}}+ \frac{\log p}{n} \Big)},
	\end{equation}
	where $0< \tau < \taubd$.
\end{theorem}
\begin{remark}
	Theorems \ref{thm: up op 1} and \ref{thm: 1 lower bound in op} together show that the minimax risk for estimating the precision matrices over $\paraspp$ stated in (\ref{eq: lop rate}) of Theorem \ref{thm: main theorem op}. It is worthwhile to notice that there is no consistent estimator over $\paraspp$ under the operator norm, when $\alpha \leq \hf$.
\end{remark}

\begin{proof}
	The lower bound of parameter space $\paraspp$ can be established by the lower bounds over its subsets. We construct two subsets $\calp_{1}$ and $\calp_{2}$ and calculate the lower bound over those two subsets separately. Let $\tau$ be a positive constant which is less than $\taubd$.
	
	First, we construct $\calp_{1}$. Set $k=\min\{\ceil{n^{\frac{1}{2\alpha}}},\frac{p}{2}\}$. Set the index set $\Theta=\{0,1\}^k$, i.e., for any $\theta\equiv \{\theta_i\}_{1\leq i\leq k} \in \Theta$, each $\theta_i$ is either $0$ or $1$. Then we define the $k \times k$ matrix $A_k^*(\theta) \equiv [a_{ij}]_{k \times k}$ with $a_{ij}={\lopelem\theta_i\indc{j=k}}$ and 
	\begin{equation*}
	A(\theta)=\begin{bmatrix}
	0_{k \times k} & 0_{k \times k} & 0_{k \times (p-2k)}\\
	A^*_k(\theta) & 0_{k \times k} & 0_{k \times (p-2k)}\\
	0_{(p-2k) \times k} & 0_{(p-2k) \times k} & 0_{(p-2k) \times (p-2k)}\\
	\end{bmatrix}.
	\end{equation*}
	We then define $\calp_{1}$ as the collection of $2^k$ matrices indexed by $\Theta$, 
	\begin{equation} \label{eq: def p11}
	\calp_{1}=\left\{ \Omega(\theta): \Omega(\theta)= (I_p-A(\theta))^T(I_p-A(\theta)),  \theta \in \Theta \right\}.
	\end{equation}
	Next, we construct $\calp_{2}$ as the collection of the diagonal matrices in the following equation, 
	\begin{equation} \label{eq: def p12}
	\begin{split}
	&\calp_{2}=  \left\lbrace  \vphantom{\big(\big)^2}\Omega(m) \equiv [w_{ij}(m)]_{p \times p}: \right.\\
	&\quad \quad \quad \quad \quad \quad \left. w_{ij}(m)= \big(\indc{i=j} + \tau a^{\frac{1}{2}}\indc{i=j=m}\big)^{-1}, m \in 0:p\right\rbrace,
	\end{split}
	\end{equation}
	where $a=\min\{\frac{\log p}{n}, 1\}$.
	
	\begin{lemma} \label{lmm: p1 p2 subset}
		$\calp_{1}$ and $\calp_{2}$ are subsets of $\paraspp$.
	\end{lemma}
	
	Note that we assume $\log p=O(n)$ and $n=O(p)$. Without loss of generality, we further assume ${\log p} < n < {p}$.
	For any estimator $\tilde{\Omega}$  based on $n$ i.i.d. observations, we establish the lower bounds over those two subsets in Sections \ref{sec: p11 assouad} and \ref{sec: p12 lecam} respectively,
	\begin{equation} \label{eq: p11 rate}
	\sup_{\calp_{1}}\mathbb{E}\opnormsq{\tilde{\Omega}-\Omega} \geq \lopratefull \geq \frac{\tau^2}{16}n^{-\frac{2\alpha-1}{2\alpha}},
	\end{equation}
	\begin{equation} \label{eq: p12 rate}
	\sup_{\calp_{2}}\mathbb{E}\opnormsq{\tilde{\Omega}-\Omega} \geq {\frac{\tau^2}{16} n^{-1} \min\{\log p,n\}} \geq \frac{\tau^2}{16}\frac{\log p}{n}.
	\end{equation}
	According to \prettyref{lmm: p1 p2 subset}, $(\calp_{1} \cup \calp_{2}) \subset \paraspp$. Therefore, we obtain
	\begin{align*}
	\sup_{\paraspp}\mathbb{E}\opnormsq{\tilde{\Omega}-\Omega} &\geq 
	\max\{
	\sup_{\calp_{1}}\mathbb{E}\opnormsq{\tilde{\Omega}-\Omega} ,  \quad \sup_{\calp_{2}}\mathbb{E}\opnormsq{\tilde{\Omega}-\Omega}
	\}\\
	&\geq 
	{\frac{\tau^2}{32} \Big( n^{-\frac{2\alpha-1}{2\alpha}}+ \frac{\log p}{n} \Big)},
	\end{align*}
	which completes the proof of \prettyref{thm: 1 lower bound in op}. 
\end{proof}

We introduce some further notation before establishing \prettyref{eq: p11 rate} using Assouad's lemma in \prettyref{sec: p11 assouad} and \prettyref{eq: p12 rate} using Le Cam's method in \prettyref{sec: p12 lecam}. Let $H(\theta, \theta')=\sum_{i=1}^k\vert \theta_i- \theta_i' \vert$ be the Hamming distance on $\{0,1\}^k$, which is the number of different elements between $\theta$ and $\theta'$. The total variation affinity  $\Vert P \wedge Q\Vert= \int p\wedge q ~~d\mu$, where $p$ and $q$ are the density functions of two probability measure $P$ and $Q$ with respect to any common dominating measure $\mu$. 
\subsubsection{Assouad's lemma in proof of (\ref{eq: p11 rate})}\label{sec: p11 assouad}
Assouad's lemma \citep{assouad1983deux} is a powerful tool to provide the lower bound over distributions indexed by the hypercube $\Theta = \{0,1\}^k$. Let $P_{\theta}$ be the distribution generated from observations indexed by $\Omega(\theta)$. The proof of Lemma \ref{lmm: assouad} can be found in \citep{yu1997assouad}, and thus omitted.
\begin{lemma}[Assouad]\label{lmm: assouad}
	Let $\tilde{\Omega}$ be an estimator based on observations from a distribution in the collection $\{{P}_\theta, \theta \in \Theta \}$, where $\Theta=\{0,1\}^k$. Then
	\begin{equation*}
	\sup_{\theta \in \Theta} 2^2 \mathbb{E}_{\theta}\Vert  \tilde{\Omega}-\Omega(\theta)\Vert_2^2 \geq \min_{H(\theta,\theta')\geq 1}\frac{\| \Omega(\theta)-\Omega(\theta')\|_2^2}{H(\theta,\theta')}\frac{k}{2}\min_{H(\theta,\theta')=1}\|{P}_\theta \wedge {P}_{\theta'}\|.
	\end{equation*}
\end{lemma}
Applying the Assouad's lemma to the subset $\calp_1$, we have the following results.
\begin{lemma} \label{lmm: 1 assouad}
	Let $P_\theta$ be the joint distribution of $n$ i.i.d. observations from $N(0, \Omega(\theta)^{-1})$, where $\Omega(\theta) \in \calp_1$ defined in \prettyref{eq: def p11}. Then
	$$\min_{H(\theta,\theta')=1}\|{P}_\theta \wedge {P}_{\theta'}\| \geq 0.5.$$
\end{lemma}
\begin{lemma} \label{lmm: 2 assouad}
	Consider all $\Omega(\theta) \in \calp_{1}$ defined in \prettyref{eq: def p11}. Then
	\begin{equation*}
	\min_{H(\theta,\theta')\geq 1}\frac{\| \Omega(\theta)-\Omega(\theta')\|_2^2}{H(\theta,\theta')} \geq (\lopelem)^2.
	\end{equation*}
\end{lemma}
Lemmas \ref{lmm: assouad}, \ref{lmm: 1 assouad} and \ref{lmm: 2 assouad} together imply the desired \prettyref{eq: p11 rate}, with the choice $k = \ceil{\lopband}$. The proofs of the above lemmas can be found in the supplement.

\subsubsection{Le Cam's method in proof of (\ref{eq: p12 rate})}\label{sec: p12 lecam}
Le Cam's method can be used to establish the lower bound via testing a single distribution against a convex hull of distributions. Set $r=\inf_{m\in[p]}\opnormsq{\Omega(0) - \Omega(m)}$.
Let $P_i$ be the distribution generated from observations indexed by $\Omega(i)$, where $0 \leq i \leq p$. Define $\bar{P}=\sum_{m=1}^p P_m$.
The proof of the following lemma can be found in \citep{yu1997assouad}, and thus omitted.
\begin{lemma}[Le Cam] \label{lmm: le cam}
	Let $\tilde{\Omega}$ be an estimator based on observations from a distribution in the collection $\{{P}_i, 0\leq i \leq p \}$. Then
	\begin{equation*}
	\sup_{0\leq m \leq p} \ep \opnormsq{\tilde{\Omega} - \Omega(m)} \geq \frac{1}{2} r\norm{P_0 \wedge \bar{P}}.
	\end{equation*} 
\end{lemma}
Applying Le Cam's method to $\calp_2$, we obtain that  $r=\Big(\frac{\tau a^{\frac{1}{2}}}{1 + \tau a^{\frac{1}{2}}}\Big)^2 \geq \frac{1}{4}\tau^2 a$ and the following results.
\begin{lemma} \label{lmm: le cam 1}
	Let $P_m$ be the joint distribution of $n$ i.i.d. observations from $N(0, \Omega(m)^{-1})$, where $\Omega(m) \in \calp_2$ defined in \prettyref{eq: def p12}. Then
	\begin{equation*}
	\norm{P_0 \wedge \bar{P}} > \frac{7}{8}.
	\end{equation*}
\end{lemma}
Combining the above results in Lemmas \ref{lmm: le cam} and \ref{lmm: le cam 1}, we obtain the desired (\ref{eq: p12 rate}), i.e., 
\begin{equation*}
\sup_{0\leq m \leq p} \ep \opnormsq{\tilde{\Omega} - \Omega} \geq \frac{7}{64}\tau^2a \geq \frac{\tau^2}{16} \min\{\frac{\log p}{n}, 1\}.
\end{equation*}

\subsection{Rate Optimality under the Operator Norm over $\paraspq$}\label{sec: risk op e}

\subsubsection{Minimax upper bound}
In this section, we establish the upper bound of the proposed local cropping estimator over $\paraspq$. Compared to that over $\paraspp$, the analysis here involves smaller bias terms, which lead to a different optimal bandwidth $k \asymp n^{\frac{1}{2\alpha+1}}$.
\begin{theorem} \label{thm: eop upper 1}
	When $\eopband \leq p$,  the local cropping estimator defined in \prettyref{eq: def est op} of the precision matrix $\Omega$ over $\paraspq$ given in \prettyref{eq: def paraspq} satisfies 
	\begin{equation*} \label{eq: lop upper thm11}
	\sup_{\paraspq} \ep\opnormsq{\tilde{\Omega}_k^{\op}-\Omega} \leq Ck^{-2\alpha}+C\frac{\log p + k}{n} ,
	\end{equation*}
	When $k=\eopband$, we have
	\begin{equation*} \label{eq: lop upper thm12}
	\sup_{\paraspq} \ep\opnormsq{\tilde{\Omega}_k^{\op}-\Omega} \leq Cn^{-\frac{2\alpha}{2\alpha+1}}+ C\frac{\log p}{n} .
	\end{equation*}
\end{theorem} 
\begin{proof}
	We employ the same proof strategy as that of \prettyref{thm: up op 1}. Only two lemmas bounding bias terms need to be replaced. We only emphasize the differences here.
	
	We replace \prettyref{lmm: tp omega close} in the proof by \prettyref{lmm: tp omega close e}, which bounds the distance of the population precision matrix and its tapered one.
	\begin{lemma} \label{lmm: tp omega close e}
		For  $\Omega$ in the parameter space $\paraspq$ defined in  \prettyref{eq: def paraspq}, $\Omega_k^*$ is defined in \prettyref{eq: tp def}, we have
		\begin{equation*} 
		\opnormsq{\Omega_k^* - \Omega} \leq Ck^{-2\alpha}.
		\end{equation*}
	\end{lemma}
	
	In addition, we replace \prettyref{lmm: bias in block up lop} by \prettyref{lmm: bias in block up lop e}, which bounds the bias term of each local estimator.
	\begin{lemma} \label{lmm: bias in block up lop e}
		For $\Omega \in \paraspq$ defined in \prettyref{eq: def paraspp} with $\alpha > 0$, we have
		\begin{equation*}
		\opnormsq{\cut{{k}}{k+1}\big( (\cut{{3k}}{m-k}( \Omega^{-1} ))^{-1} \big) - \cut{k}{m}(\Omega)}\leq Ck^{-2\alpha}.
		\end{equation*}
	\end{lemma}
	The remaining part of the proof remains the same, including a similar upper bound for the variance term stated in Lemma \ref{lmm: sample cov max}. Therefore, we complete our proof.
\end{proof}
\subsubsection{Minimax lower bound}
\begin{theorem} \label{thm: eop lower}
	The minimax risk for estimating the precision matrix $\Omega$ over $\paraspq$ defined in \prettyref{eq: def paraspq} under the operator norm with $\alpha>0$  satisfies
	\begin{equation}
	\inf_{\tilde{\Omega}}\sup_{\paraspq}\mathbb{E}\opnormsq{\tilde{\Omega}-\Omega} \geq
	\eopratefull.
	\end{equation}
\end{theorem}
\begin{remark}
	Theorems \ref{thm: eop upper 1} and \ref{thm: eop lower} together show that the minimax risk for estimating the precision matrices over $\paraspq$ stated in (\ref{eq: eop rate}) of Theorem \ref{thm: main theorem op}. In contrast to  $\paraspp$, the optimal rate of convergence over $\paraspq$ is faster. In particular, rate-optimal local cropping estimators are always consistent as long as $\alpha>0$.
\end{remark}
\begin{proof}
	To establish the lower bound for $\paraspq$ in which the decay of $a_{ij}$ is in the entry-wise fashion, we repeat the proof scheme in \prettyref{sec: low op l} with a few changes. Let $\tau$ be a positive constant which is less than $\taubd$.
	
	Set $k=\min\{\eopband,\frac{p}{2}\}$ and the index set $\Theta = \{0,1\}^k$, i.e.,  for any $\theta \in \Theta$, $\theta \equiv \{\theta_i\}_{1\leq i \leq k}$, each $\theta_i$ is either 0 or 1. Define the $k \times k$ matrix $B_k^*(\theta) \equiv [b_{ij}]_{k \times k}$ with $b_{ij}=\tau(nk)^{-\frac{1}{2}}\theta_i$. Define 
	\begin{equation*}
	B(\theta)=\begin{bmatrix}
	0_{k \times k} & 0_{k \times k} & 0_{k \times (p-2k)}\\
	B^*_k(\theta) & 0_{k \times k} & 0_{k \times (p-2k)}\\
	0_{(p-2k) \times k} & 0_{(p-2k) \times k} & 0_{(p-2k) \times (p-2k)}\\
	\end{bmatrix}.
	\end{equation*}
	We construct the collection of $2^k$ matrices as
	\begin{equation} \label{eq: def p21}
	\calp_{3}=\left\{ \Omega(\theta): \Omega(\theta)= (I_p-B(\theta))^T(I_p-B(\theta)),  \theta \in \Theta \right\}.
	\end{equation}
	\begin{lemma} \label{lmm: p3 subset}
		$\calp_3$ is a subset of $\paraspq$.
	\end{lemma}
	
	Let $P_\theta$ be the joint distribution of $n$ i.i.d. observations from $N(0, \Omega(\theta)^{-1})$, where $\Omega(\theta) \in \calp_3$ defined in \prettyref{eq: def p21}. Parallel to Lemmas \ref{lmm: 1 assouad} and \ref{lmm: 2 assouad} and the lower bound (\ref{eq: p11 rate}) for $\calp_1$, we establish the following lower bound for $\calp_3$.
	\begin{lemma} \label{lmm: e assouad}
		Consider all $\Omega(\theta) \in \calp_{3}$ defined in \prettyref{eq: def p21}. Then
		\begin{equation} \label{eq: lmm e 1}
		\min_{H(\theta,\theta')=1}\|{P}_\theta \wedge {P}_{\theta'}\| \geq 0.5,
		\end{equation}
		\begin{equation}  \label{eq: lmm e 2}
		\min_{H(\theta,\theta')\geq 1}\frac{\| \Omega(\theta)-\Omega(\theta')\|_2^2}{H(\theta,\theta')} \geq (\lopelem)^2.
		\end{equation}
		According to Assouad's lemma, for any estimator $\tilde{\Omega}$ based on $n$ i.i.d. observations, we have
		\begin{equation} \label{eq: p21 rate} 
		\sup_{\calp_{3}}\mathbb{E}\opnormsq{\tilde{\Omega}-\Omega} \geq \frac{\tau^2}{16}n^{-1} \min\{n^{\frac{1}{2\alpha+1}}, \frac{p}{2}\}.
		\end{equation}
	\end{lemma}
	It is easy to show $(\calp_{3} \cup \calp_{2}) \subset \paraspq$, where $\calp_{2}$ is defined in \prettyref{eq: def p12}.
	Therefore, combining \prettyref{eq: p21 rate} and \prettyref{eq: p12 rate}, we complete the proof of \prettyref{thm: eop lower}.
\end{proof}

\begin{remark}
	The estimation of the covariance matrix $\Sigma$ is of significant importance as well. We propose the estimator of $\Sigma$ by inverting our estimator $\tilde{\Omega}_k^{\op}$ given in \prettyref{eq: def est op}. The results and the analysis given in \prettyref{sec: risk op} can be used to establish the minimax optimality of our estimator under the operator norm.
	According to the inequality
	$
	\opnorm{(\tilde{\Omega}_k^{\op})^{-1}- {\Sigma}}\leq \opnorm{(\tilde{\Omega}_k^{\op})^{-1}} \opnorm{\tilde{\Omega}_k^{\op}- {\Omega}}\opnorm{\Omega^{-1}}
	$
	and the fact that both $\opnorm{\tilde{\Omega}_k^{\op})^{-1}}$ and $\opnorm{\Omega^{-1}}$ are bounded by $\eigenbd$, we establish the upper bound of our estimator $(\tilde{\Omega}_k^{\op})^{-1}$. Furthermore, considering the analog between the covariance matrix and the precision matrix in the subset $\calp_1$ and $\calp_2$ defined in \prettyref{eq: def p11} and \prettyref{eq: def p12}, the matching lower bound can be proved by a similar argument in \prettyref{sec: low op l}. Therefore, we have the following rate optimality of estimating the covariance matrix under the operator norm, which can be achieved by estimator $(\tilde{\Omega}_k^{\op})^{-1}$,
	\begin{equation*} 
	\inf_{\tilde{\Omega}}\sup_{\paraspp} \mathbb{E}\opnormsq{\tilde{\Omega}^{-1}-\Omega^{-1}} \asymp 
	n^{-\frac{2\alpha-1}{2\alpha}}+ \frac{\log p}{n}  ,
	\end{equation*}
	\begin{equation*} 
	\inf_{\tilde{\Omega}}\sup_{\paraspq} \ep\opnormsq{\tilde{\Omega}^{-1}-\Omega^{-1}} \asymp  n^{-\frac{2\alpha}{2\alpha+1}}+ \frac{\log p}{n} .
	\end{equation*}
\end{remark}

\section{Rate Optimality under the Frobenius Norm}\label{sec: risk f}
In this section, we establish that the optimal rates of convergence for estimating the precision matrix over the parameter spaces $\paraspp$ and $\paraspq$ are identical under the Frobenius norm. Intuitively, estimating precision matrix under the  Frobenius norm is equivalent to estimating each row of Cholesky factor $A$ under the $\ell_2$ vector norm. Consequently, it is not a surprise to see that $\paraspq$ and $\paraspp$ enjoy the same optimal rates here. Since $\mathcal{Q}_\alpha(\eigenbd, \alpha M) \subset \paraspp$, one immediately obtain that
\begin{align}
\inf_{\tilde{\Omega}}\sup_{\paraspp} \frac{1}{p} \ep\fnormsq{\tilde{\Omega}-\Omega} &\geq \inf_{\tilde{\Omega}}\sup_{\mathcal{Q}_\alpha(\eigenbd, \alpha M)} \frac{1}{p} \ep\fnormsq{\tilde{\Omega}-\Omega}. \label{eq: low relation f}
\end{align}
In order to show (\ref{eq: lef rate}) in Theorem \ref{thm: main theorem f}, it suffices to establish the upper bound over the parameter space $\paraspp$ and the matching lower bound over the parameter space $\paraspq$. We assume that $\mx$ follows the $p$-variate Gaussian distribution, with mean zero and precision matrix $\Omega$ in this section.
\subsection{Minimax Upper Bound under the Frobenius Norm }\label{sec: up f}
In this section, we establish the following risk upper bound of the regression-based thresholding estimation procedure we proposed in \prettyref{sec: est f} under the Frobenius norm over $\paraspp$. 
\begin{theorem} \label{thm: upper in f}
	Assume $\ceil{n^{\frac{1}{2\alpha+2}} } \leq p $. The estimator defined in  \prettyref{eq: lf 1 est} of the precision matrix $\Omega$ over $\paraspp$ and $\mathcal{Q}_\alpha(\eigenbd, \alpha M)$ given in \prettyref{eq: def paraspp} and \prettyref{eq: def paraspq} with $k=\ceil{n^{\frac{1}{2\alpha+2}} }$ satisfies
	\begin{equation} \label{eq: upper in f}
	\sup_{\mathcal{Q}_\alpha(\eigenbd, \alpha M)} \frac{1}{p} \ep\fnormsq{\tilde{\Omega}_k^{\f}-\Omega} \leq \sup_{\paraspp} \frac{1}{p} \ep\fnormsq{\tilde{\Omega}_k^{\f}-\Omega} \leq C\efrateup.
	\end{equation}
\end{theorem}
\begin{proof}
	We focus on the second inequality since the first one is trivial. Note that $\tilde{\Omega}_k^{\f} =(\widetilde{I - A})^T\tilde{D}^{-1}(\widetilde{I - A})$ according to \prettyref{eq: 
		lf 1 est} while $\Omega=(I-A)^T D^{-1}(I-A)$. The risk upper bound can be controlled by bounding $\widetilde{I-A}-(I-A)$ and $\tilde{D}-D$. To this end, we first provide some properties of our estimator.
	\begin{lemma}\label{lmm: lf upper thresholding}
		Assume that $\mx$ follows the $p$-variate Gaussian distribution with mean zero and precision matrix $\Omega=(I-A)^TD^{-1}(I-A)$, which belongs to parameter space $\paraspp$ defined in \prettyref{eq: def paraspp}. For any fixed $i$, $d_i$ is the $i$-th diagonal of $D$, $\mathbf{a}_i \in \RR^{i-1}$ corresponds the $i$-th row of the lower triangle in $A$. $\hat{d}_i$ is defined in \prettyref{eq: def e}, and $\hat{\mathbf{a}}_i^* \in \RR^{i-1}$ corresponds the $i$-th row of the lower triangle in $\hat{A}^*$ defined in \prettyref{eq: def thresholding}. Then we have
		\begin{align*}
		\ep\abs{\hat{d}_i-d_i}^2 &\leq C\efrateup, \\
		\ep \vecnormsq{\hat{\mathbf{a}}^*_i-\mathbf{a}_i} &\leq C\efrateup.
		\end{align*}
	\end{lemma}
	We are ready to establish the upper bounds of $\widetilde{I-A}-(I-A)$ and $\tilde{D}-D$ separately. Note that $\opnorm{\tilde{D}^{-1}} \leq \eigenbd$ and  $\opnorm{D^{-1}} \leq \eigenbd$, which is due to \prettyref{lmm: prop of paraspp}. Therefore, \prettyref{lmm: projection} yields $\ep \fnormsq{\tilde{D} - D } \leq 4 \ep \fnormsq{\hat{D} -D }$, which further implies that
	\begin{align*}
	\frac{1}{p} \ep\fnormsq{D^{-1}-\tilde{D}^{-1}}
	&\leq \frac{1}{p} \ep \opnormsq{\tilde{D}^{-1}}\fnormsq{\tilde{D} -D } \opnormsq{D^{-1}} \\
	&\leq 4\eigenbd^4\frac{1}{p}\ep\fnormsq{\hat{D}-D}\\
	&\leq 4\eigenbd^4\frac{1}{p} \sum_i \ep\abs{\hat{d}_i-d_i}^2.
	\end{align*} 
	Together with \prettyref{lmm: lf upper thresholding}, it follows that  
	\begin{equation} \label{eq: lf upper 1}
	\frac{1}{p} \ep\fnormsq{D^{-1}-\tilde{D}^{-1}} \leq C\efrateup.
	\end{equation}
	Next, we turn to prove that  $\frac{1}{p} \ep\fnormsq{\widetilde{(I-A)}-(I-A)} \leq C\efrateup.$ 
	\prettyref{lmm: projection} implies 
	\begin{equation*}
	\frac{1}{p} \ep\fnormsq{\widetilde{(I-A)}-(I-A)} \leq \frac{4}{p} \ep\fnormsq{\hat{A}^*-A}\leq \frac{4}{p} \sum_i \ep \vecnormsq{\hat{\mathbf{a}}^*_i-\mathbf{a}_i}.
	\end{equation*}
	Combining above equation with \prettyref{lmm: lf upper thresholding}, we have 
	\begin{equation} \label{eq: lf upper 2}
	\frac{1}{p} \ep\fnormsq{\widetilde{(I-A)}-(I-A)} \leq C\efrateup.
	\end{equation}
	
	At last, we derive the risk upper bound of our estimator.
	It is clear that  $\opnorm{\widetilde{I-A}} \leq \eigenbd$, $\opnorm{\tilde{D}^{-1}} \leq \eigenbd$. According to \prettyref{lmm: prop of paraspp}, $\opnorm{I-A} \leq \eigenbd$, $\opnorm{D^{-1}}\leq \eigenbd$. Combining these facts with \prettyref{eq: lf upper 1} and \prettyref{eq: lf upper 2}, we have
	\begin{align*}
	\frac{1}{p} \ep \fnormsq{\tilde{\Omega}-\Omega} \leq&  \frac{3}{p} \ep \big( \opnormsq{I-A} \opnormsq{D^{-1}} \fnormsq{\widetilde{(I-A)}-(I-A)}\\ &\quad + \opnormsq{I-A}\fnormsq{D^{-1}-\tilde{D}^{-1}}\opnormsq{\widetilde{I-A}}\\ &\quad + \fnormsq{\widetilde{(I-A)}-(I-A)}\opnormsq{\tilde{D}^{-1}}\opnormsq{\widetilde{I-A}} \big)\\
	\leq& 6\eigenbd^4 \frac{1}{p} \ep \fnormsq{\widetilde{(I-A)}-(I-A)} + 3\eigenbd^4 \frac{1}{p} \ep \fnormsq{D^{-1}-\tilde{D}^{-1}}\\
	\leq& C\efrateup.
	\end{align*}
	Therefore, we finish the proof of \prettyref{thm: upper in f}.
\end{proof}

\subsection{Minimax Lower Bound under the Frobenius Norm }\label{sec: low f}
In this section, we establish the matching lower bound $\efrateup$ over parameter spaces $\paraspp$ and $\paraspq$. 

\begin{theorem}\label{thm: lower in f}
	The minimax risk for estimating the precision matrix $\Omega$ over $\paraspp$ and $\mathcal{Q}_\alpha(\eigenbd, \alpha M)$ under the Frobenius norm satisfies
	\begin{align*}
	\inf_{\tilde{\Omega}} \sup_{\paraspp} \frac{1}{p} \mathbb{E}\Vert \tilde{\Omega}-\Omega\Vert_F^2 \geq \inf_{\tilde{\Omega}} \sup_{\mathcal{Q}_\alpha(\eigenbd, \alpha M)} \frac{1}{p} \mathbb{E}\Vert \tilde{\Omega}-\Omega\Vert_F^2    
	\geq \frac{\tau^{2}}{32}n^{-\frac{2\alpha+1}{2\alpha+2}}.
	\end{align*}
\end{theorem}
\begin{remark}
	The minimax risk for estimating the precision matrices over $\paraspp$ and $\paraspq$ under the Frobenius norm in Theorem \ref{thm: main theorem f} immediately follows from Theorems \ref{thm: upper in f} and \ref{thm: lower in f}.
\end{remark}

\begin{proof}
	It is sufficient to establish the lower bound over $\paraspq$ since the first inequality immediately follows from \prettyref{eq: low relation f}. We construct a least favorable subset in $\paraspq$. Without loss of generality, we assume $
	\frac{p}{2k}$ is an integer where $k=\min\{\ceil{n^{\frac{1}{2\alpha+2}}},\frac{p}{2}\}$. Define the index set $\Theta' = \{0,1\}^{\frac{kp}{2}}$. For each $\theta \in \Theta'$, we further denote it as $\frac{p}{2k}$ many $k^2$ dimensional vectors, i.e., $\theta= \{\theta(s)\}_{1\leq s\leq \ceil{\frac{p}{2k}}}$, where $\theta(s)_{ij}$ is equal to $0$ or $1$. For such an index $\theta$, there is a corresponding $p\times p$ block diagonal matrix $C(\theta)$ such that each $k\times k$ block $C_s(\theta(s))\equiv [c(s)_{ij}]_{k \times k}$, where $c(s)_{ij}=\efelem \theta(s)_{ij}$, $s \in \ceil{\frac{p}{2k}}$. We set $\tau$ as a positive constant which is less than $\taubd$.
	\begin{equation*}
	C(\theta)= 
	\begin{bmatrix}
	\fbox{$\begin{matrix}
		0_k & 0_k \\
		C_1(\theta(1))  & 0_k 
		\end{matrix}$}   & 0_{2k} & \hdots & 0_{2k} \\
	0_{2k} & \fbox{$\begin{matrix}
		0_k & 0_k \\
		C_2(\theta(2)) & 0_k 
		\end{matrix}$} & \hdots & 0_{2k} \\
	\vdots & \vdots &  \ddots & \vdots \\
	0_{2k} & 0_{2k} & \hdots
	& \fbox{$\begin{matrix}
		0_k & 0_k \\
		C_{\ceil{\frac{p}{2k}}}(\theta(\ceil{\frac{p}{2k}}))  & 0_k 
		\end{matrix}$}   
	\end{bmatrix}.
	\end{equation*}
	Finally, we define the subset of $\paraspq$ indexed by $\Theta'$ as follows
	\begin{equation} \label{eq: def p4}
	\calp_{4}=\left\{ \Omega(\theta): \Omega(\theta)= (I_p-C(\theta))^T(I_p-C(\theta)),  \theta \in \Theta' \right\}.
	\end{equation}
	\begin{lemma}\label{lmm: p4 subset}
		$\calp_{4}$ is a subset of $\paraspq$.
	\end{lemma}
	Applying \prettyref{lmm: assouad} to $\calp_4$, we obtain that,
	\begin{equation}\label{eq: use lemm ef low 12}
	\inf_{\tilde{\Omega}} \max_{\theta \in \Omega(\Theta')} 2^2 \mathbb{E}_{\theta}\Vert  \tilde{\Omega}-\Omega(\theta)\Vert_F^2 \geq \min_{H(\theta,\theta')\geq 1}\frac{\| \Omega(\theta)-\Omega(\theta')\|_F^2}{H(\theta,\theta')}\frac{kp}{4}\min_{H(\theta,\theta')=1}\|{P}_\theta \wedge {P}_{\theta'}\|
	\end{equation}
	\begin{lemma} \label{lmm: ef low}
		Let $P_\theta$ be the joint distribution of $n$ i.i.d. observations from $N(0, \Omega(\theta)^{-1})$, where $\Omega(\theta) \in \calp_4$ defined in \prettyref{eq: def p4}. Then
		\begin{equation}\label{eq: lmm ef low 1}
		\min_{H(\theta,\theta')=1}\|{P}_\theta \wedge {P}_{\theta'}\|\geq 0.5,
		\end{equation}
		and
		\begin{equation}\label{eq: lmm ef low 2}
		\min_{H(\theta,\theta')\geq 1}\frac{\| \Omega(\theta)-\Omega(\theta')\|_F^2}{H(\theta,\theta')}\geq \tau^{2}n^{-1}.
		\end{equation}
	\end{lemma}
	Applying Lemma \ref{lmm: ef low} into (\ref{eq: use lemm ef low 12}), we obtain
	\begin{equation*}
	\inf_{\tilde{\Omega}} \sup_{\paraspq} \frac{1}{p} \mathbb{E}\Vert \tilde{\Omega}-\Omega\Vert_F^2  \geq \inf_{\tilde{\Omega}} \sup_{\calp_4} \frac{1}{p} \mathbb{E}\Vert \tilde{\Omega}-\Omega\Vert_F^2 \geq \efratefull,
	\end{equation*}
	which completes the proof of Theorem \ref{thm: lower in f}, noting that $n<p$.
\end{proof}
\section{Adaptive Estimation}\label{sec: Adaptivity}
To achieve the minimax rates in Theorem \ref{thm: main theorem op} under the operator norm, the local cropping estimator $\tilde{\Omega}_{k}^{\op}$  requires the knowledge of smoothness parameter $\alpha$ as the optimal choice of bandwidth $k=\ceil{\lopband}$ and $k=\eopband$ over $\paraspp$ and $\paraspq$ respectively. In this section, we consider adaptive estimation where the goal is to construct a single procedure which is minimax rate optimal simultaneously over each parameter space $\paraspp$ ($\alpha>1/2$) and $\paraspq$ ($\alpha>0$). Throughout this section, we assume that $\mx$ follows certain sub-Gaussian distribution defined in (\ref{eq: def sub gaussian}).

Recall that for each $k$, the local cropping estimator $\tilde{\Omega}_{k}^{\op}$ is defined in \prettyref{eq: def est op}. Without the knowledge of $\alpha$,  the bandwidth $k$ needs to be picked in a data-driven fashion. Motivated by the Lepski's methods for nonparametric function estimation problems \cite{lepskii1992asymptotically}, we select the bandwidth $\hat{k}$ through the following procedure,
\begin{equation}\label{eq: adaptive bd}
\hat{k} = \min\{ k \in \mathcal{H}: \opnormsq{\tilde{\Omega}_{k}^{\op} - \tilde{\Omega}_{l}^{\op}} \leq C_L\frac{\log p + l}{n}, \text{ for all } l \geq k\},
\end{equation}
where $\mathcal{H} = \{1, 2, \dots \ceil{\frac{n}{\log p}}\}$ and $C_L>0$ is a sufficiently large constant. If the set that is minimized over is empty, we use the convention $\hat{k}=\ceil{\frac{n}{\log p}}$.
The adaptive local cropping estimator $\tilde{\Omega}^{\op}_{\hat{k}}$ enjoys the following theoretical guarantee, and thus is adaptive minimax rate optimal.

\begin{theorem} \label{thm: adap op p}
	Assume $\log p =O (n)$, $ n=O(p)$. Then the adaptive estimator $\tilde{\Omega}_{\hat{k}}^{\op}$ with $\hat{k}$ defined in \prettyref{eq: adaptive bd} of the precision matrix $\Omega$ over $\paraspp$ with $\alpha>\frac{1}{2}$ satisfies 
	\begin{equation*} 
	\sup_{\paraspp} \ep\opnormsq{\tilde{\Omega}^{\op}_{\hat{k}}-\Omega} \leq C n^{-\frac{2\alpha-1}{2\alpha}}+ C\frac{\log p}{n} .
	\end{equation*}
	In addition, the adaptive estimator $\tilde{\Omega}_{\hat{k}}^{\op}$ over $\paraspq$ with $\alpha>0$ satisfies
	\begin{equation*} 
	\sup_{\paraspq} \ep\opnormsq{\tilde{\Omega}^{\op}_{\hat{k}}-\Omega} \leq C n^{-\frac{2\alpha}{2\alpha+1}}+ C\frac{\log p}{n} .
	\end{equation*}
\end{theorem}

\begin{proof}
	We only show the upper bound over $\paraspp$ with $\alpha>\frac{1}{2}$. The proof over space $\paraspq$ with $\alpha>0$ can be shown similarly and thus omitted. 
	
	Set the oracle bandwidth $k^* = \ceil{\lopband}$. For any $\Omega\in\paraspp$, we decompose the risk as follows, 
	\begin{equation}\label{eq: adap temp1}
	\ep\opnormsq{\tilde{\Omega}^{\op}_{\hat{k}}-\Omega} \leq 2\ep \opnormsq{\tilde{\Omega}^{\op}_{\hat{k}}-\tilde{\Omega}^{\op}_{{k}^*}} +2\ep\opnormsq{\tilde{\Omega}^{\op}_{{k}^*}-\Omega}.
	\end{equation}
	Since $k^*$ is deterministic, we immediately obtain from \prettyref{thm: up op 1} that 
	\begin{equation}\label{eq: adap temp2}
	\ep\opnormsq{\tilde{\Omega}^{\op}_{k^*}-\Omega} \leq Cn^{-\frac{2\alpha-1}{2\alpha}}+ C\frac{\log p}{n},
	\end{equation}
	which controls the second term of the risk decomposition (\ref{eq: adap temp1}). 
	
	We turn to bound the first term of (\ref{eq: adap temp1}). Due to the definition of $\hat{k}$ and $k^*$, we have that on the event $\{\hat{k} \leq k^*\}$,
	\begin{equation}\label{eq: adap temp}
	\opnormsq{\tilde{\Omega}^{\op}_{\hat{k}}-\tilde{\Omega}^{\op}_{{k}^*}} \leq  C_L\frac{\log p +k^*}{n}\leq Cn^{-\frac{2\alpha-1}{2\alpha}}+ C\frac{\log p}{n}.
	\end{equation}
	It suffices to show that $\hat{k} \leq k^*$ with high probability. The following lemma, a probability version of \prettyref{thm: up op 1}, facilitates our proof of this claim.
	
	\begin{lemma}\label{lmm: prob op p}
		Assume $\ceil{\lopband} \leq p$. Then for any constant $C_1>0$, there exists a sufficiently large constant $C>0$ irrelevant of $\alpha$ such that the local cropping estimator defined in \prettyref{eq: def est op}  satisfies 
		\begin{equation*} \label{eq: up op 1}
		\pb (\opnormsq{\tilde{\Omega}^{\op}_k-\Omega} \leq Ck^{-2\alpha+1}+C\frac{\log p + k}{n}) > 1- \exp(-C_1(\log p + k)),
		\end{equation*}
		simultaneously for each $k\in \mathcal{H}$ and each $\Omega \in {\paraspp}$ with $\alpha>\frac{1}{2}$.
	\end{lemma}
	
	Notice that for any $l$, we have $\opnormsq{\tilde{\Omega}_{k^*}^{\op} - \tilde{\Omega}_{l}^{\op}} \leq 2\opnormsq{\tilde{\Omega}_{k^*}^{\op} - \Omega} + 2\opnormsq{\Omega - \tilde{\Omega}_{l}^{\op}}$. Thus, 
	\begin{align}
	&\pb(\hat{k} > k^*) \notag \\
	\leq& \sum_{l \geq k^*} \pb (\opnormsq{\tilde{\Omega}_{k^*}^{\op} - \tilde{\Omega}_{l}^{\op}} > C_L\frac{\log p + l}{n})  \notag \\
	\leq&\sum_{l \geq k^*}\Big(\pb(\opnormsq{\tilde{\Omega}_{k^*}^{\op} - \Omega} > \frac{C_L}{4}\frac{\log p + k^*}{n}) +\pb (\opnormsq{\tilde{\Omega}_{l}^{\op} - \Omega} >\frac{C_L}{4}\frac{\log p + l}{n})\Big)  \notag\\
	\leq& n\Big(\exp\big(-C_1(\log p + k^*)\big) + \exp\big(-C_1(\log p + l)\big)\Big) \notag\\
	\leq& n^{-1}\eta^{-2} \label{eq: adap temp3}.
	\end{align}
	We have used  the fact  $k^* \leq l$ and the definition of $k^*$ in the inequalities above, noting that a sufficiently large $C_1>0$ can be picked  to guarantee the last inequality holds. The second to last inequality holds because of \prettyref{lmm: prob op p} and a sufficiently large $C_L$. Therefore, we have shown
	that the event $\hat{k} \leq k^*$ holds with probability at least $1 -n^{-1}\eta^{-2}$.

	In the end, combining (\ref{eq: adap temp1})-(\ref{eq: adap temp3}), we obtain that for any $\Omega\in\paraspp$, 
	\begin{align*}
	&\ep\opnormsq{\tilde{\Omega}^{\op}_{\hat{k}}-\Omega}\\
	\leq& 2\ep\opnormsq{\tilde{\Omega}^{\op}_{k^*}-\Omega}+2\ep\big(\opnormsq{\tilde{\Omega}^{\op}_{\hat{k}}-\tilde{\Omega}^{\op}_{{k}^*}}:\hat{k}\leq k^*\big) +2 \ep\big(\opnormsq{\tilde{\Omega}^{\op}_{\hat{k}}-\tilde{\Omega}^{\op}_{{k}^*}}:\hat{k}>k^*\big)\\
	\leq &Cn^{-\frac{2\alpha-1}{2\alpha}}+ C\frac{\log p}{n} + 8\eta^2\pb(\hat{k}>k^*)\\
	\leq &Cn^{-\frac{2\alpha-1}{2\alpha}}+ C\frac{\log p}{n} + 8n^{-1}\\
	\leq &C(n^{-\frac{2\alpha-1}{2\alpha}}+ \frac{\log p}{n}),
	\end{align*}
	where we also used that $\opnormsq{\tilde{\Omega}^{\op}_{\hat{k}}-\tilde{\Omega}^{\op}_{{k}^*}}\leq 4\eta^2$ in the second inequality. Therefore, we complete the proof.
\end{proof}

\section{An Extension to Nonparanormal Distributions}\label{sec: rank}
In this section, we extend the minimax framework to the nonparanormal model.
Assume that $\textbf{X}=(X_1, X_2, \dots, X_p)^T$ follows the $p$-variate Gaussian distribution with covariance matrix ${\Sigma}$. 
Instead of $n$ i.i.d. copies $\mx_1,\mx_1,\dots,\mx_n$ of $\textbf{X}$, we only observe their transformations. Specifically,
we denote the transformed variables of $\mx$ by $\textbf{Y}=(f_1(X_1), f_2(X_2), \dots, f_p(X_p))^T$, where each $f_i$ is some unknown strictly increasing function. 
Then our observation is $\mz=(\textbf{Y}_1,\textbf{Y}_2, \dots,\textbf{Y}_n)^T \in \mathbb{R}^{n\times p}$, where each $\textbf{Y}_i$ is the transformed $\mx_i$.
This is a form of the Gaussian copula model \citep{bickel1993efficient}, or the nonparanormal model \citep{liu2009nonparanormal}.
To avoid the identifiability issue, we set $\diag(\Sigma) = I$, which makes $\Sigma$ the correlation matrix.
Here we consider the same structural assumption as in previous sections on the inverse of the correlation matrix, which is denoted by $\Omega$. Based on $\paraspp$ and $\paraspq$ defined in \prettyref{eq: def paraspp} and \prettyref{eq: def paraspq}, the following two types of parameter spaces are of interest,
\begin{equation}\label{eq: def nparaspp}
\calp'_\alpha(\eigenbd, M) = \left\{\{\Omega, \{f_i\}\}: 
\begin{split}
&\diag(\Omega^{-1}) = I, \quad \Omega \in \paraspp;\\ 
&f_i \text{ is strictly increasing, } i\in[p].
\end{split} \right\},
\end{equation}
and  
\begin{equation}\label{eq: def nparaspq}
\calq'_\alpha(\eigenbd, M) = \left\{\{\Omega, \{f_i\}\}: 
\begin{split}
&\diag(\Omega^{-1}) = I, \quad \Omega \in \paraspq;\\ 
&f_i \text{ is strictly increasing, } i\in[p].
\end{split} \right\}.
\end{equation}

Our goal is to estimate the latent correlation structure, the inverse of the correlation matrix $\Omega$, using the observation $\mz$. We establish the minimax risk of estimating $\Omega$ over the parameter spaces $\nparaspp$ and $\nparaspq$ under the operator norm in the following theorem.
\begin{theorem} \label{thm: npara}
	Assume $\log p =O (n)$, $ n=O(p)$. Then for the nonparanormal model, the minimax risk of estimating $\Omega$ under the operator norm over $\nparaspp$ with $\alpha > \frac{1}{2}$ satisfies
	\begin{equation} \label{eq: optimality nparaspp}
	\inf_{\tilde{\Omega}} \sup_{\{\Omega, \{f_i\}\} \in \nparaspp} \ep \opnorm{\tilde{\Omega} - \Omega}^2 \asymp n^{-\frac{2\alpha-1}{2\alpha}} + \frac{\log p}{n}.
	\end{equation}
	The minimax risk of estimating $\Omega$ under the operator norm over $\nparaspq$ satisfies
	\begin{equation} \label{eq: optimality nparaspq}
	\inf_{\tilde{\Omega}} \sup_{\{\Omega, \{f_i\}\} \in \nparaspq} \ep \opnorm{\tilde{\Omega} - \Omega}^2 \asymp n^{-\frac{2\alpha}{1+2\alpha}} + \frac{\log p}{n}.
	\end{equation}
\end{theorem}

Finally, we introduce our rate-optimal estimation procedure over the parameter spaces $\nparaspp$ and $\nparaspq$ under the operator norm.
The approach to estimate the inverse of the correlation matrix in nonparanormal model is almost the same as the estimation scheme of the precision matrix under the operator norm in \prettyref{sec: est op}, except that the sample covariance matrix needs to be replaced by its rank-based nonparametric variant via Kendall's tau ($\tau$) \cite{kendall1938new} or Spearman's correlation coefficient rho ($\rho$) \cite{spearman1904spearman}. Rank-based estimator are widely applied in the nonparanormal model. Progress has been made in this field during the last decade especially for high-dimensional statistics. For instance, see \cite{mitra2014multivariate}  for bandable correlation matrix estimation, \citep{barber2015rocket} for Gaussian graphical models, and \cite{fan2016multitask} for multi-task regression via Cholesky decomposition.

Kendall's tau is defined as 
\begin{equation*}
\hat{\tau}_{ij}=\frac{2}{n(n-1)}\sum_{1\leq k_1 < k_2 \leq n} \sgn(Z_{k_1i}-Z_{k_2i})\sgn(Z_{k_1j}-Z_{k_2j}).
\end{equation*}
Then define
\begin{equation}
\hat{\Sigma}^{\tau}=[\sin (\frac{\pi}{2}\hat{\tau}_{ij})]_{p \times p}.
\end{equation}
Spearman's rho is defined as 
\begin{equation*}
\hat{\rho}_{ij}=\frac{\sum_{k=1}^n(r_{ki}-(n+1)/2)(r_{kj}-(n+1)/2)}{\sqrt{\sum_{k=1}^n(r_{ki}-(n+1)/2)^2\sum_{k=1}^n(r_{kj}-(n+1)/2)^2}},
\end{equation*}
where $r_{ij}$ is the rank of $Z_{ij}$ among $Z_{1j}, Z_{2j}, \dots, Z_{nj}$.
Define
\begin{equation}
\hat{\Sigma}^{\rho}=[2\sin (\frac{\pi}{6}\hat{\rho}_{ij})]_{p \times p}.
\end{equation}
It is well-known that both $\hat{\Sigma}^{\tau}$ and $\hat{\Sigma}^{\rho}$ are unbiased estimators of the population correlation matrix $\Sigma$. We adopt almost the same estimation procedure proposed in \prettyref{sec: est op}, but replacing $\frac{1}{n}\mz^T\mz$ in  \prettyref{eq: 1 leop est} with either $\hat{\Sigma}^{\tau}$ or $\hat{\Sigma}^{\rho}$. In this way, we construct the nonparametric local cropping estimators $\tilde\Omega_k^{\tau}$ and $\tilde\Omega_k^{\rho}$ in replace of 
$\tilde{\Omega}_{k}^{\op}$ in (\ref{eq: def est op}). Note that the optimal choices of the bandwidth $k$ are picked differently over two types of parameter spaces $\nparaspp$ and $\nparaspq$ as we did over $\paraspp$ and $\paraspq$ in Section \ref{sec: est op}.  To provide some technical insights, we rely on some recent results in \cite{mitra2014multivariate} to bound the variance of each local estimator in  \prettyref{eq: 1 leop est} under the operator norm, which is the key to establish the upper bounds in Theorem \ref{thm: npara}.

\section{Numerical Studies} \label{sec: simulation}
In this section, we turn to the numerical performance of the proposed rate-optimal estimators under the operator norm for $\paraspp$ and $\paraspq$ defined in \prettyref{eq: def paraspp} and \prettyref{eq: def paraspq} to further illustrate the fundamental difference of $\paraspp$ and $\paraspq$. In addition, we compare them with the banding estimator proposed in \citep{bickel2008covariance}, which is based on the auto-regression between variables. Specifically, for a given bandwidth $k<n$, the banding estimator is defined as $\tilde{\Omega}%
^{BL}=( I-\tilde{A}^{BL})^T( \tilde{D}^{BL}) ^{-1}( I-\tilde{A}^{BL})$. Here the $i$-th row of the lower triangular matrix $\tilde{A}^{BL}$ is the vector $\hat{\mathbf{a}}_i$ in (\ref{eq: def thresholding coeff}), i.e., the least square estimates of the coefficients for the regression of $X_i$ against $\mx_{i-k:i-1}$. The $i$-th entry of the diagonal matrix $\tilde{D}^{BL}$ is the estimate of the residual variance for the regression of $X_i$ against $\mx_{i-k:i-1}$.

\subsection{Simulation in $\paraspq$ under the operator norm}\label{sec: simulation eop}
We first focus on the parameter space $\paraspq$ and compare the performance of local cropping estimator and the banding estimator. Specifically, we generate the precision matrix in the following form:
\begin{equation*}
\Omega=(I-A)^TD^{-1}(I-A), \quad
A\equiv [a_{ij}]_{p \times p}, \quad D=I_p,
\end{equation*}
where $a_{ij}=-(i-j)^{-\alpha-1}$ when $i>j$; otherwise $a_{ij}=0$. It is easy to check that $\Omega \in \calq_\alpha(\eigenbd, 1)$ with some large $\eta>0$.
The simulation is done with a range of parameter values for $p$, $n$, $\alpha$. Specifically, the decay rate $\alpha$ ranges from 0.5 to 2 with a step of 0.5, the sample size $n$ ranges from 500 to 4000, the dimension $p$ ranges from 500 to 2000.

In this setting, we compare our local cropping estimator (denoted as cropping.Q.) with the banding estimator (denoted as BL) proposed in \citep{bickel2008covariance}. According to \citep{bickel2008covariance}, the bandwidth of banding estimator is chosen as $k \asymp {(n/\log p)^{1/(2\alpha+2)}}$. The optimal bandwidth over $\paraspq$ is $k \asymp {n^{1/(2\alpha + 1)}}$. In the simulation, the bandwidth of BL estimator is $\floor{(n/\log p)^{1/(2\alpha+2)}}$ and the bandwidth of crop.Q is $\floor{n^{1/(2\alpha + 1)}}$.

\prettyref{tab: simulation 1} reports the average errors of the banding estimator (BL) and local cropping estimator (crop.Q) under the operator norm over 100 replications. The smaller errors in each experiment are highlighted in boldface. \prettyref{fig: simulation 1} displays the boxplots of the errors of BL and crop.Q.

It can be seen from \prettyref{tab: simulation 1} that crop.Q outperforms BL in most cases with a few exceptions when $n$ is small. As the sample size increases, the average errors of both methods decrease, which matches our intuition. In addition, the dimension $p$ has minor effect on the errors of both estimators, which is partially reflected by the optimal rates (dominating term $n^{-\frac{2\alpha}{2\alpha+1}}$) obtained in Theorem \ref{thm: main theorem op}. For each fixed dimension $p$, the superiority crop.Q over BL becomes more significant as the sample size $n$ increases, which implies that BL estimator is indeed sub-optimal. 

\subsection{Simulation in $\paraspp$ under the operator norm}\label{sec: simulation lop}
We demonstrate the fundamental difference between two types of parameter space $\paraspp$ and $\paraspq$ by numerical studies in this section. Of note, although local cropping estimators proposed in \prettyref{eq: def est op} are rate-optimal over both $\paraspp$ and $\paraspq$, the corresponding optimal choices of bandwidth are distinct. We generate precision matrices in the following way to guarantee that $\Omega$ is always in $\paraspp$ but not in $\paraspq$ with some fixed $\eta$ and $M$. Considering
\begin{equation*}
\Omega=(I-A)^TD^{-1}(I-A), \quad
A\equiv [a_{ij}]_{p \times p}, \quad D=I_p,
\end{equation*}
where the first column of $A$ is $a_{i1}=-2(i-1)^{-\alpha}$,  $2 \leq i \leq p$. The remaining entries are all zeros. It is easy to check that $\Omega \in \calp_\alpha(\eigenbd, 2)$ with some large $\eta>0$.
The simulation is carried out with a similar range of values for $p$, $n$, $\alpha$ as in Section \ref{sec: simulation eop}. Note that the consistent estimator exists only if $\alpha > 0.5$. Therefore, in this setting, the decay rate $\alpha$ varies among $1$, $1.5$ and $2$.

The optimal choice of bandwidth of local cropping estimator over $\paraspp$ is $k \asymp {n^{\frac{1}{2\alpha}}}$, which is different from the one of crop.Q. We denote this rate-optimal estimator in $\paraspp$ by crop.P. In the simulation, the bandwidth of crop.P is $\floor{n^{\frac{1}{2\alpha}}}$. We also include BL estimator as a reference.

\prettyref{tab: simulation 2} reports the average errors of the three procedures, crop.P, crop.Q and BL, under the operator norm over $100$ replications. The smallest errors in each experiment are highlighted in boldface. \prettyref{fig: simulation 2} plots the boxplots of their errors for $p = 500, 1000, 2000$.

Since $\Omega$ always belongs to $\paraspp$ but not $\paraspq$, the estimator crop.Q is sub-optimal and thus expected to have an inferior performance. \prettyref{tab: simulation 2} shows this point, i.e., for fixed $p$ and $\alpha$, the advantage of crop.P is more obvious as $n$ increases. Especially, crop.P outperforms the other two estimators when $n=4000$. We also see a similar pattern as in \prettyref{tab: simulation 1} that $p$ has minor effect on the errors of all the estimators.

\begin{table}[]
	\centering
	\caption{The average errors under the operator norm of the banding estimator (BL) and the local cropping estimator (crop.Q) over 100 replications.}
	\label{tab: simulation 1}
	\begin{tabular}{cccccccccc}
		\hline
		\multirow{2}{*}{$p$}  & \multirow{2}{*}{$n$} & \multicolumn{2}{c}{$\alpha=0.5$} & \multicolumn{2}{c}{$\alpha=1$} & \multicolumn{2}{c}{$\alpha=1.5$} & \multicolumn{2}{c}{$\alpha=2$} \\ \cline{3-10} 
		&                      & crop.Q               & BL         & crop.Q              & BL        & crop.Q           & BL             & crop.Q          & BL            \\ \hline
		\multirow{4}{*}{500}  & 500                  & \textbf{4.68}       & 5.44       & \textbf{1.64}      & 2.38      & 1.18            & \textbf{1.16}  & 0.93           & \textbf{0.81} \\
		& 1000                 & \textbf{3.29}       & 4.89       & \textbf{1.17}      & 1.72      & \textbf{0.82}   & 1.08           & \textbf{0.66}  & 0.69          \\
		& 2000                 & \textbf{2.47}       & 4.45       & \textbf{0.89}      & 1.33      & \textbf{0.59}   & 0.69           & \textbf{0.48}  & 0.59          \\
		& 4000                 & \textbf{1.84}       & 3.80       & \textbf{0.62}      & 1.07      & \textbf{0.41}   & 0.64           & \textbf{0.34}  & 0.53          \\ \hline
		\multirow{4}{*}{1000} & 500                  & \textbf{4.96}       & 5.74       & \textbf{1.75}      & 2.40      & 1.30            & \textbf{1.19}  & 0.99           & \textbf{0.84} \\
		& 1000                 & \textbf{3.43}       & 5.19       & \textbf{1.24}      & 1.74      & \textbf{0.86}   & 1.10           & \textbf{0.68}  & 0.70          \\
		& 2000                 & \textbf{2.58}       & 4.75       & \textbf{0.93}      & 1.35      & \textbf{0.62}   & 0.71           & \textbf{0.51}  & 0.60          \\
		& 4000                 & \textbf{1.93}       & 4.10       & \textbf{0.66}      & 1.33      & \textbf{0.44}   & 0.65           & \textbf{0.36}  & 0.55          \\ \hline
		\multirow{4}{*}{2000} & 500                  & \textbf{5.14}       & 5.97       & \textbf{1.85}      & 2.41      & 1.33            & \textbf{1.21}  & 1.06           & \textbf{0.89} \\
		& 1000                 & \textbf{3.58}       & 5.41       & \textbf{1.30}      & 1.76      & \textbf{0.90}   & 1.12           & 0.72           & \textbf{0.71} \\
		& 2000                 & \textbf{2.69}       & 4.97       & \textbf{0.98}      & 1.37      & \textbf{0.65}   & 0.73           & \textbf{0.54}  & 0.62          \\
		& 4000                 & \textbf{2.01}       & 4.32       & \textbf{0.69}      & 1.34      & \textbf{0.45}   & 0.66           & \textbf{0.38}  & 0.55          \\ \hline
	\end{tabular}
\end{table}

\begin{table}[]
	\centering
	\caption{The average errors under the operator norm of the banding estimator (BL) and the local cropping estimators (crop.P \& crop.Q) over 100 replications.}
	\label{tab: simulation 2}
	\begin{tabular}{ccccccccccc}
		\hline
		\multirow{2}{*}{$p$}  & \multirow{2}{*}{$n$} & \multicolumn{3}{c}{$\alpha=1$}       & \multicolumn{3}{c}{$\alpha=1.5$} & \multicolumn{3}{c}{$\alpha=2$} \\ \cline{3-11} 
		&                      & crop.P         & crop.Q         & BL   & crop.P           & crop.Q  & BL    & crop.P          & crop.Q  & BL   \\ \hline
		\multirow{4}{*}{500}  & 500                  & 1.50          & \textbf{1.18} & 2.32 & \textbf{0.66}   & 0.73   & 0.86  & \textbf{0.52}  & 0.65   & 0.53 \\
		& 1000                 & 1.09          & \textbf{0.96} & 1.80 & \textbf{0.47}   & 0.56   & 0.83  & \textbf{0.38}  & 0.56   & 0.45 \\
		& 2000                 & 0.83          & \textbf{0.80} & 1.53 & \textbf{0.35}   & 0.43   & 0.55  & \textbf{0.27}  & 0.32   & 0.41 \\
		& 4000                 & \textbf{0.64} & 0.68          & 1.33 & \textbf{0.26}   & 0.35   & 0.54  & \textbf{0.19}  & 0.24   & 0.38 \\ \hline
		\multirow{4}{*}{1000} & 500                  & 1.50          & \textbf{1.20} & 2.36 & \textbf{0.68}   & 0.74   & 0.91  & \textbf{0.57}  & 0.68   & 0.59 \\
		& 1000                 & 1.12          & \textbf{0.98} & 1.82 & \textbf{0.49}   & 0.58   & 0.81  & \textbf{0.39}  & 0.55   & 0.46 \\
		& 2000                 & 0.84          & \textbf{0.81} & 1.54 & \textbf{0.37}   & 0.44   & 0.55  & \textbf{0.27}  & 0.32   & 0.41 \\
		& 4000                 & \textbf{0.65} & 0.68          & 1.52 & \textbf{0.26}   & 0.35   & 0.53  & \textbf{0.19}  & 0.24   & 0.38 \\ \hline
		\multirow{4}{*}{2000} & 500                  & 1.51          & \textbf{1.21} & 2.39 & \textbf{0.69}   & 0.75   & 0.96  & \textbf{0.62}  & 0.71   & 0.63 \\
		& 1000                 & 1.16          & \textbf{1.00} & 1.81 & \textbf{0.51}   & 0.60   & 0.84  & \textbf{0.39}  & 0.56   & 0.46 \\
		& 2000                 & 0.85          & \textbf{0.81} & 1.55 & \textbf{0.39}   & 0.44   & 0.56  & \textbf{0.27}  & 0.33   & 0.41 \\
		& 4000                 & \textbf{0.65} & 0.69          & 1.70 & \textbf{0.26}   & 0.35   & 0.53  & \textbf{0.19}  & 0.24   & 0.38 \\ \hline
	\end{tabular}
\end{table}

\begin{figure}[ht!]
	\centering
	
	\includegraphics[width=80mm]{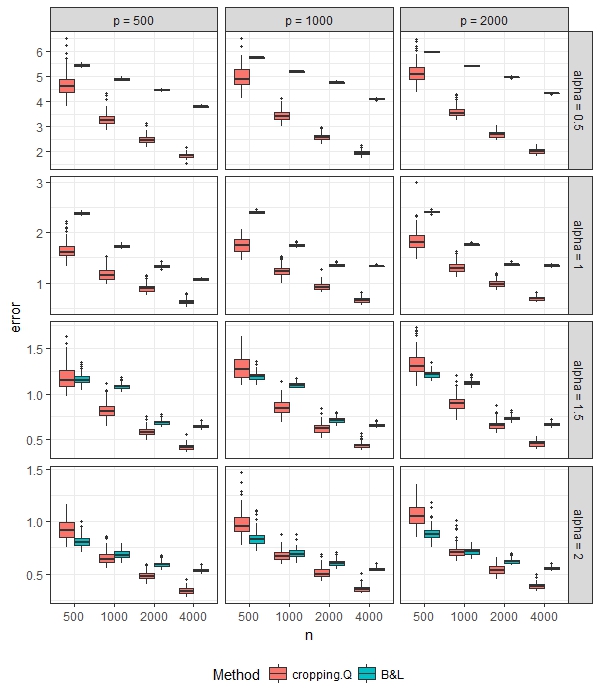}
	\caption{The boxplot of the errors from the local cropping estimator with the optimal bandwidth in $\paraspq$ (cropping.Q) and the banding estimator (BL) over 100 replications. }
	\label{fig: simulation 1}
\end{figure}

\begin{figure}[ht!]
	\centering
	
	\includegraphics[width=80mm]{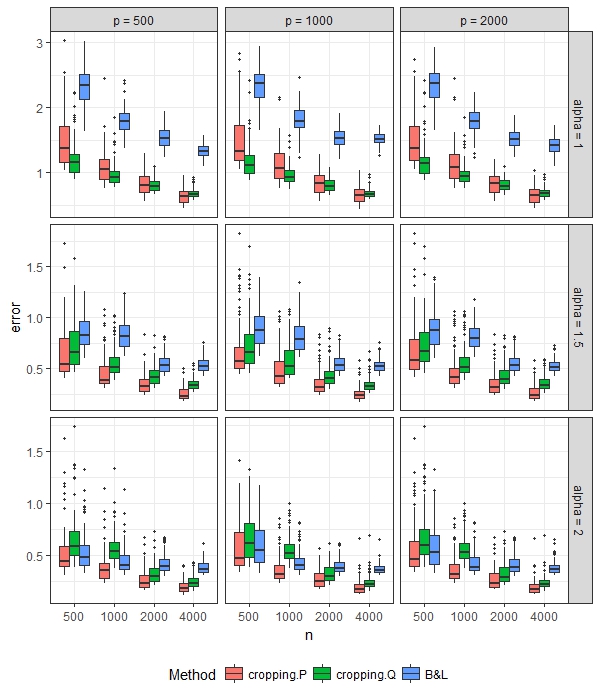}
	\caption{The boxplot of the errors from the local cropping estimator with the optimal bandwidth in $\paraspp$ (cropping.P), the local cropping estimator with the optimal bandwidth in $\paraspq$ (cropping.Q) and the banding estimator (BL) over 100 replications.}
	\label{fig: simulation 2}
\end{figure}

\bibliographystyle{plain}
\bibliography{test}

\newpage
\begin{center}
	{\Large Supplement to ``Minimax  Estimation of Large Precision Matrices with Bandable Cholesky Factor''}\\
	~\\
	Yu Liu and Zhao Ren\\
	~\\
	University of Pittsburgh
\end{center}


\setcounter{page}{1}

\appendix
\renewcommand{\theequation}{A.\arabic{equation}}
\setcounter{equation}{0}

\renewcommand{\thelemma}{A.\arabic{lemma}}
\setcounter{lemma}{0}

Theorem \ref{thm: main theorem op} immediately follows from Theorems \ref{thm: up op 1}, \ref{thm: 1 lower bound in op}, \ref{thm: eop upper 1} and \ref{thm: eop lower} while Theorem \ref{thm: main theorem f} immediately follows from Theorems \ref{thm: upper in f} and \ref{thm: lower in f}. We provide key lemmas in this section in the proofs of those main theorems. A discussion on the difference and connection between the parameter spaces considered in our paper and that in \cite{hu2017minimax} is also provided. In the end, we provide additional numerical studies.

Recall that for $1 \leq q 
\leq \infty$, the matrix $\ell_q$ norm of a matrix $S$ is defined by $\norm{S}_q = \max_{{\norm{x}_q}=1}\norm{Sx}_q$. We use the following two facts repeatedly in this section: (i) $\frac{\fnorm{S}}{\sqrt{p}} \leq \opnorm{S} \leq \sqrt{\colnorm{S}\rownorm{S}}$; (ii) for a real symmetric matrix $S$, $\opnorm{S} \leq \colnorm{S}$. 

\section{Discussion}

Various structures on precision matrices besides the structure we discussed in this paper, precision matrices with bandable Cholesky factor, have been proposed in recent years. During the finalizing process of this paper, a similar structure was discussed by \cite{hu2017minimax}, where the bandable structure is imposed on the entire precision matrix in contrast to the Cholesky factor. More specifically, a parameter space of bandable precision matrices was defined as follows,
\begin{align*}\label{eq p space l1}
\mathcal{F}_{\alpha}(\eta,M)=\Big\{
\Omega :\quad &0 < \eta^{-1} \leq \lambda_{\min}(\Omega) \leq \lambda_{\max}(\Omega) < \eta, \\
&\max_j \sum_{|i-j|\geq k} \omega_{ij} \leq Mk^{-\alpha}, \quad k \in [p]
\Big\}.
\end{align*}
In their paper, Hu and Negahban also established the minimax rate of convergence under the operator norm loss for this parameter space as follows, 
\begin{equation*}
\inf_{\tilde{\Omega}} \sup_{\mathcal{F}_\alpha(\eta, M) } \mathbb{E} \opnorm{\tilde{\Omega} - \Omega }^2 \asymp n^{-\frac{2\alpha}{2\alpha + 1}} + \frac{\log p}{n}.
\end{equation*}

Although the structure of the parameter space $\mathcal{F}_{\alpha}(\eta,M)$ above looks similar to that of $\mathcal{P}_\alpha(\eta, M)$ considered in our paper, there are fundamental differences in terms of interpretation as well as rates of convergence, which deserve a brief discussion. First of all, note that the interpretation of the entry in the precision matrix and the one in its Cholesky factor are different: $\omega_{ij}$ in $\Omega$ represents the partial covariance of $X_i$ and $X_j$ conditioned on the rest of variables, while $a_{ij}$ in $A$ represents the partial covariance of $X_i$ and $X_j$ conditioned on the variables whose index is smaller than $j$ (assuming $i < j$). As a result, it is more common to consider parameter spaces $\mathcal{P}_\alpha(\eta, M)$  when all variable have a natural order as is the case in auto-regressive processes while the parameter space $\mathcal{F}_{\alpha}(\eta,M)$ is often used to capture a sense of locality among all variables. In addition, the minimax rate over $\mathcal{P}_\alpha(\eta, M)$ under the operator norm loss is $n^{-\frac{2\alpha-1}{2\alpha}} + \frac{\log p}{n}$, which is distinguished with the one $n^{-\frac{2\alpha}{2\alpha + 1}} + \frac{\log p}{n}$ over $\mathcal{F}_\alpha(\eta, M)$. We provide an example to partially explain this difference. Let $A \equiv [a_{ij}]_{p \times p} $ with $a_{i1} = i^{-\alpha}$ for $i > 1$, otherwise $a_{ij} = 0$. One can easily verify that for all dimension $p$, $\Omega = (I-A)^T(I-A)$ belongs to $\mathcal{P}_\alpha(\eta, 1) $ with some constant $\eta>0$ but is not always in the space $\mathcal{F}_\alpha(\eta, M)$ with any fixed $M>0$. 

A more careful investigation on those two parameter spaces $\mathcal{F}_\alpha(\eta, M)$, $\mathcal{P}_\alpha(\eta, M)$, as well as the one $\mathcal{Q}_\alpha(\eta, M)$ considered in our paper reveals an interesting connection. On the one hand, we have that $\mathcal{Q}_\alpha(\eta, M) \subset \mathcal{F}_\alpha(\eta, CM)$. To see this,
we need some facts in the proof of \prettyref{lmm: tp omega close e}. According to \prettyref{eq: temp para sp 1} and \prettyref{eq: temp para sp 2} in \prettyref{lmm: tp omega close e}, for any $\Omega \in \mathcal{Q}_\alpha(\eta, M)$, we have $\norm{bd_k(\Omega) - \Omega}_{1} \leq \colnorm{U_h} + \rownorm{U_h} \leq CMk^{-\alpha}$ for any $k$, which immediately follows $\mathcal{Q}_\alpha(\eta, M) \subset \mathcal{F}_\alpha(\eta, CM)$, where $bd_h(\cdot)$ is defined in Equation \prettyref{eq: bd def}. On the other hand, we can show that $\mathcal{F}_\alpha(\eta, M) \subset \mathcal{P}_\alpha(\eta, M')$ when $\alpha > 1$ and $M$ is sufficiently small, which is summarized in the following lemma. 
\begin{lemma}\label{lem: para sp hu}
	Assume $\alpha > 1$ and $M$ is sufficiently small (depending on $\alpha$), we have $\mathcal{F}_\alpha(\eta, M) \subset \mathcal{P}_\alpha(\eta, M') $, where constant $M'$ depends on $M$ and $\eta$ only.
\end{lemma}

\begin{proof}
	To facilitate the proof, we restate the definition of $\mathcal{F}_\alpha(\eta, M)$ and $\mathcal{P}_\alpha(\eta, M)$. Define a general bandable matrix space $\mathcal{L}_{\alpha}(M)$,
	\begin{align*}
	\mathcal{L}_{\alpha}(M)=\Big\{
	X : \max_j \sum_{|i-j|\geq k} |x_{ij}| \leq Mk^{-\alpha}, \quad k \in [p]
	\Big\}.
	\end{align*}
	It is clear that $X \in \mathcal{F}_\alpha(\eta, M)$ is equivalent to that $0 < \eta^{-1} \leq \lambda_{\min}(X) \leq \lambda_{\max}(X) < \eta$ and $X \in \mathcal{L}_{\alpha}(M)$, while $X \in \mathcal{P}_\alpha(\eta, M)$ is equivalent to that $0 < \eta^{-1} \leq \lambda_{\min}(X) \leq \lambda_{\max}(X) < \eta$ and $A \in \mathcal{L}_{\alpha}(M)$, where $X = (I-A)^TD^{-1}(I-A)$. Recall that \prettyref{lmm: prop of paraspp} shows that $\eta^{-1} \leq \lambda_{\min}(X) \leq \lambda_{\max}(X) < \eta$ implies that $\eta^{-1} \leq \lambda_{\rm{min}} (D) \leq  \lambda_{\rm{max}} (D) \leq \eta$. Therefore, we have that $(I-A)^T(I-A) \in \mathcal{L}_{\alpha}(\eta M)$ whenever $X \in \mathcal{F}_\alpha(\eta, M)$. Then it suffices to prove that $(I-A)^T(I-A) \in \mathcal{L}_{\alpha}(M') \Rightarrow A \in \mathcal{L}_{\alpha}(M')$.
	
	We use induction to show above claim. Define $\textrm{R}_i(A)$ as the $p$ by $p$ matrix, which keeps the $i$-th row of the lower triangular matrix $A$ and sets the rest zero. Similarly, we define $\textrm{R}_{i:j}(A)$ as the $p$ by $p$ matrix, which keeps all rows between the $i$-th row and the $j$ th row (inclusive) of $A$. To prove $A \in \mathcal{L}_{\alpha}(M')$, it is enough to show for any $i$, $\textrm{R}_i(A) \in \mathcal{L}_{\alpha}(M')$.
	
	We first show that $\textrm{R}_p(A) \in \mathcal{L}_{\alpha}(M')$. Note that $A$ is a lower triangular matrix with zeros on the diagonal. Therefore, all entries in the last row of $A^TA$ are equal to zero, which means,
	$$
	\textrm{R}_p((I-A)^T(I-A)) = \textrm{R}_p(I) - \textrm{R}_p(A). 
	$$
	Since $(I-A)^T(I-A) \in \mathcal{L}_{\alpha}(M')$, we obtain that $\textrm{R}_p(A) \in \mathcal{L}_{\alpha}(M')$.
	
	Second, suppose $\textrm{R}_i(A) \in \mathcal{L}_{\alpha}(M')$ for all $i \in k+1:p$, we show that $\textrm{R}_k(A) \in \mathcal{L}_{\alpha}(M')$. We denote the first $k-1$ columns of a matrix $B$ by $\col_{1:k-1}(B)$. By simple algebra, we have
	$$
	\col_{1:k-1}\Big( \textrm{R}_k\big( (I-A)^T(I-A) \big) \Big) =  \col_{1:k-1}(G)-\col_{1:k-1}\Big( \textrm{R}_k\big( A \big) \Big),
	$$
	where $G =  \textrm{R}_{k}\big( (\textrm{R}_{k+1:p}( A ))^T \big) \textrm{R}_{k+1:p}\big( A \big)$. 
	Note that $\col_{1:k-1}\Big( \big( R_{1:k-1}(G^T) \big)^T \Big) = \col_{1:k-1}(G)$. Therefore, we have
	$$
	\col_{1:k-1}\Big( \textrm{R}_k\big( A \big) \Big) = \col_{1:k-1}\Big( \big( R_{1:k-1}(G^T) \big)^T \Big) -\col_{1:k-1}\Big( \textrm{R}_k\big( (I-A)^T(I-A) \big) \Big).
	$$
	By assumption, $\textrm{R}_k\big( (I-A)^T(I-A) \big)$ belongs to $\mathcal{L}_{\alpha}(M')$. We claim that $\big( R_{1:k-1}(G^T) \big)^T $ also belongs to $\mathcal{L}_{\alpha}(M')$. Combining above result with the fact that only the first $k-1$ columns of $\textrm{R}_k(A)$ are non-zero, we have proven that $\textrm{R}_k(A) \in \mathcal{L}_{\alpha}(M')$.
	
	By induction, we finish the proof that  $A \in \mathcal{L}_{\alpha}(M') $. 
	
	It remains to show that $\big( R_{1:k-1}(G^T) \big)^T  \in \mathcal{L}_{\alpha}(M')$. Recall that $G =  \textrm{R}_{k}\big( \textrm{R}_{k+1:p}( A )^T \big) \textrm{R}_{k+1:p}\big( A \big)$. To see this, assume  $\big( R_{1:k-1}(G^T) \big)^T \equiv [s_{ij}]_{p \times p}$ and $A \equiv [a_{ij}]_{p \times p}$. Note that $s_{ij} = 0$ when $i \neq k$ or $j \geq k$. For the rest of $k_{ij}$, one can verify that
	$$\sum_{j \leq (k-h)} |s_{kj}| \leq \sum_{i = 1}^{\infty} M'^2i^{-\alpha}(i+h)^{-\alpha} \leq \zeta(\alpha)M'^2h^{-\alpha}, $$
	where $\alpha > 1$, $\zeta(\alpha)$ is the sum of hyperharmonic series with power of $\alpha$. When $M' < \zeta(\alpha)^{-1}$, since the above inequality holds for any $0 < h < k$, It follows $\big( R_{1:k-1}(G^T) \big)^T  \in \mathcal{L}_{\alpha}(M')$.
\end{proof}

\renewcommand{\theequation}{B.\arabic{equation}}
\setcounter{equation}{0}

\renewcommand{\thelemma}{B.\arabic{lemma}}
\setcounter{lemma}{0}

\section{Preliminary Lemmas
}
In this section, we provide three preliminary lemmas and their proofs, which are important in the proofs of the other lemmas.

\subsection{}  
\begin{lemma}\label{lmm: projection}
	Operator $\projj(\cdot)$ is defined in \prettyref{eq: projection}. For any square matrix $A \equiv [a_{ij}]_{p \times p}$ such that $\eigenbd^{-1} \leq \lambda_{\min}(A) \leq \lambda_{\max}(A) \leq \eigenbd$,  we have,
	\begin{align}
	\opnorm{A - \mathbf{P}_\eigenbd(S)} &\leq  2\opnorm{A - S}, \label{eq: proj 1}\\
	\fnorm{A - \mathbf{P}_\eigenbd(S)} &\leq  2\fnorm{A - S}. \label{eq: proj 2}
	\end{align}
\end{lemma}
\begin{proof}
	Since  
	$$\norm{A - \projj(S)}_* \leq || S - \projj(S)||_* +  || A - S||_*,$$
	we only need to prove that 
	$$|| S - \projj(S)||_* \leq  || A - S||_*,$$ 
	where $||\cdot ||_*$ is either the operator norm or Frobenius norm.
	
	For an asymmetric matrix $S$, assume the singular value decomposition of $S$ is $U^TSV=W$. Since operator norm and Frobenius norm are invariant to the orthogonal transformation, it is sufficient  to prove
	$$|| U^TSV - U^T\projj(S)V||_* \leq  || U^TAV - U^TSV||_* .$$
	Here, $U^TSV$ and $U^T\projj(S)V$ are the diagonal matrices. $\eigenbd^{-1} \leq \lambda_{\min}(U^TAV) \leq \lambda_{\max}(U^TAV) \leq \eigenbd$. Without loss of generality, we only need to prove that for any diagonal matrix $W \equiv \diag(w)$, where $w=(w_1,\ldots,w_p)^T$,
	$$|| W - W_\eta||_* \leq  || A - W||_*,$$
	where $W_\eta = \diag(\max\{\min\{w_i, \eta\}, \eta^{-1}\})$ and $\eigenbd^{-1} \leq \lambda_{\min}(A) \leq \lambda_{\max}(A) \leq \eigenbd$. For the operator norm,
	note the fact that, 
	\begin{align*}
	\opnorm{A - W} &= \sup_{v \in \RR^p, \vecnorm{v}=1}\vecnorm{Av - Wv}\\
	&\geq \max_i \vecnorm{Ae_i - We_i}\\
	&\geq \max_i \max\big\{\vecnorm{Ae_i}-\vecnorm{We_i},  \vecnorm{We_i}-\vecnorm{Ae_i}\big\}\\
	&\geq \max_i \max\big\{\eta^{-1}-w_i, w_i-\eta,0\big\}\\
	&= \max_i\big\{|w_i - \max\{\min\{w_i, \eta\}, \eta^{-1}\}| \big\}\\
	&= \opnorm{W_\eta - W}.
	\end{align*}
	where $e_i$ denotes the vector all of whose components are zero, except the $i$-th component being one.
	For the Frobenius norm, we have 
	\begin{align*}
	\fnorm{A- W}^2 &= \sum_{i \neq j} a_{ij}^2 + \sum_i (a_{ii}-w_i)^2 
	\\
	&= \sum_i (\sum_j a_{ij}^2 + w_i^2 - 2a_{ii}w_i)\\
	&\geq \sum_i (\sum_j a_{ij}^2 + w_i^2 - 2\sqrt{\sum_j a_{ij}^2}w_i)\\
	&= \sum_i (w_i - \sqrt{\sum_j a_{ij}^2})^2.
	\end{align*}
	Since $\eigenbd^{-1} \leq \lambda_{\min}(A) \leq \lambda_{\max}(A) \leq \eigenbd$, we obtain $\eigenbd^{-1} \leq \sqrt{\sum_j a_{ij}^2}  \leq \eigenbd$. It is easy to check that for any $i$, $ (\sqrt{\sum_j a_{ij}^2}-w_i)^2 \geq (\max\{\min\{w_i, \eta\}, \eta^{-1}\}-w_i)^2$. Then we have
	$$\fnorm{A-W}^2 \geq \fnorm{W_\eta-W}^2,$$
	that yields $\fnorm{A-W} \geq \fnorm{W_\eta-W}$.
	
	The same result can be derived easily in the case of symmetric matrix $S$ using similar argument on the eigen-decomposition. Thus, we finish the proof.
\end{proof}

\subsection{}
\begin{lemma} \label{lmm: prop of paraspp}
	Assume precision matrix $\Omega$ satisfies that $\eigenbd^{-1} \leq \lambda_{\rm{min}} (\Omega) \leq  \lambda_{\rm{max}} (\Omega) \leq \eigenbd$. Let $\Sigma =  \Omega ^{-1}$ and $\Omega =(I-A)^TD^{-1}(I-A)$ (see \prettyref{eq: def cd of omega}), then,
	\begin{align}
	\eigenbd^{-1} \leq \lambda_{\rm{min}} (\Sigma) &\leq  \lambda_{\rm{max}} (\Sigma) \leq \eigenbd,\\
	\eigenbd^{-1} \leq \lambda_{\rm{min}} (D) &\leq  \lambda_{\rm{max}} (D) \leq \eigenbd,\\
	\eigenbd^{-1} \leq \lambda_{\rm{min}} (I-A) &\leq  \lambda_{\rm{max}} (I-A) \leq \eigenbd.
	\end{align}
\end{lemma}
\begin{proof}
	Since $\Sigma=\Omega^{-1}$, it is trivial that $\eigenbd^{-1} \leq \lambda_{\rm{min}} (\Sigma) \leq  \lambda_{\rm{max}} (\Sigma) \leq \eigenbd$.
	For the symmetric positive-definite matrices $\Omega \equiv [\omega_{ij}]_{p \times p}$ and $\Sigma \equiv [\sigma_{ij}]_{p \times p}$, $\omega_{ii} \leq \lambda_{\max}(\Omega) \leq \eigenbd;  \sigma_{ii} \leq \lambda_{\max}(\Sigma) \leq \eigenbd$ for any $i \in [p]$.
	$\omega_{ii}^{-1}$ can be interpreted as the variance of the residual of the regression of $X_i$ on $\mx_{-i}=(X_1\dots X_{i-1}, X_{i+1} \dots X_p)^T$. $\sigma_{ii}$ can be interpreted as the variance of the residual of the regression of $X_i$ on $\textbf{0}$. Therefore, the $i$-th entry of $D$, the variance of the residual of the regression of $X_i$ on $\mx_{1:i-1}=(X_1\dots X_{i-1})^T$, has the bound $\omega_{ii}^{-1} \leq d_i \leq \sigma_{ii}$, then $\eigenbd^{-1} \leq \lambda_{\rm{min}} (D) \leq  \lambda_{\rm{max}} (D) \leq \eigenbd$.\\
	
	For any unit vector $u, v$ such that $(I-A)u=\lambda v$, $\lambda > 0$, we have 
	\begin{equation*}
	u^T\Omega u = u^T(I-A)^TD^{-1}(I-A)u = \lambda^2 v^TD^{-1}v,
	\end{equation*}
	we obtain that $\lambda^2 = (u^T\Omega u)/(v^TD^{-1}v) \in [\eigenbd^{-2}, \eigenbd^2]$, that yields
	\begin{equation*}
	\eigenbd^{-1} \leq \lambda_{\min}(I-A) \leq \lambda_{\max}(I-A) \leq \eigenbd. \qedhere
	\end{equation*}
\end{proof}

\subsection{}
\begin{lemma}\label{lmm: regression relation}
	Assume $\mx$ is an $i$-variate random vector with covariance matrix $\Sigma$ such that $\eigenbd^{-1} \leq \lambda_{\rm{min}} (\Sigma) \leq  \lambda_{\rm{max}} (\Sigma) \leq \eigenbd$.
	Let the linear projection of $X_i$ onto $\mx_{i-k:i-1}$ in population be $\hat{X}^{\langle k \rangle}_i$. With the corresponding coefficients padded with $i-k-1$ zeros in the front, we can rewrite it as
	$ \hat{X}^{\langle k \rangle}_i=\mx_{1:i-1}^T\boldsymbol{\beta}_i^{\langle k \rangle}$, where $\boldsymbol{\beta}_i^{\langle k \rangle} \in \RR^{i-1}$ with its first $i-k-1$ coordinates being zeros. In addition, set $\epsilon_i^{\langle k \rangle}=\var(X_i-\hat{X}^{\langle k \rangle}_i)$.
	
	(i) Whenever $\boldsymbol{\beta}_i^{\langle i-1 \rangle} = \big(\beta_{i1}^{\langle i-1 \rangle},\dots,\beta_{i(i-1)}^{\langle i-1 \rangle}\big)^T$ satisfies 
	\begin{equation} \label{eq: reg condition 1}
	\sum_{j<i-k}|\beta_{ij}^{\langle i-1 \rangle}|<Mk^{-\alpha}, \quad k \in [i-1],
	\end{equation}
	we have, 
	\begin{align}
	\vecnorm{\boldsymbol{\beta}_i^{\langle i-1 \rangle}-\boldsymbol{\beta}_i^{\langle k \rangle}} &\leq 2\eigenbd^2 M k^{-\alpha}, \label{eq: lmm regression relation 1}\\
	\abs{\epsilon_i^{\langle i-1 \rangle}-\epsilon_i^{\langle k \rangle}} &\leq 4\eigenbd^4 M k^{-\alpha}. \label{eq: lmm regression relation 2}
	\end{align}
	(ii) Whenever $\boldsymbol{\beta}_i^{\langle i-1 \rangle} = \big(\beta_{i1}^{\langle i-1 \rangle},\dots,\beta_{i(i-1)}^{\langle i-1 \rangle}\big)^T$ satisfies 
	\begin{equation} \label{eq: reg condition 2}
	|\beta_{ij}^{\langle i-1 \rangle}|<M(i-j)^{-\alpha-1}, \quad k \in  [i-1],
	\end{equation}
	it holds that, 
	\begin{align}
	\vecnorm{\boldsymbol{\beta}_i^{\langle i-1 \rangle}-\boldsymbol{\beta}_i^{\langle k \rangle}} &\leq 2\eigenbd^2 M (k-1)^{-\alpha-\hf}, \label{eq: lmm regression relation e 1}\\
	\abs{\epsilon_i^{\langle i-1 \rangle}-\epsilon_i^{\langle k \rangle}} &\leq 4\eigenbd^4 M (k-1)^{-\alpha-\hf}. \label{eq: lmm regression relation e 2}
	\end{align}
\end{lemma}
\begin{proof}
	For any fixed $i$ and $k$, let $\boldsymbol{\beta}_i^{\langle i-1 \rangle} = ((\boldsymbol{\beta}^{\langle i-1 \rangle}_A)^T, (\boldsymbol{\beta}^{\langle i-1 \rangle}_B)^T)^T$ and $\boldsymbol{\beta}_i^{\langle k \rangle}= ((\boldsymbol{\beta}^{\langle k \rangle}_A)^T, (\boldsymbol{\beta}^{\langle k \rangle}_B)^T)^T$, where the sizes of $\boldsymbol{\beta}^{\langle i-1 \rangle}_A$ and $\boldsymbol{\beta}^{\langle k \rangle}_A$ are $i-k-1$, the sizes of $\boldsymbol{\beta}^{\langle i-1 \rangle}_B$ and $\boldsymbol{\beta}^{\langle k \rangle}_B$ are $k$. Note that  $\boldsymbol{\beta}^{\langle k \rangle}_A = \textbf{0}$.
	To facilitate the proof, we divide $\Sigma$ into several block matrices,
	\begin{equation*}
	\Sigma=
	\begin{bmatrix}
	\Sigma_{11} & \Sigma_{12}& \Sigma_{13}\\
	\Sigma_{21} & \Sigma_{22}& \Sigma_{23}\\
	\Sigma_{31} & \Sigma_{32}& \Sigma_{33}\\
	\end{bmatrix},
	\end{equation*}
	where $\Sigma_{11}, \Sigma_{22}, \Sigma_{33}$ are the covariance matrices of $\mx_{1:i-k-1}, \mx_{i-k:i-1}$ and $X_i$ respectively.
	
	By the definition of the linear projection, we have
	\begin{align*}
	\Sigma_{23}&=\Sigma_{21}\boldsymbol{\beta}^{\langle i-1 \rangle}_A + \Sigma_{22}\boldsymbol{\beta}^{\langle i-1 \rangle}_B, \\
	\epsilon^{\langle i-1 \rangle}&=(-(\boldsymbol{\beta}_i^{\langle i-1 \rangle})^T,1)\Sigma(-(\boldsymbol{\beta}_i^{\langle i-1 \rangle})^T,1)^T,\\
	\Sigma_{23}&= \Sigma_{22}\boldsymbol{\beta}^{\langle k \rangle}_B,\\
	\epsilon^{\langle k \rangle}&=(-(\boldsymbol{\beta}_i^{\langle k \rangle})^T,1)\Sigma(-(\boldsymbol{\beta}_i^{\langle k \rangle})^T,1)^T.
	\end{align*}
	Condition in \prettyref{eq: reg condition 1} implies 
	\begin{equation*}
	\vecnorm{\boldsymbol{\beta}^{\langle i-1 \rangle}_A} \leq \colnorm{\boldsymbol{\beta}^{\langle i-1 \rangle}_A} \leq Mk^{-\alpha}.
	\end{equation*}
	Then we have,
	\begin{align*}
	\vecnorm{\boldsymbol{\beta}_i^{\langle i-1 \rangle}-\boldsymbol{\beta}_i^{\langle k \rangle}} 
	&\leq \vecnorm{\boldsymbol{\beta}^{\langle i-1 \rangle}_A-\boldsymbol{0}}+\vecnorm{\boldsymbol{\beta}^{\langle i-1 \rangle}_B-\boldsymbol{\beta}^{\langle k \rangle}_B}\\
	&\leq \vecnorm{\boldsymbol{\beta}^{\langle i-1 \rangle}_A}+\vecnorm{\Sigma_{22}^{-1}\Sigma_{21}\boldsymbol{\beta}^{\langle i-1 \rangle}_A}\\
	&\leq Mk^{-\alpha} +\eigenbd^2 Mk^{-\alpha} \\
	&\leq 2\eigenbd^2 Mk^{-\alpha}.
	\end{align*}
	that yields the desired \prettyref{eq: lmm regression relation 1}.
	
	Assume the decomposition of $\Sigma^{-1}$ is $  (I-A)^TD^{-1}(I-A)$. Note that $(-(\boldsymbol{\beta}_i^{\langle i-1 \rangle})^T,1)$ corresponds to the  $i$th row of the lower triangle of $I-A$. According to \prettyref{lmm: projection}, $\opnorm{I-A} \leq \eigenbd$, then $\vecnorm{(-(\boldsymbol{\beta}_i^{\langle i-1 \rangle})^T,1)^T} \leq \opnorm{I-A} \leq \eigenbd$. Applying the same argument on the covariance matrix of $\mx_{i-k:i}$, we have $\vecnorm{(-(\boldsymbol{\beta}_i^{\langle k \rangle})^T,1)^T} \leq  \eigenbd$. Moreover, 
	\begin{align*}
	&\abs{\epsilon_i^{\langle i-1 \rangle}-\epsilon_i^{\langle k \rangle}} \\
	\leq& \abs{(-(\boldsymbol{\beta}_i^{\langle i-1 \rangle})^T,1)\Sigma(-(\boldsymbol{\beta}_i^{\langle i-1 \rangle})^T,1)^T - (-(\boldsymbol{\beta}_i^{\langle k \rangle})^T,1)\Sigma(-(\boldsymbol{\beta}_i^{\langle k \rangle})^T,1)^T } \\
	\leq& \big(\vecnorm{(-(\boldsymbol{\beta}_i^{\langle i-1 \rangle})^T,1)^T} + \vecnorm{(-(\boldsymbol{\beta}_i^{\langle k\rangle})^T,1)^T}\big) \opnorm{\Sigma} \vecnorm{\boldsymbol{\beta}_i^{\langle i-1 \rangle}-\boldsymbol{\beta}_i^{\langle k \rangle}}\\
	\leq& 4\eigenbd^4M k^{-\alpha},
	\end{align*}
	that yields the desired \prettyref{eq: lmm regression relation 2}. 
	
	Similarly, condition in \prettyref{eq: reg condition 2} implies 
	\begin{equation*}
	\vecnorm{\boldsymbol{\beta}^{\langle i-1 \rangle}_A} \leq \colnorm{\boldsymbol{\beta}^{\langle i-1 \rangle}_A} \leq M(k-1)^{-\alpha-\hf}.
	\end{equation*}
	Then we have
	\begin{align*}
	\vecnorm{\boldsymbol{\beta}_i^{\langle i-1 \rangle}-\boldsymbol{\beta}_i^{\langle k \rangle}} &\leq 2\eigenbd^2 M(k-1)^{-\alpha-\hf},\\
	\abs{\epsilon_i^{\langle i-1 \rangle}-\epsilon_i^{\langle k \rangle}}
	&\leq 4\eigenbd^4M (k-1)^{-\alpha-\hf},
	\end{align*}
	which completes the proofs of \prettyref{eq: lmm regression relation e 1} and \prettyref{eq: lmm regression relation e 2}.
\end{proof}

\renewcommand{\theequation}{C.\arabic{equation}}
\setcounter{equation}{0}
\renewcommand{\thelemma}{C.\arabic{lemma}}
\setcounter{lemma}{0}

\section{Proofs of \prettyref{lmm: tp omega close} and \prettyref{lmm: bias in block up lop} in analysis of \prettyref{thm: up op 1} }
In this section, we prove \prettyref{lmm: tp omega close} and \prettyref{lmm: bias in block up lop} to establish \prettyref{thm: up op 1}.
\subsection{Proof of \prettyref{lmm: tp omega close}}
For $\Omega \equiv [\omega_{ij}]_{p \times p}$, define 
\begin{equation}\label{eq: bd def}
bd_k(\Omega) \equiv [\omega_{ij}\indc{|i-j|\leq k}]_{p \times p}.
\end{equation}
It is easy to check
$\Omega_k^* = \frac{1}{k}\sum_{i=k}^{2k-1} bd_i (\Omega).$ Then we have
\begin{equation*}
\opnorm{\Omega - \Omega_k^*} \leq \frac{1}{k}\sum_{i=k}^{2k-1} (\opnorm{\Omega - bd_i (\Omega)} ).
\end{equation*}

We turn to the analysis of $\opnorm{\Omega - bd_k (\Omega)}.$
Define $D^{-\frac{1}{2}}(I-A) $ as $ B \equiv [b_{ij}]$. We know $\Omega = B^TB $, and 
\begin{equation*}
\max_i \sum_{j<i-k} |b_{ij}| \leq M\eigenbd^{\frac{1}{2}} k^{-\alpha}.
\end{equation*}
Set $bd_k(B) $ as $B_k$, and $B-bd_k(B)$ as $B_{-k}$. Then we can rewrite $\Omega$ as
\begin{align*}
\Omega &= B^TB \\
&= (B_k +B_{-k})^T(B_k +B_{-k})\\
&= B_k^TB_k + B_{-k}^TB_k + B_k^TB_{-k} + B_{-k}^TB_{-k}.
\end{align*}
Checking the entries of $B_k^TB_k$, we find that
$
bd_k\big( B_k^TB_k \big) = B_k^TB_k.
$ Thus,
\begin{align} 
&\opnorm{\Omega - bd_k (\Omega)} \label{eq: lemma 23 temp 1}\\
=& \opnorm{B^TB - bd_k(B^TB)} \nonumber \\
=&\opnorm{B_{-k}^TB_k + B_k^TB_{-k} + B_{-k}^TB_{-k} - bd_k(B_{-k}^TB_k + B_k^TB_{-k} + B_{-k}^TB_{-k})}\nonumber\\
\leq & \opnorm{B_{-k}^TB_k + B_k^TB_{-k} + B_{-k}^TB_{-k}} + \opnorm{bd_k(B_{-k}^TB_k + B_k^TB_{-k} + B_{-k}^TB_{-k})}\nonumber\\
=& \opnorm{B_{-k}^TB + B^TB_{-k} - B_{-k}^TB_{-k}} + \opnorm{bd_k(B_{-k}^TB_k + B_k^TB_{-k} + B_{-k}^TB_{-k})}\nonumber\\
\leq & \opnorm{B_{-k}^TB} + \opnorm{B^TB_{-k}} + \opnorm{B_{-k}^TB_{-k}}\nonumber\\
&+ \opnorm{bd_k(B_{-k}^TB_k)} + \opnorm{bd_k(B_k^TB_{-k})} + \opnorm{bd_k(B_{-k}^TB_{-k})}\nonumber.
\end{align}

The key is to control $\opnorm{B_{-k}}$. In addition, in order to get rid off the operator $bd_k(\cdot)$ when handling the term $\opnorm{bd_k(B_{-k}^TB_{-k})}$, we adopt a technique which requires controlling the operator norm of the matrix in which each entry is the absolute value of the entry in $B_{-k}$. To this end, we state the key lemma below.

\begin{lemma}\label{lem:key bias}
	Assume that matrix $B_{-k}$ is defined above. We have
	\begin{equation*}
	\opnorm{B_{-k}^+} \leq Ck^{\frac{1}{2}-\alpha },
	\end{equation*}
	where $X^+$ is the matrix in which each entry is the magnitude of the corresponding entry in $X$.
\end{lemma}
\begin{proof}
	Assume $F_i$ is the matrix composed by the $(2^{i-1}k+1)$-th to $(2^i k)$-th sub-diagonals in matrix $B^+$, $i \in [\ceil{\log_2 (p/k)}]$. Note that $B_{-k}^+ = \sum_{i=1}^{\ceil{\log_2 (p/k)}}(F_i)$.
	For any $i \geq 1$, 
	$
	\rownorm{F_i} \leq M\eigenbd^{\frac{1}{2}}(2^{i-1}k)^{-\alpha}.
	$
	Since there are at most $2^{i-1}k$ entries in each column of $F_i$, 
	$
	\colnorm{F_i} \leq M\eigenbd^{\frac{1}{2}}(2^{i-1}k)^{1-\alpha}.
	$
	The operator norm of $F_i$ can be bounded by the two terms above,
	\begin{align*}
	\opnorm{F_i} &\leq (\rownorm{F_i}\colnorm{F_i} )^{\frac{1}{2}} \\
	&\leq M\eigenbd^{\frac{1}{2}}(2^{i-1}k)^{\frac{1}{2}-\alpha} \\
	&= M\eigenbd^{\frac{1}{2}}k^{\frac{1}{2}-\alpha} \times (2^{\frac{1}{2}-\alpha})^{i-1}.
	\end{align*}
	Consequently,
	\begin{align*}
	\opnorm{B_{-k}^+ } &\leq \sum_{i=1}^{\ceil{\log_2 (p/k)}} \opnorm{F_i} \\
	&\leq \sum_{i=1}^{\ceil{\log_2 (p/k)}}( M\eigenbd^{\frac{1}{2}}k^{\frac{1}{2}-\alpha} \times (2^{\frac{1}{2}-\alpha})^{i-1}) \\
	&= M\eigenbd^{\frac{1}{2}}k^{\frac{1}{2}-\alpha} \times \sum_{i=1}^{\ceil{\log_2 (p/k)}}(2^{\frac{1}{2}-\alpha})^{i-1} \\
	&\leq CM\eigenbd^{\frac{1}{2}}k^{\frac{1}{2}-\alpha}. \qedhere
	\end{align*}
\end{proof}
The bounds of the operator norms of many other matrices can be derived from the above result combining the following lemma.

\begin{lemma}\label{lmm: positive}
	Let $X$ be a $square$ matrix, then 
	\begin{equation*}
	\opnorm{X} \leq \opnorm{X^+},
	\end{equation*}
	and
	\begin{equation*}
	\opnorm{bd_k(X)} \leq \opnorm{X^+},
	\end{equation*}
	where $X^+$ is the matrix in which each entry is the magnitude of the corresponding entry in $X$.
\end{lemma}
\begin{proof}
	Assume $X \equiv [x_{ij}]_{p \times p}$, let $\opnorm{X} = u^TXv$, then
	\begin{equation*}
	\opnorm{X} = u^TXv = \sum_{i,j}u_i x_{ij} v_j \leq \sum_{i,j}|u_i||x_{ij}||v_j| = (u^+)^TX^+v^+ \leq \opnorm{X^+}.
	\end{equation*}
	Assume $bd_k(X) = [x_{ij}^*]_{p \times p}$, $\opnorm{bd_k(X)} = u^T(bd_k(X))v$, then 
	\begin{equation*}
	\opnorm{bd_k(X)} = u^T(bd_k(X))v = \sum_{i,j}u_i x^*_{ij} v_j \leq \sum_{i,j}|u_i||x_{ij}||v_j| \leq \opnorm{X^+}. \qedhere
	\end{equation*}
\end{proof}

It follows from the previous two lemmas that $
\opnorm{B_{-k}}  \leq \opnorm{B_{-k}^+} \leq Ck^{\frac{1}{2}-\alpha }.$
Then we have
\begin{align} \label{eq: lemma 23 temp 2}
\begin{split}
\opnorm{B^TB_{-k}} &\leq \opnorm{B}\opnorm{B_{-k}}  \leq Ck^{\frac{1}{2}-\alpha },\\
\opnorm{B_{-k}^TB} &\leq \opnorm{B}\opnorm{B_{-k}}  \leq Ck^{\frac{1}{2}-\alpha },\\
\opnorm{B_{-k}^TB_{-k}} &\leq  \opnorm{B_{-k}}\opnorm{B_{-k}}  \leq Ck^{1-2\alpha }.
\end{split}
\end{align}
In addition, Lemma \ref{lmm: positive} implies that, 
\begin{equation} \label{eq: lemma 23 temp 3}
\opnorm{bd_k({B_{-k}^TB_{-k}})} \leq \opnorm{(B_{-k}^TB_{-k})^+} \leq \opnorm{(B_{-k}^+)^TB_{-k}^+} \leq Ck^{1-2\alpha }.
\end{equation}

Then we turn to bound $\opnorm{bd_k(B_{-k}^TB_k)} + \opnorm{bd_k(B_k^TB_{-k})}$. We control $\rownorm{bd_k(B_{-k}^TB_k)}$ and $\colnorm{bd_k(B_{-k}^TB_k)}$ first. For $h \in [p]$, we have
\begin{align*}
\rownorm{bd_k(B_{-k}^TB_k)} 
\leq& \max_h \colnorm{\row_h(bd_k(B_{-k}^TB))}\\
\leq& \sum_{i=h+1}^{h+k}(\sum_{j=h+k+1}^{h+2k}|b_{jh}||b_{ij}|)\\
=& \sum_{j=h+k+1}^{h+2k} (\sum_{i=h+1}^{h+k}|b_{jh}||b_{ij}|  )\\
\leq& \sum_{j=h+k+1}^{h+2k} M^2\eigenbd((j-h)^{-\alpha}(j-h-k)^{-\alpha})\\
=& \sum_{j=1}^{k} M^2\eigenbd((j+k)^{-\alpha}j^{-\alpha})\\
\leq& M^2\eigenbd k^{-\alpha}\sum_{j=1}^{k}j^{-\alpha}\\
\leq& Ck^{1-2\alpha},
\end{align*}
and
\begin{align*}
\colnorm{bd_k(B_{-k}^TB_k)} &\leq \max_h \colnorm{\col_h(bd_k(B_{-k}^TB))}\\
&\leq \sum_{i=h-k}^{h-1}(\sum_{j=h+1}^{h+k}|b_{ji}||b_{jh}|)\\
&\leq \sum_{j=h+1}^{h+k}( \sum_{i=h-k}^{h-1}|b_{ji}||b_{jh}| )\\
&\leq \sum_{j=h+1}^{h+k} M^2\eigenbd (k^{-\alpha}(j-h)^{-\alpha}  )\\
&\leq \sum_{j=1}^{k} M^2\eigenbd (k^{-\alpha} j^{-\alpha}  )\\
&\leq Ck^{1-2\alpha}.
\end{align*}
Therefore, we obtain that
\begin{align} \label{eq: lemma 23 temp 4}
\begin{split}
\opnorm{bd_k(B_{-k}^TB_k)} \leq (\rownorm{bd_k(B_{-k}^TB_k)}\colnorm{bd_k(B_{-k}^TB_k)} )^{1/2} &\leq Ck^{1-2\alpha},\\
\opnorm{bd_k(B_k^TB_{-k})} \leq (\rownorm{bd_k(B_{k}^TB_{-k})}\colnorm{bd_k(B_{k}^TB_{-k})} )^{1/2} &\leq Ck^{1-2\alpha}.
\end{split}
\end{align}

Combining \prettyref{eq: lemma 23 temp 1}, \prettyref{eq: lemma 23 temp 2}, \prettyref{eq: lemma 23 temp 3} and \prettyref{eq: lemma 23 temp 4}, we prove that
\begin{equation*}
\opnorm{\Omega - bd_k (\Omega)} \leq Ck^{\frac{1}{2}-\alpha},
\end{equation*}
which follows
$\opnormsq{\Omega - \Omega_k^*} \leq Ck^{1-2\alpha}$.

\subsection{Proof of \prettyref{lmm: bias in block up lop}}
\begin{figure}
	\centering
	\begin{tikzpicture}[scale = 0.7]
	\node at( -4.4 , -1.2 ) {$\mathbf{C}_m^k(\Omega)$} ;
	\node at( -2 , -1.2 ) {$=$} ;
	\draw[black, fill=lightgray]( -0.8 , 0.8 ) -- ( 0.8 , -0.8 ) -- ( 0.8 , 0.8 );
	\draw[black]( -0.8 , 0.8 ) rectangle ( 0.8 , -0.8 );
	\draw [fill=cyan ]( 0.8 , 0.1 ) -- ( -0.1 , 0.1 )-- ( 0.1 , -0.1 )-- ( 0.8 , -0.1 );
	\node at( 0 , -1.2 ) {$G^T$} ;
	\node at( 2.15 , -1.2 ) {$\times$} ;
	\draw[black]( 3.5 , 0.8 ) -- ( 5.1 , -0.8 );
	\draw[ultra thick, cyan]( 4.2 , 0.1 ) -- ( 5.1 , -0.8 );
	\draw[black]( 3.5 , 0.8 ) rectangle ( 5.1 , -0.8 );
	\node at( 4.3 , -1.2 ) {$D^{-1}$} ;
	\node at( 6.45 , -1.2 ) {$\times$} ;
	\draw[black, fill=lightgray]( 7.8 , 0.8 ) -- ( 9.4 , -0.8 ) -- ( 7.8 , -0.8 );
	\draw[black]( 7.8 , 0.8 ) rectangle ( 9.4 , -0.8 );
	\draw [fill=cyan ]( 8.5 , 0.1 ) -- ( 8.7 , -0.1 )-- ( 8.7 , -0.8 )-- ( 8.5 , -0.8 );
	\node at( 8.6 , -1.2 ) {$G$} ;
	\node at( -4.4 , -3.7 ) {$\mathbf{C}_m^k(\Omega)^*$} ;
	\node at( -2 , -3.7 ) {$=$} ;
	\draw[black, fill=lightgray]( -0.8 , -1.7 ) -- ( 0.8 , -3.3 ) -- ( 0.8 , -1.7 );
	\draw[black]( -0.8 , -1.7 ) rectangle ( 0.8 , -3.3 );
	\draw [fill=cyan ]( 0.3 , -2.4 ) -- ( -0.1 , -2.4 )-- ( 0.1 , -2.6 )-- ( 0.3 , -2.6 );
	\node at( 0 , -3.7 ) {$G^TR_p^{m,2k}(R_p^{m,2k})^T$} ;
	\node at( 2.15 , -3.7 ) {$\times$} ;
	\draw[black]( 3.5 , -1.7 ) -- ( 5.1 , -3.3 );
	\draw[ultra thick, cyan]( 4.2 , -2.4 ) -- ( 5.1 , -3.3 );
	\draw[black]( 3.5 , -1.7 ) rectangle ( 5.1 , -3.3 );
	\node at( 4.3 , -3.7 ) {$D^{-1}$} ;
	\node at( 6.45 , -3.7 ) {$\times$} ;
	\draw[black, fill=lightgray]( 7.8 , -1.7 ) -- ( 9.4 , -3.3 ) -- ( 7.8 , -3.3 );
	\draw[black]( 7.8 , -1.7 ) rectangle ( 9.4 , -3.3 );
	\draw [fill=cyan ]( 8.5 , -2.4 ) -- ( 8.7 , -2.6 )-- ( 8.7 , -2.8 )-- ( 8.5 , -2.8 );
	\node at( 8.6 , -3.7 ) {$R_p^{m,2k}(R_p^{m,2k})^TG$} ;
	\node at( -4.4 , -6.2 ) {$\mathbf{C}_m^k(\Omega)^*$} ;
	\node at( -2 , -6.2 ) {$=$} ;
	\draw[black, fill=lightgray]( -0.8 , -4.2 ) -- ( 0.8 , -5.8 ) -- ( 0.8 , -4.2 );
	\draw[black]( -0.8 , -4.2 ) rectangle ( 0.8 , -5.8 );
	\draw [fill=cyan ]( 0.3 , -4.9 ) -- ( -0.1 , -4.9 )-- ( 0.1 , -5.1 )-- ( 0.3 , -5.1 );
	\node at( 0 , -6.2 ) {$H^T$} ;
	\node at( 2.15 , -6.2 ) {$\times$} ;
	\draw[black]( 3.5 , -4.2 ) -- ( 5.1 , -5.8 );
	\draw[ultra thick, cyan]( 4.2 , -4.9 ) -- ( 4.6 , -5.3 );
	\draw[black]( 3.5 , -4.2 ) rectangle ( 5.1 , -5.8 );
	\node at( 4.3 , -6.2 ) {$\mathbf{C}_m^{2k}(D^{-1})$} ;
	\node at( 6.45 , -6.2 ) {$\times$} ;
	\draw[black, fill=lightgray]( 7.8 , -4.2 ) -- ( 9.4 , -5.8 ) -- ( 7.8 , -5.8 );
	\draw[black]( 7.8 , -4.2 ) rectangle ( 9.4 , -5.8 );
	\draw [fill=cyan ]( 8.5 , -4.9 ) -- ( 8.7 , -5.1 )-- ( 8.7 , -5.3 )-- ( 8.5 , -5.3 );
	\node at( 8.6 , -6.2 ) {$H$} ;
	\node at( -4.4 , -8.7 ) {$\mathbf{C}_{k+1}^k((\mathbf{C}_{m-k}^{3k}(\Omega^{-1}))^{-1})$} ;
	\node at( -2 , -8.7 ) {$=$} ;
	\draw[lightgray]( -0.8 , -6.7 ) rectangle ( 0.8 , -8.3 );
	\draw[black, fill=yellow]( -0.3 , -7.2 ) -- ( 0.3 , -7.8 ) -- ( 0.3 , -7.2 );
	\draw[black]( -0.3 , -7.2 ) rectangle ( 0.3 , -7.8 );
	\draw [fill=red ]( 0.3 , -7.4 ) -- ( -0.1 , -7.4 )-- ( 0.1 , -7.6 )-- ( 0.3 , -7.6 );
	\node at( 0 , -8.7 ) {$K^T$} ;
	\node at( 2.15 , -8.7 ) {$\times$} ;
	\draw[lightgray]( 3.5 , -6.7 ) rectangle ( 5.1 , -8.3 );
	\draw[black]( 4 , -7.2 ) -- ( 4.6 , -7.8 );
	\draw[ultra thick, red]( 4.2 , -7.4 ) -- ( 4.6 , -7.8 );
	\draw[black]( 4 , -7.2 ) rectangle ( 4.6 , -7.8 );
	\node at( 4.3 , -8.7 ) {$\mathbf{C}_{k+1}^{2k}(E^{-1})$} ;
	\node at( 6.45 , -8.7 ) {$\times$} ;
	\draw[lightgray]( 7.8 , -6.7 ) rectangle ( 9.4 , -8.3 );
	\draw[black, fill=yellow]( 8.3 , -7.2 ) -- ( 8.9 , -7.8 ) -- ( 8.3 , -7.8 );
	\draw[black]( 8.3 , -7.2 ) rectangle ( 8.9 , -7.8 );
	\draw [fill=red ]( 8.5 , -7.4 ) -- ( 8.7 , -7.6 )-- ( 8.7 , -7.8 )-- ( 8.5 , -7.8 );
	\node at( 8.6 , -8.7 ) {$K$} ;
	\end{tikzpicture}
	\caption{An illustration of the proof strategy in \prettyref{lmm: bias in block up lop}.}
	\label{fig:lemma proof}
\end{figure}
The proof strategy in this lemma is not complicated, although the notation is quite involved. Our target is to bound the distance between $\cut{k}{m}(\Omega)$ and $\cut{k}{k+1}\big( (\cut{3k}{m-k}( \Omega^{-1} ))^{-1}  \big)$ under the operator norm. To this end, we introduce an intermediate term $\cut{k}{m}(\Omega)^*$ to facilitate our proof. Specifically, we break the target into two terms $\opnormsq{\cut{k}{m}(\Omega) -\cut{k}{m}(\Omega)^*}$ and $\opnormsq{\cut{k}{m}(\Omega)^* -\cut{k}{k+1}\big( (\cut{3k}{m-k}( \Omega^{-1} ))^{-1}  \big)}$ and derive their bounds respectively. The construction of $\cut{k}{m}(\Omega)^*$ with corresponding Cholesky decomposition are illustrated in \prettyref{fig:lemma proof}, in contrast with those of $\cut{k}{m}(\Omega)$ and $\cut{k}{k+1}\big( (\cut{3k}{m-k}( \Omega^{-1} ))^{-1}  \big)$.

To express the decomposition of $\cut{k}{m}(\Omega)$ in the matrix form, we define the $p\times k$ matrix 
\begin{equation*} \label{eq: matrix selector}
R_{p}^{m,k} \equiv [r_{ij}]_{p\times k}, \quad r_{ij}=\indc{i-m=j-1}.
\end{equation*}
Assume $\Omega=(I-A)^TD^{-1}(I-A)$. Set $(I-A) R_{p}^{m,k}$ as $G$. One can check
\begin{align*}
\cut{k}{m}(\Omega) &= (R_{p}^{m,k})^T \Omega R_{p}^{m,k}  \\
&= (R_{p}^{m,k})^T (I-A)^TD^{-1}(I-A) R_{p}^{m,k} \\
&= G^T  D^{-1} G.
\end{align*}
Define 
\begin{align*}
\cut{k}{m}(\Omega)^* &= (R_{p}^{m,k})^T (I-A)^T R_{p}^{m,2k} (R_{p}^{m,2k})^T D^{-1} R_{p}^{m,2k} (R_{p}^{m,2k})^T (I-A) R_{p}^{m,k}\\
&= G^T  R_{p}^{m,2k} (R_{p}^{m,2k})^T D^{-1} R_{p}^{m,2k} (R_{p}^{m,2k})^T G . 
\end{align*}

We first bound $\opnormsq{\cut{k}{m}(\Omega) - \cut{k}{m}(\Omega)^*}$. Since $I-A$ is a lower triangular matrix, $R_{p}^{m,2k} (R_{p}^{m,2k})^T G$ consists of the first $2k$ columns of $G$. Then we have
\begin{equation*}
G - R_{p}^{m,2k} (R_{p}^{m,2k})^T G = (0,0,\dots,0,g_{m+2k},\dots , g_p)^T,
\end{equation*}
where $g_i = \row_i(G)$, and for $ i \in (m+2k) : p$,
\begin{equation}\label{eq: lop upper lmm temp 1}
\vecnorm{g_i} \leq \colnorm{g_i} \leq M(i-m-k+1)^{-\alpha}.
\end{equation}
Consequently,
\begin{align*}
&\opnormsq{G - R_{p}^{m,2k} (R_{p}^{m,2k})^T G} \\ 
\leq& \vecnormsq{(\vecnorm{g_{m+2k}}, \dots, \vecnorm{g_p})^T } \\
\leq& \vecnormsq{(M(k+1)^{-\alpha}, \dots, M(p-m-k)^{-\alpha})^T }\\
\leq& M^2 k^{-2\alpha+1}.
\end{align*}
Then we have
\begin{align}
&\opnormsq{\cut{k}{m}(\Omega) - \cut{k}{m}(\Omega)^*} \label{eq: phi phi prime}\\
\leq& \opnormsq{ (G^T D^{-1} G -G^T R_{p}^{m,2k} (R_{p}^{m,2k})^T D^{-1} R_{p}^{m,2k} (R_{p}^{m,2k})^T G} \nonumber\\ 
\leq& \opnormsq{G - R_{p}^{m,2k} (R_{p}^{m,2k})^T G} \opnormsq{D^{-1}} (\opnormsq{G} + \opnormsq{R_{p}^{m,2k} (R_{p}^{m,2k})^TG}) \nonumber\\
\leq& 2\opnormsq{D^{-1}}\opnormsq{G} \opnormsq{G - R_{p}^{m,2k} (R_{p}^{m,2k})^T G}  \nonumber\\
\leq& 2\eigenbd^2  M^2 k^{-2\alpha+1}. \nonumber
\end{align}

Next, we turn to derive the bound of $\opnormsq{\cut{k}{m}(\Omega)^* -\cut{k}{k+1}\big( (\cut{3k}{m-k}( \Omega^{-1} ))^{-1}  \big)}$. 
Assume that $(\cut{3k}{m-k}( \Omega^{-1} ))^{-1} = (I-B)^TE^{-1}(I-B)$. One can also check
\begin{align*}
&\cut{k}{k+1}\big( (\cut{3k}{m-k}( \Omega^{-1}) )^{-1}  \big)\\
=&(R_{3k}^{k+1,k})^T (\cut{3k}{m-k}( \Omega^{-1} ))^{-1} R_{3k}^{k+1,k}  \\
=&(R_{3k}^{k+1,k})^T(I-B)^TE^{-1}(I-B) R_{3k}^{k+1,k} \\
=&(R_{3k}^{k+1,k})^T(I-B)^T R_{3k}^{k+1,2k}(R_{3k}^{k+1,2k})^T E^{-1} R_{3k}^{k+1,2k}(R_{3k}^{k+1,2k})^T (I-B) R_{3k}^{k+1,k}.
\end{align*}
To ease our notation a little bit, one can check that $(R_{p}^{m,2k})^T D R_{p}^{m,2k} = \cut{2k}{m}(D)$, and that $(R_{3k}^{k+1,2k})^T E R_{3k}^{k+1,2k}= \cut{2k}{k+1}(E)$. In addition, we set $(R_{p}^{m,2k})^T (I-A) R_{p}^{m,k}$ as $H$, $(R_{3k}^{k+1,2k})^T (I-B) R_{3k}^{k+1,k}$ as $K$. Then we can rewrite $\cut{k}{m}(\Omega)^*=H^T \cut{2k}{m}(D^{-1}) H$ and $\opnormsq{\cut{k}{m}(\Omega) - \cut{k}{m}(\Omega)^*}=K^T \cut{2k}{k+1}(E^{-1})K$. We bound $\opnormsq{H-K}$ and $\opnormsq{\cut{2k}{m}(D^{-1}) -  \cut{2k}{k+1}(E^{-1})}$ separately below.

Referring to the instruction in \prettyref{fig:lemma proof}, one can check 
$\row_i(H ) $ is part of the coefficients of the regression $X_{m+i} \sim \mx_{1:m+i-1}$ and  $\row_i(K)$ is part of the coefficients of the regression $X_{m+i} \sim \mx_{m-k:m+i-1}$. According to \prettyref{lmm: regression relation},  $\row_i(H-K)$ can be bounded by $2M\eigenbd^2 (k+i)^{-\alpha}$. The dimension of the non-zero part of $H-K$ is $2k \times k$. So we have 
\begin{equation*}
\opnormsq{H-K} \leq 8\eigenbd^4M^2 k^{-2\alpha+1}.
\end{equation*}
Similarly, the $i$-th elements of $\cut{2k}{m}(D)$ and $ \cut{2k}{k+1}(E)$ are the residuals of the above two regressions. According to \prettyref{lmm: regression relation}, $\opnormsq{\cut{2k}{m}(D) -  \cut{2k}{k+1}(E)} \leq  16\eigenbd^8 M^2 k^{-2\alpha} $. Therefore, we have
\begin{align*}
&\opnormsq{\cut{2k}{m}(D^{-1}) -  \cut{2k}{k+1}(E^{-1})}\\
\leq& \opnormsq{\cut{2k}{m}(D^{-1})} \opnormsq{\cut{2k}{m}(D) -  \cut{2k}{k+1}(E)}  \opnormsq{\cut{2k}{k+1}(E^{-1})}\\
\leq& 16M^2\eigenbd^{12} k^{-2\alpha}.
\end{align*}
Combing the above two results, we have 
\begin{equation} \label{eq: phi prime psi}
\begin{split}
&\opnormsq{\cut{k}{m}(\Omega)^* - \cut{k}{k+1}\big( (\cut{3k}{m-k}( \Omega^{-1} ))^{-1}  \big) }  \\
\leq& \opnormsq{H^T \cut{2k}{m}(D^{-1}) H - K^T \cut{2k}{k+1}(E^{-1})K } \\
\leq& \opnormsq{H^T} \opnormsq{\cut{2k}{m}(D^{-1}) - \cut{2k}{k+1}(E^{-1})} \opnormsq{H}\\
& + (\opnormsq{H} +\opnormsq{K})\opnormsq{\cut{2k}{k+1}(E^{-1})} \opnormsq{H-K} \\
\leq&\eigenbd^4 \times 16M^2\eigenbd^{12} k^{-2\alpha} +  6\eigenbd^4\times 8\eigenbd^4M^2 k^{-2\alpha+1} \\
\leq& 96M^2\eigenbd^{16} k^{-2\alpha+1} .
\end{split}
\end{equation}

In the end, based on the Equations \prettyref{eq: phi phi prime} and \prettyref{eq: phi prime psi}, we have
\begin{align*}
&\opnormsq{\cut{k}{m}(\Omega) - \cut{k}{k+1}\big( (\cut{3k}{m-k}( \Omega^{-1} ))^{-1}  \big) }  \\
\leq &2\opnormsq{\cut{k}{m}(\Omega) -\cut{k}{m}(\Omega)^* } + 2\opnormsq{\cut{k}{m}(\Omega)^* - \cut{k}{k+1}\big( (\cut{3k}{m-k}( \Omega^{-1} ))^{-1}  \big) } \\
\leq &2\times (2\eigenbd^2  M^2 k^{-2\alpha+1} + 96M^2\eigenbd^{16} k^{-2\alpha+1})\\
\leq &200 M^2\eigenbd^{16} k^{-2\alpha+1}.
\end{align*}
We finish the proof of \prettyref{lmm: bias in block up lop}.

\renewcommand{\theequation}{D.\arabic{equation}}
\setcounter{equation}{0}
\renewcommand{\thelemma}{D.\arabic{lemma}}
\setcounter{lemma}{0}

\section{Proofs of \prettyref{lmm: p1 p2 subset}, \prettyref{lmm: 1 assouad}, \prettyref{lmm: 2 assouad}, and \prettyref{lmm: le cam 1} in the analysis of \prettyref{thm: 1 lower bound in op}}
In this section, we prove \prettyref{lmm: p1 p2 subset}, \prettyref{lmm: 1 assouad}, \prettyref{lmm: 2 assouad}, and \prettyref{lmm: le cam 1} to establish \prettyref{thm: 1 lower bound in op}.
\subsection{Proof of \prettyref{lmm: p1 p2 subset}}
First we prove that $\calp_1 \in \paraspp$. Let $A_k^*(\theta) \equiv [a_{ij}]_{k \times k}$. We know the exact value of each entry. It is easy to check $\sum_{i-j>k}|a_{ij}| \leq M k^{-\alpha}$. One can check $\Omega(\theta) \in \calp_1$ has the specific form of 
\begin{equation*} \label{eq: def sig in p11}
\Omega(\theta)=\begin{bmatrix}
I_k+(A^{*}_k(\theta))^T A^*_k(\theta)  & -(A^{*}_k(\theta))^T & \textbf{0}_{k \times (p-2k)}\\
- A^*_k(\theta) & I_k & \textbf{0}_{k \times (p-2k)}\\
\textbf{0}_{(p-2k) \times k} & \textbf{0}_{(p-2k) \times k} & I_{p-2k}\\
\end{bmatrix}.
\end{equation*}

Let $\mathbf{1}$ denote the vector with all $1$'s in $\Theta$.  One can check that
\begin{align*}
\lambda_{\rm{max}}(\Omega(\theta)) =& \big(\lambda_{\rm{max}}(I - A(\theta))\big)^2 \leq \big(\lambda_{\rm{max}}(I + A(\theta))\big)^2 {\leq}(1+\lambda_{\rm{max}}(A_k^*(\mathbf{1})))^2 \\=& (1 + k^{\hf}n^{-\hf} \tau)^2 \leq \eigenbd.
\end{align*}
The second inequality above is due to that the entries of $A(\theta)$ are all non-negative and \prettyref{lmm: positive}. Recall $\Sigma(\theta) = (I+A(\theta))(I+A(\theta))^T$. Thus, we can check
\begin{align*}
\lambda_{\rm{min}}(\Omega(\theta)) =& \big(\lambda_{\rm{max}}(\Sigma(\theta))\big)^{-1}= \big(\lambda_{\rm{max}}(I + A(\theta))\big)^{-2} {\geq}(1+\lambda_{\rm{max}}(A_k^*(\mathbf{1})))^{-2} \\
=& (1 + k^{\hf}n^{-\hf} \tau)^{-2} \geq  \eigenbd^{-1}. 
\end{align*}
The eigenvalues of $\Omega(\theta)$ are in the interval $[\eigenbd^{-1}, \eigenbd]$. So $\calp_1 \in \paraspp$.

Then we turn to prove that $\calp_2 \in \paraspp$. The Cholesky factor $A$ of $\Omega(m)$ is the zero matrix. The minimum eigenvalue of $\Omega(m)$ is $(1 + \tau a^{\hf})^{-1}$, which is greater than $\eigenbd^{-1}$ and maximum one is $1$, which is less than $\eigenbd$. So $\calp_2 \in \paraspp$.

\subsection{Proof of \prettyref{lmm: 1 assouad}}
Since $\Vert {P}_\theta \wedge {P}_{\theta'} \Vert = 1 - \frac{1}{2} \Vert {P}_\theta - {P}_{\theta'} \Vert_1$, we turn to control 
$\max_{H(\theta,\theta')= 1}\Vert {P}_\theta - {P}_{\theta'} \Vert_1$. The following First Pinsker's inequality will facilitate our analysis. 
\begin{lemma}[First Pinsker's Inequality \citep{csiszar1967information}]\label{lem: pinsker}
	\begin{align*}
	\Vert {P}_\theta - {P}_{\theta'} \Vert_1^2 &\leq \frac{1}{2} K(P_{\theta'} \vert P_\theta) \\
	&= \frac{n}{2} [\frac{1}{2}tr (\Sigma(\theta')\Sigma(\theta)^{-1})-\frac{1}{2}\log\det(\Sigma(\theta')\Sigma(\theta)^{-1})-\frac{p}{2}] 
	\end{align*}
	where $K(\cdot \vert \cdot)$ is the Kullback-Leibler divergence.
\end{lemma}
One can check $\Sigma(\theta) \in \calp_{1}$ has the form  of
\begin{align*}
\Sigma(\theta)&= (I+A(\theta))(I+A(\theta)^T)\\
&=\begin{bmatrix}
I_k &  (A^*_k(\theta))^T & 0_{k \times (p-2k)}\\
( A^*_k(\theta)) & I_k+(A^*_k(\theta))( A^*_k(\theta))^T & 0_{k \times (p-2k)}\\
0_{(p-2k) \times k} & 0_{(p-2k) \times k} & I_{p-2k}.\\
\end{bmatrix}.
\end{align*}
Set $\Sigma(\theta') = D + \Sigma(\theta)$ and $d=\theta'-\theta$. Note that whenever $\max H(\theta,\theta')= 1$, $D$ has the form of
\begin{align*}
D&=(A(\theta')-A(\theta))+(A(\theta')-A(\theta))^T+(A(\theta')A(\theta')^T-A(\theta)A(\theta)^T) \nonumber \\
&=A(d)+A(d)^T+A(\theta')A(d)^T + A(d)A(\theta)^T,
\end{align*}
where $A(d)$ is similarly defined as $A(\theta)$ except that $\theta_i$ is replaced by $d_i$. Let $\mathbf{1}$ denote the all-ones vector, one can check 
\begin{align} 
\opnorm{D} \leq 2(1+\opnorm{A_k^*(\mathbf{1})})\opnorm{A_k^*(d)}\leq 4\tau n^{-\hf},\label{eq: lmm low op 1} \\
\fnorm{D} \leq 2(1+\fnorm{A_k^*(\mathbf{1})})\fnorm{A_k^*(d)}\leq 4\tau n^{-\hf}.
\label{eq: lmm low op 2}
\end{align}
Furthermore, we have $\opnorm{D\Sigma(\theta)^{-1}}\leq \eta  \times 4\tau n^{-\hf} \leq n^{-\hf}$, $\fnorm{D\Sigma(\theta)^{-1}} \leq \opnorm{\Sigma(\theta)^{-1}} \fnorm{D}\leq \eta   4\tau n^{-\hf}  \leq n^{-\hf}$.
One can easily check that $\Sigma(\theta')\Sigma(\theta)^{-1} = I + D\Sigma(\theta)^{-1}$,
$\frac{1}{2}tr (I +D\Sigma(\theta)^{-1}) = \frac{p}{2} +\frac{1}{2} tr(D\Sigma(\theta)^{-1})$, $
\log \det (I +D\Sigma(\theta)^{-1}) = tr(D\Sigma(\theta)^{-1}) +\sum_i (\log(1+\lambda_i)-\lambda_i)
$, where $\lambda_i$'s are the eigenvalues of $D\Sigma(\theta)^{-1}$.
Applying the First Pinsker's inequality in this case, we have
\begin{align*}
\Vert {P}_\theta - {P}_{\theta'} \Vert_1^2 &\leq \frac{n}{2} [\frac{1}{2}tr (\Sigma(\theta')\Sigma(\theta)^{-1})-\frac{1}{2}\log\det(\Sigma(\theta')\Sigma(\theta)^{-1})-\frac{p}{2}] \\
&= \frac{n}{2} [\frac{1}{2}tr (I + D\Sigma(\theta)^{-1})-\frac{1}{2}\log\det(I + D\Sigma(\theta)^{-1})-\frac{p}{2}]\\
&= \frac{n}{2} \sum_i (\lambda_i - \log(1+\lambda_i)).
\end{align*}

Since $|\lambda_i| \leq \opnorm{D\Sigma(\theta)^{-1}}$, all the $\lambda_i$ are bounded by $\pm n^{-\hf}$. By Taylor expansion, $  \sum_i (\lambda_i - \log(1+\lambda_i))\leq 2\sum_i \lambda^2_i=2\fnormsq{D\Sigma(\theta)^{-1}}$,
we have 
\begin{equation*}
\Vert {P}_\theta - {P}_{\theta'} \Vert_1^2   \leq n \fnormsq{D\Sigma(\theta)^{-1}} \leq 1,
\end{equation*}
which implies $\Vert {P}_\theta \wedge {P}_{\theta'} \Vert \geq 0.5$. 

\subsection{Proof of \prettyref{lmm: 2 assouad}}
The proof is as follows,
\begin{align*}
&\opnorm{\Omega(\theta')-\Omega(\theta)} \\
\geq& \sup_{\norm{u}_2=1,\norm{v}_2=1} (u^T,\textbf{0}^T,\textbf{0}^T)(\Omega(\theta')-\Omega(\theta))(\textbf{0},v,\textbf{0})^T \\
\geq & \sup_{\norm{u}_2=1,\norm{v}_2=1} u^T(A^*_k(\theta)-A^*_k(\theta'))^Tv\\
=& \opnorm{(A^*_k(\theta')-A^*_k(\theta))} \\
=& (H(\theta',\theta))^{1/2}\lopelem,
\end{align*}
which immediately implies,
\begin{equation*}
\min_{H(\theta,\theta') \geq 1}\frac{\| \Omega(\theta)-\Omega(\theta')\|_2^2}{H(\theta,\theta')} \geq (\lopelem)^2.
\end{equation*}

\subsection{Proof of \prettyref{lmm: le cam 1}}
Denote the density functions of $P_i$ and $\bar{P}$ by $f_i$ and $\bar{f}$ respectively, where $0\leq i\leq p$. It is sufficient to bound $\int \frac{\bar{f}^2}{f_0}du -1$ because
$\norm{P_0 \wedge \bar{P}} \geq 1-\frac{1}{2}(\int \frac{\bar{f}^2}{f_0}du -1)^{\frac{1}{2}}$.
To this end, note that
\begin{align*}
\int \frac{\bar{f}^2}{f_0}du -1 &= \frac{1}{p^2}(\sum_{k\in[p]}\int \frac{f_k^2}{f_0}du + \sum_{1\leq i \neq j\leq p} \int \frac{f_if_j}{f_0}du)-1\\
&= \frac{1}{p^2}(p(1-\tau^2kn^{-1})^{-\frac{n}{2}} + p^2 - p)  -1 \\
&\leq \frac{1}{2p}\tau^2k\leq \frac{1}{16},
\end{align*}
which further implies
$\norm{P_0 \wedge \bar{P}} \geq 1-\frac{1}{2}\times \frac{1}{4} = \frac{7}{8}$.

\renewcommand{\theequation}{E.\arabic{equation}}
\setcounter{equation}{0}
\renewcommand{\thelemma}{E.\arabic{lemma}}
\setcounter{lemma}{0}

\section{Proofs of \prettyref{lmm: tp omega close e} and \prettyref{lmm: bias in block up lop e} in the analysis of \prettyref{thm: eop upper 1}}
In this section, we prove \prettyref{lmm: tp omega close e} and \prettyref{lmm: bias in block up lop e} to establish \prettyref{thm: eop upper 1}.
\subsection{Proof of \prettyref{lmm: tp omega close e}}  
Recall $\Omega_k^* = \frac{1}{k}\sum_{h=k}^{2k-1}bd_h(\Omega)$, where $bd_h(\cdot)$ is defined in Equation \prettyref{eq: bd def}. Then,
\begin{equation*}
\opnorm{\Omega_k^* - \Omega} \leq \frac{1}{k}\sum_{h=k}^{2k-1} \opnorm{bd_h(\Omega) - \Omega}.
\end{equation*}
For any fix $k\leq h<2k$, define $U_h \equiv [\omega_{ij}\indc{i-j>h} ]_{p \times p}$, then $bd_h(\Omega) - \Omega = U_h + U_h^T$.
Note that $\Omega=(I-A)^TD^{-1}(I-A)$, where $I-A \equiv [a'_{ij}]_{p \times p}$.  We have
\begin{align} \label{eq: temp para sp 1}
\rownorm{U_h} &\leq \max_i\colnorm{\textrm{row}_i(U_h)}\leq  \max_i \eigenbd \sum_{s=i}^p \abs{a'_{si}} \sum_{j=1}^{i-k} \abs{a'_{sj}} \nonumber \\
&\leq \max_i \eigenbd M \sum_{s=i}^p (s-i)^{-\alpha-1} (s-i+k)^{-\alpha} \nonumber \\
&\leq C\eigenbd M k^{-\alpha}.
\end{align}
Similarly, we have
\begin{align}  \label{eq: temp para sp 2}
\colnorm{U_h} &\leq \max_j\colnorm{\textrm{col}_j(U_h)}\leq \max_j \eigenbd M \sum_{s=j+k}^p \abs{a'_{sj}}  \nonumber \\
&\leq \max_j \eigenbd M \sum_{s=j+k}^p (s-j)^{-\alpha-1} \nonumber \\
&\leq C\eigenbd M k^{-\alpha}. 
\end{align} 
With the bounds for $\colnorm{U_h}$ and $\rownorm{U_h}$, we have
\begin{align*}
\opnorm{\Omega_k^* - \Omega} &\leq \frac{1}{k}\sum_{i=k}^{2k-1} \opnorm{bd_i(\Omega) - \Omega}  \\ 
&\leq \max_{k\leq h<2k}\opnorm{bd_h(\Omega) - \Omega} \\
&\leq \max_{k\leq h<2k}(\opnorm{U_h} + \opnorm{U_h^T})\\
&\leq \max_{k\leq h<2k}2\sqrt{\rownorm{U_h} \colnorm{U_h}}\\
&\leq C\eigenbd M k^{-\alpha}.
\end{align*}

\subsection{Proof of \prettyref{lmm: bias in block up lop e}}  
The proof of \prettyref{lmm: bias in block up lop e} is basically the same as the one of \prettyref{lmm: bias in block up lop}, except for a few steps, which are highlighted below.

Since $\Omega \in \paraspq$, \prettyref{eq: lop upper lmm temp 1} should be updated by
\begin{equation*} \label{eq: eop upper lmm temp 1}
\vecnorm{g_i} \leq M(i-m-k+1)^{-\alpha-1/2},
\end{equation*}
and consequently \prettyref{eq: phi phi prime} should be replaced by 
\begin{equation*}\label{eq: temp 1}
\opnormsq{\cut{k}{m}(\Omega)-\cut{k}{m}(\Omega)^*} \leq 2M^2k^{-2\alpha}.
\end{equation*}
According to \prettyref{lmm: regression relation}, \prettyref{eq: phi prime psi} is replaced by 
\begin{equation*}\label{eq: temp 2}
\opnormsq{\cut{k}{m}(\Omega)^*-\cut{k}{k+1}((\cut{3k}{m-k}(\Omega^{-1}))^{-1})} \leq 96M^2 \eigenbd^{16}k^{-2\alpha}.
\end{equation*}
Therefore, we have
\begin{equation*}
\opnormsq{\cut{k}{m}(\Omega)-\cut{k}{k+1}((\cut{3k}{m-k}(\Omega^{-1}))^{-1})} \leq 200 M^2 \eigenbd^{16}k^{-2\alpha}.
\end{equation*}

\renewcommand{\theequation}{F.\arabic{equation}}
\setcounter{equation}{0}
\renewcommand{\thelemma}{F.\arabic{lemma}}
\setcounter{lemma}{0}

\section{Proofs of \prettyref{lmm: p3 subset} and \prettyref{lmm: e assouad} in the analysis of \prettyref{thm: eop lower}}
In this section, we prove \prettyref{lmm: p3 subset} and \prettyref{lmm: e assouad} to establish \prettyref{thm: eop lower}.
\subsection{Proof of \prettyref{lmm: p3 subset}}
Let $B_k^*(\theta) \equiv [b_{ij}]_{k \times k}$, then we have the specific value of each entry. It is easy to check $b_{ij} \leq M(i-j)^{-\alpha-1}$. One can check $\Omega(\theta) \in \calp_3$ has the specific form of 
\begin{equation} \label{eq: form of sig in p11}
\Omega(\theta)=\begin{bmatrix}
I_k+B^{*}_k(\theta)^T B^*_k(\theta)  & -B^{*}_k(\theta)^T & 0_{k \times (p-2k)}\\
-B^*_k(\theta) & I_k & 0_{k \times (p-2k)}\\
0_{(p-2k) \times k} & 0_{(p-2k) \times k} & I_{p-2k}\\
\end{bmatrix}.
\end{equation}

Let $\mathbf{1}$ denote the vector with all $1$'s in $\Theta$. Then one can check that
\begin{align*}
\lambda_{\rm{max}}(\Omega(\theta)) =& \big(\lambda_{\rm{max}}(I - B(\theta))\big)^2 \leq \big(\lambda_{\rm{max}}(I + B(\theta))\big)^2 {\leq}(1+\lambda_{\rm{max}}(B_k^*(\mathbf{1})))^2 \\=& (1 + k^{\hf}n^{-\hf} \tau)^2 \leq \eigenbd.
\end{align*}
The second inequality above follows from that the entries of $A(\theta)$ are all non-negative and \prettyref{lmm: positive}. Recall $\Sigma(\theta) = (I+B(\theta))(I+B(\theta))^T$. Therefore, we have
\begin{align*}
\lambda_{\rm{min}}(\Omega(\theta)) =& \big(\lambda_{\rm{max}}(\Sigma(\theta))\big)^{-1}= \big(\lambda_{\rm{max}}(I + B(\theta))\big)^{-2} {\geq}(1+\lambda_{\rm{max}}(B_k^*(\mathbf{1})))^{-2} \\
=& (1 + k^{\hf}n^{-\hf} \tau)^{-2} \geq  \eigenbd^{-1}. 
\end{align*}
The eigenvalues of $\Omega(\theta)$ are in the interval $[\eigenbd^{-1}, \eigenbd]$. Therefore, $\calp_3 \in \paraspp$.

\subsection{Proof of \prettyref{lmm: e assouad}}  
The proof follows most part of the one of \prettyref{lmm: 1 assouad}. Some inequalities in the proof of \prettyref{lmm: 1 assouad} need to be rechecked.
Assume that $\Sigma(\theta)=\{\Omega(\theta)^{-1}: \Omega(\theta) \in \calp_{3} \}$ and $D=\Sigma(\theta')-\Sigma(\theta)$, one can verify that $D$ has the same decomposition as that in the proof of \prettyref{lmm: 1 assouad} except that $A_k^*$ is replaced by $B_k^*$. We can show
\begin{equation*} \label{eq: lemma lower lop eop 1}
\opnorm{D\Sigma(\theta)^{-1}} \leq  \eigenbd \rownorm{D} \leq \eigenbd(\eopelem k + (\eopelem)^2k^2) \leq \frac{1}{2},
\end{equation*}
\begin{equation*} \label{eq: lemma lower lop eop 2}
\fnormsq{D\Sigma(\theta)^{-1}} \leq \eigenbd^2(2k(\eopelem)^2+(2k-1)((\eopelem)^2 k)^2)\leq 4\eigenbd^2(\lopelem)^2.
\end{equation*}
Thus, \prettyref{eq: lmm e 1} still holds in this case. 
As for \prettyref{eq: lmm e 2}, similarly,
\begin{align*}
\opnormsq{\Omega(\theta')-\Omega(\theta)} \geq \opnormsq{B^*_k(\theta')-B^*_k(\theta)} &=H(\theta,\theta')(\eopelem)^2k=H(\theta,\theta')(\lopelem)^2, \\
\min_{H(\theta,\theta') \geq 1}\frac{\| \Omega'(\theta)-\Omega'(\theta')\|_2^2}{H(\theta,\theta')} &\geq (\lopelem)^2.
\end{align*}
Plugging the results of \prettyref{eq: lmm e 1} and \prettyref{eq: lmm e 2} into \prettyref{lmm: assouad}, we have 
$$ \inf_{\tilde{\Omega}}\sup_{\calp_{3}}\mathbb{E}\opnormsq{\tilde{\Omega}-\Omega} \geq \frac{\tau^2}{16} n^{-1}k= \frac{\tau^2}{16} n^{-1}\min\{n^{\frac{1}{2\alpha+1}}, \frac{p}{2} \}. $$

\renewcommand{\theequation}{G.\arabic{equation}}
\setcounter{equation}{0}
\renewcommand{\thelemma}{G.\arabic{lemma}}
\setcounter{lemma}{0}

\section{Proof of \prettyref{lmm: lf upper thresholding} in the analysis of Theorem \ref{thm: upper in f}}
In this section, we prove \prettyref{lmm: lf upper thresholding} to establish Theorem \ref{thm: upper in f}.
\subsection{Proof of \prettyref{lmm: lf upper thresholding}} 
Let the linear projection of $X_i$ onto  the linear span of $\mx_{i-k:i-1}$ in population be $\hat{X}^{\langle k \rangle}_i$. With the corresponding coefficients padded with $i-k-1$ zeros in the front, we can rewrite it as
$ \hat{X}_i^{\langle k \rangle}=\mx_{1,i-1}^T\mathbf{a}_i^{\langle k \rangle}$, where $\boldsymbol{a}_i^{\langle k \rangle} \in \RR^{i-1}$ with its first $i-k-1$ coordinates being zeros. In addition, we set $d_i^{\langle k \rangle}=\var(X_i-\hat{X}_i^{\langle k \rangle})$. Note that $\textbf{a}_i = \mathbf{a}_i^{\langle i-1 \rangle}$, $d_i=d_i^{\langle i-1 \rangle}$.
Let $\hat{\mathbf{a}}_i^{\langle k_1 \rangle}$ and $\hat{d}_i^{\langle k_1 \rangle}$ be the empirical coefficient and residual of the regression $X_i \sim \mx_{i-k_1:i-1}$, where $k_1 = \ceil{\frac{n}{c}}$ with some sufficiently large $c>1$. The coefficients with threshold $\hat{\mathbf{a}}_i^{\langle k_1 \rangle*}$ is defined as \prettyref{eq: def thresholding}. According to \prettyref{eq: def thresholding} and \prettyref{eq: def e}, we know $\hat{\mathbf{a}}_i^*=\hat{\mathbf{a}}_i^{\langle k_1 \rangle*}$, $\hat{d}_i =\hat{d}_i^{\langle k_1 \rangle}$. 

First, we prove that $\ep \abs{\hat{d}_i - d_i}^2 \leq C \efrateup$. To this end, we decompose it as follow,
\begin{align*}
\ep \abs{\hat{d}_i - d_i}^2 
&= \ep \abs{\hat{d}_i^{\langle k_1 \rangle} - d_i^{\langle i-1 \rangle}}^2\\
&\leq 2\ep \abs{\hat{d}_i^{\langle k_1 \rangle} - d_i^{\langle k_1 \rangle}}^2 + 2 \abs{d_i^{\langle k_1 \rangle} - d_i^{\langle i-1 \rangle}}^2.
\end{align*}
According to \prettyref{lmm: regression relation} we know that $\abs{d_i^{\langle k_1 \rangle} - d_i^{\langle i-1 \rangle}}^2 \leq Cn^{-2\alpha}\leq C\efrateup $, noting that $k_1 = \ceil{\frac{n}{c}}$ and $\alpha > \hf$. Besides, the regression theory implies that, $(n-k_1)\hat{d}_i^{\langle k_1 \rangle}/d_i^{\langle k_1 \rangle} \sim \chi^2(n-k_1)$. So we have
\begin{equation*}
\pb (\abs{\hat{d}_i^{\langle k_1 \rangle}/d_i^{\langle k_1 \rangle}-1} >t ) \leq 2\exp(-(n-k_1)t^2/8), \quad  t \in (0,1).
\end{equation*}
Then one can check that
\begin{equation*}
\ep \abs{\hat{d}_i^{\langle k_1 \rangle} - d_i^{\langle k_1 \rangle}}^2  \leq C \efrateup.
\end{equation*}
Therefore, we have shown that $\ep \abs{\hat{d}_i^{\langle k_1 \rangle} - d_i}^2 \leq C \efrateup$ by combining the above two equations together.

Then we turn to prove $\ep \vecnormsq{\hat{\mathbf{a}}^{*}_i- {\mathbf{a}}_i } \leq  C n^{-\frac{2\alpha+1}{2\alpha+2}}$. 
\begin{align*}
\ep \vecnormsq{\hat{\mathbf{a}}^{*}_i- {\mathbf{a}}_i } =& \ep \vecnormsq{\hat{\mathbf{a}}^{\langle k_1 \rangle*}_i- {\mathbf{a}}_i^{\langle i-1 \rangle} } \\
\leq& 2\ep \vecnormsq{\hat{\mathbf{a}}^{\langle k_1 \rangle*}_i- {\mathbf{a}}_i^{\langle k_1 \rangle} } + 2\vecnormsq{\mathbf{a}_i^{\langle i-1 \rangle} - \mathbf{a}_i^{\langle k_1 \rangle}}.
\end{align*}
According to \prettyref{lmm: regression relation}, we know that $\vecnormsq{\mathbf{a}_i^{\langle i-1 \rangle} - \mathbf{a}_i^{\langle k_1 \rangle}} \leq Cn^{-2\alpha} \leq C\efrateup$, since $\alpha > \hf$. It is sufficient to prove that $\ep \vecnormsq{\hat{\mathbf{a}}^{\langle k_1 \rangle*}_i- {\mathbf{a}}_i^{\langle k_1 \rangle} } \leq  C n^{-\frac{2\alpha+1}{2\alpha+2}}$. 

We focus on the regression coefficients $\hat{\textbf{a}}_i^{\langle k_1 \rangle}$ first.
The following analysis is conditioned on $\mz_{i-k_1:i-1}$. It is worthwhile to mention that with probability one that $(\mz_{i-k_1:i-1}^T \mz_{i-k_1:i-1})^{-1}$ exists since $k_1=\ceil{n/c}$. For any fixed $i$, we have 
\begin{equation*}
\hat{\mathbf{a}}_i^{\langle k_1 \rangle}|\mz_{i-k_1:i-1} \sim N(\mathbf{a}_i^{\langle k_1 \rangle}, (\mz_{i-k_1:i-1}^T \mz_{i-k_1:i-1})^{-1}\var (X_i|\mz_{i-k_1:i-1})).
\end{equation*}
For each coordinate in $\hat{\mathbf{a}}_i^{\langle k_1 \rangle} = (\textbf{0}, \hat{{a}}_{i(i-k_1)}^{\langle k_1 \rangle} \dots, \hat{{a}}_{i(i-1)}^{\langle k_1 \rangle})^T $,
$\hat{{a}}_{ij}^{\langle k_1 \rangle}|\mz_{i-k_1:i-1}$ with $j \in i-k_1:i-1$ can be represented as $a_{ij}^{\langle k_1 \rangle} + \sigma_j$, where 
$\sigma_j$ follows the normal distribution with variance which can be bounded as follows since $\Omega\in \paraspp$,
\begin{equation} \label{eq: threshold inequality0}
\begin{split}
\var(\sigma_j) 
&\leq \opnorm{(\mz_{i-k_1:i-1}^T \mz_{i-k_1:i-1})^{-1}} \var (X_i|\mz_{i-k_1:i-1}) \\
&\leq \eigenbd \opnorm{(\mz_{i-k_1:i-1}^T \mz_{i-k_1:i-1})^{-1}} .
\end{split}
\end{equation}

Next, we can show,  
\begin{equation} \label{eq: threshold inequality}
\abs{\hat{a}^{\langle k_1 \rangle*}_{ij}-{a}^{\langle k_1 \rangle}_{ij}} \leq \min\{ |{a}^{\langle k_1 \rangle}_{ij}| , \frac{3}{2}\lambda_{j} \} + |3\sigma_j|\indc{\abs{\sigma_j}> \frac{1}{2}\lambda_{j}}.
\end{equation}
To see this, one can check that
\begin{equation*}
\abs{\hat{a}^{\langle k_1 \rangle*}_{ij}-{a}_{ij}^{\langle k_1 \rangle}} = \max\{ |\sigma_j|\indc{|{a}_{ij}^{\langle k_1 \rangle}+\sigma_j|> \lambda_{j}} , |{a}_{ij}^{\langle k_1 \rangle}|\indc{|{a}_{ij}^{\langle k_1 \rangle}+\sigma_j|\leq  \lambda_{j}}\}.
\end{equation*}
For the first term, we have
\begin{align*}
&|\sigma_j|\indc{|{a}_{ij}^{\langle k_1 \rangle}+\sigma_j|> \lambda_j} \\
\leq& |\sigma_j|\indc{|\sigma_j|> \lambda_j/2} +   |\sigma_j|\indc{|\sigma_j|\leq  \lambda_j/2 \cap |{a}_{ij}^{\langle k_1 \rangle}|> \lambda_j/2 } \\
\leq& |\sigma_j|\indc{|\sigma_j|> \lambda_j/2} + \min\{|{a}_{ij}^{\langle k_1 \rangle}|, \lambda_j/2\}.
\end{align*}
Similarly, for the second term, we have
\begin{align*}
&|{a}_{ij}^{\langle k_1 \rangle}|\indc{|{a}_{ij}^{\langle k_1 \rangle}+\sigma_j|\leq  \lambda_j} \\
\leq& |{a}_{ij}^{\langle k_1 \rangle}| \indc{|{a}_{ij}^{\langle k_1 \rangle}+\sigma_j|\leq  \lambda_j \cap |\sigma_j|> \lambda_j/2} \\
&+ |{a}_{ij}^{\langle k_1 \rangle}| \indc{|{a}_{ij}^{\langle k_1 \rangle}+\sigma_j|\leq  \lambda_j \cap |\sigma_j|\leq \lambda_j/2}\\
\leq& |3\sigma_j|\indc{\abs{\sigma_j}>\lambda_j/2} + \min\{ |{a}_{ij}^{\langle k_1 \rangle}| , 3\lambda_j/2 \}.
\end{align*}
We finish the proof of Equation \prettyref{eq: threshold inequality}.

Equation (\ref{eq: threshold inequality}) further implies
\begin{equation} \label{eq: lemma threshold temp 1}
\begin{split}
&\ep (\vecnormsq{\hat{\mathbf{a}}_i^{\langle k_1 \rangle*} - {\mathbf{a}}_i^{\langle k_1 \rangle} } |\mz_{i-k_1:i-1})\\ 
\leq & 2\sum_{j=i-k_1}^{i-1} \min\{ |{a}_{ij}^{\langle k_1 \rangle}| , 3/2\lambda_j \}^2 + 2\sum_{j=i-k_1}^{i-1}\ep\big( (3\sigma_j)^2\indc{\abs{\sigma_j}>\lambda_j/2}|\mz_{i-k_1:i-1} \big) \\
\leq & 4\sum_{j=i-k_1}^{i-1} \min\{ |{a}_{ij}^{\langle i-1 \rangle} |, 3/2\lambda_j \}^2 + 18\sum_{j=i-k_1}^{i-1} \ep( \sigma_j^2\indc{\abs{\sigma_j}>\lambda_j/2}|\mz_{i-k_1:i-1} ) \\
&+ 4\vecnormsq{{\mathbf{a}}_i^{\langle k_1 \rangle} - {\mathbf{a}}_i^{\langle i-1 \rangle} }\\
\leq & 6\sum_{j=i-k_1}^{i-1}|{a}_{ij}^{\langle i-1 \rangle}|\lambda_j + 18\sum_{j=i-k_1}^{i-1} \ep( \sigma_j^2\indc{\abs{\sigma_j}>\lambda_j/2}|\mz_{i-k_1:i-1} ) + C\efrateup.
\end{split}
\end{equation}

Set $J_0$ as $\log_2^{k_0}$, $J_1$ as $\log_2^{k_1}$. Due to that $\Omega\in \paraspp$ and (\ref{eq: threshold inequality0}), we can show
\begin{align*}
&\sum_{j=i-k_1}^{i-1}|{a}_{ij}^{\langle i-1 \rangle}|\lambda_j \\
\leq& \sum_{k=J_0}^{J_1+1} \sqrt{({k-J_0})R}\times M(2^{-(k-J_0)\alpha} -2^{-(k+1-J_0)\alpha} )2^{-J_0\alpha} \\ 
\leq& n^{-1/2} 2^{-J_0\alpha} M\sqrt{8\eta\opnorm{n(\mz_{i-k_1:i-1}^T\mz_{i-k_1:i-1})^{-1}}} \big( \sum_{k=0}^{J_1-J_0+1} 2^{-k\alpha}\sqrt{k} \big) \\
\leq& C n^{-\frac{2\alpha+1}{2\alpha+2}}\sqrt{\opnorm{n(\mz_{i-k_1:i-1}^T\mz_{i-k_1:i-1})^{-1}}}\\
\leq& C  (\opnorm{n(\mz_{i-k_1:i-1}^T\mz_{i-k_1:i-1})^{-1}} + 1) n^{-\frac{2\alpha+1}{2\alpha+2}},
\end{align*}
and
\begin{align*}
&\sum_{j=i-k_1}^{i-1} \ep( \sigma_j^2\indc{\abs{\sigma_j}>\lambda_j/2}|\mz_{i-k_1:i-1} )\\
\leq &  \sum_{j=i-k_1}^{i-1} \ep (\eigenbd \opnorm{(\mz_{i-k_1:i-1}^T\mz_{i-k_1:i-1})^{-1}} \indc{\abs{\sigma_j}>\lambda_j/2}|\mz_{i-k_1:i-1} ) \\
\leq & \eigenbd \opnorm{(\mz_{i-k_1:i-1}^T\mz_{i-k_1:i-1})^{-1}} \sum_{j=i-k_1}^{i-1}\pb ({\abs{\sigma_j}>\lambda_j/2}|\mz_{i-k_1:i-1}) \\
\leq& n^{-1}C\opnorm{n(\mz_{i-k_1:i-1}^T\mz_{i-k_1:i-1})^{-1}} \sum_{j=i-k_1}^{i-1} \exp(-\frac{\lambda_j^2}{8\eigenbd\opnorm{(\mz_{i-k_1:i-1}^T\mz_{i-k_1:i-1})^{-1}}}) \\
\leq& n^{-1}C\opnorm{n(\mz_{i-k_1:i-1}^T\mz_{i-k_1:i-1})^{-1}} \sum_{k=J_0}^{J_1+1}  \exp(-(k-J_0))(2^k)\\
\leq& C \opnorm{n(\mz_{i-k_1:i-1}^T\mz_{i-k_1:i-1})^{-1}}n^{-\frac{2\alpha+1}{2\alpha+2}}.
\end{align*}

Plugging the above two equations in \prettyref{eq: lemma threshold temp 1}, we have
\begin{align*}
&\ep (\vecnormsq{\hat{\mathbf{a}}_i^{\langle k_1 \rangle*} - {\mathbf{a}}_i^{\langle k_1 \rangle} }) \\
=&\ep (\ep (\vecnormsq{\hat{\mathbf{a}}_i^{\langle k_1 \rangle*} - \mathbf{a}_i ^{\langle k_1 \rangle} }|\mz_{i-k_1:i-1}))\\
\leq &  C  (\ep\opnorm{n(\mz_{i-k_1:i-1}^T\mz_{i-k_1:i-1})^{-1}} + 1)n^{-\frac{2\alpha+1}{2\alpha+2}}\\ 
&+ C  \ep \opnorm{n(\mz_{i-k_1:i-1}^T\mz_{i-k_1:i-1})^{-1}} n^{-\frac{2\alpha+1}{2\alpha+2}}
+ C n^{-\frac{2\alpha+1}{2\alpha+2}} \\
\leq& C n^{-\frac{2\alpha+1}{2\alpha+2}}.
\end{align*}
The last inequality holds since $\ep \opnorm{n(\mz_{i-k_1:i-1}^T\mz_{i-k_1:i-1})^{-1}} $ can be bounded by constant, noting that $k_1=\ceil{n/c}$ with some sufficiently large $c>1$. Indeed, one can follow the strategy in the proof of Theorem 5 in \cite{cai2010optimal}, together with the concentration inequality of the sample covariance matrix under the operator norm (e.g., from Theorem 5.39 in \cite{vershynin2010introduction}).
It follows that
\begin{equation*}
\ep (\vecnormsq{\hat{\mathbf{a}}^*_i - {\mathbf{a}}_i })  \leq C n^{-\frac{2\alpha+1}{2\alpha+2}}.
\end{equation*}

\renewcommand{\theequation}{H.\arabic{equation}}
\setcounter{equation}{0}
\renewcommand{\thelemma}{H.\arabic{lemma}}
\setcounter{lemma}{0}

\section{Proofs of \prettyref{lmm: p4 subset} and \prettyref{lmm: ef low} in the analysis of \prettyref{thm: lower in f}}
In this section, we prove \prettyref{lmm: p4 subset} and \prettyref{lmm: ef low} to establish \prettyref{thm: lower in f}.
\subsection{Proof of \prettyref{lmm: p4 subset}}
Let $C_s(\theta(s)) \equiv [c(s)_{ij}]_{k \times k}$,  it is easy to check $\sum_{i-j>k}|c_{ij}| \leq M k^{-\alpha-1}$. One can check $\Omega(\theta)$ in $\calp_{4}$ has the specific form of
\begin{equation*}
\Omega(\theta) =
\begin{bmatrix}
E_1(\theta(1))   & 0_{2k} & \hdots & 0_{2k} \\
0_{2k} & E_2(\theta(2)) & \hdots & 0_{2k} \\
\vdots & \vdots &  \ddots & \vdots \\
0_{2k} & 0_{2k} & \hdots
& E_{\ceil{\frac{p}{2k}}} (\theta(\ceil{\frac{p}{2k}}))
\end{bmatrix},
\end{equation*}
where, 
\begin{equation*}
E_s=
\begin{bmatrix}
I_k+C_s(\theta(s))^TC_s(\theta(s)) & -C_s(\theta(s))^T \\
-C_s(\theta(s))  & I_k 
\end{bmatrix}.
\end{equation*}
Let $\mathbf{1}$ denote the vector with all $1$'s in $\Theta$. Then one can check
\begin{align*}
\lambda_{\rm{max}}(\Omega(\theta)) =& \big(\lambda_{\rm{max}}(I - C(\theta))\big)^2 \leq \big(\lambda_{\rm{max}}(I + C(\theta))\big)^2 {\leq}(1+\lambda_{\rm{max}}(C_s(\mathbf{1})))^2 \\=& (1 + kn^{-\hf} \tau)^2 \leq \eigenbd.
\end{align*}
The second inequality above is due to fact that the entries of $C(\theta)$ are all non-negative and \prettyref{lmm: positive}. Recall $\Sigma(\theta) = (I+C(\theta))(I+C(\theta))^T$. Therefore, we have
\begin{align*}
\lambda_{\rm{min}}(\Omega(\theta)) =& \big(\lambda_{\rm{max}}(\Sigma(\theta))\big)^{-1}= \big(\lambda_{\rm{max}}(I + C(\theta))\big)^{-2} {\geq}(1+\lambda_{\rm{max}}(C_s(\mathbf{1})))^{-2} \\
=& (1 + kn^{-\hf} \tau)^{-2} \geq  \eigenbd^{-1}. 
\end{align*}
The eigenvalues of $\Omega(\theta)$ are in the interval $[\eigenbd^{-1}, \eigenbd]$. So $\calp_4 \in \paraspp$.

\subsection{Proof of \prettyref{lmm: ef low}}  
For \prettyref{eq: lmm ef low 1}, the proof is almost the same as that of \prettyref{lmm: 1 assouad}. Since \prettyref{lmm: ef low} is about the distributions in the subset $\calp_4$, we need to recheck the inequalities \prettyref{eq: lmm low op 1} and \prettyref{eq: lmm low op 2}.
Assume that $\Sigma(\theta)=\{\Omega(\theta)^{-1}: \Omega(\theta) \in \calp_{4} \}$ and $D=\Sigma(\theta')-\Sigma(\theta)$, note that $H(\theta, \theta')=1$,  in this case, one can verify that we only need to replace $A_k^*$ by $C_s$. Specifically, $C_s$ corresponds to the one where $\theta$ and $\theta'$ are different. Then we still have
\begin{align*} 
\opnorm{D} \leq 2(1+\opnorm{C_s(\mathbf{1})})\opnorm{C_s(d)}\leq 4\tau n^{-\hf}, \\
\fnorm{D} \leq 2(1+\fnorm{C_s(\mathbf{1})})\fnorm{C_s(d)}\leq 4\tau n^{-\hf}.
\end{align*} 
So \prettyref{eq: lmm ef low 1} holds in this case. As for \prettyref{eq: lmm ef low 2}, similarly,
\begin{align*}
\fnormsq{\Omega(\theta')-\Omega(\theta)} \geq \sum_s\fnormsq{C_s(\theta'(s)-\theta(s))} &=H(\theta,\theta')(\lopelem)^2, \\
\min_{H(\theta,\theta') \geq 1}\frac{\| \Omega'(\theta)-\Omega'(\theta')\|_2^2}{H(\theta,\theta')} &\geq \tau^2n^{-1}.
\end{align*}

\renewcommand{\theequation}{I.\arabic{equation}}
\setcounter{equation}{0}
\renewcommand{\thelemma}{I.\arabic{lemma}}
\setcounter{lemma}{0}

\section{Proofs of \prettyref{thm: npara} }
In this section, we first establish \prettyref{thm: npara} based on some important lemmas, and then provide the proofs of lemmas later.

\subsection{Proof of \prettyref{thm: npara} : Minimax Optimality}

In this section, we prove that both $\tilde\Omega_k^{\tau}$ and $\tilde\Omega_k^{\rho}$ proposed in Section \ref{sec: rank} achieve the minimax optimality under the operator norm over the parameter space $\nparaspp$. The minimax optimality over $\nparaspq$ can be established in the same way. 

First, we derive the risk lower bound over $\nparaspp$. Following the same strategy in \prettyref{sec: low op l}, we construct the two subset $\calp'_1$  and $\calp'_2$ based on $\calp_1$ and $\calp_2$ given in \prettyref{eq: def p11} and \prettyref{eq: def p12}. Define the subset
\begin{equation}
\calp'_1 = \left\{\{\Omega, \{f_i\}\}: 
\begin{split}
&\Omega = \diag(\Omega'^{-1})^{\hf}\Omega'\diag(\Omega'^{-1})^{\hf}, \quad \Omega' \in \calp_1;\\ 
&f_i(x) = \diag(\Omega'^{-1})_i^{\hf}x,\quad i \in [p].
\end{split} \right\},
\end{equation}

\begin{lemma}\label{lmm: np1 subset} $\calp'_1$ is a subset of $\calp'_\alpha(\eigenbd^2, M\eigenbd)$.
\end{lemma}
One can easily check that the probability measure of the Gaussian distribution with precision matrix $\Omega$ and transformation $\{f_i\}$, where $\{\Omega, \{f_i\}\} \in \calp'_1$ is equivalent to the the probability measure of the Gaussian distribution with precision matrix $\Omega'$, where $\Omega' \in \calp_1$. Therefore, by this one-to-one correspondence of probability measure between index sets $\calp'_1$ and $\calp_1$, we immediately have
\begin{equation}
\sup_{\{\Omega, \{f_i\}\} \in \nparaspp} \ep \opnorm{\tilde{\Omega} - \Omega}^2 \geq \sup_{\{\Omega, \{f_i\}\} \in \calp'_1} \ep \opnorm{\tilde{\Omega} - \Omega}^2 =  \sup_{\Omega' \in \calp_1} \ep \opnorm{\tilde{\Omega} - \Omega}^2.
\end{equation}
Applying the same proof strategy in \prettyref{sec: low op l}, we have the following result.
\begin{lemma}\label{lmm: 2 assoaud nonpara}
	We set $\Omega = \rm{diag} (\Omega'^{-1})^{\hf}\Omega' \rm{diag}(\Omega'^{-1})^{\hf}$ for each $\Omega'(\theta) \in \calp_{1}$ in \prettyref{eq: def p11}. Then it holds that
	\begin{equation*}
	\min_{H(\theta,\theta')\geq 1}\frac{\| \Omega(\theta)-\Omega(\theta')\|_2^2}{H(\theta,\theta')} \geq (\lopelem)^2,
	\end{equation*}
	where $0 < \tau < \min\{M\eigenbd, \frac{1}{4}\eigenbd^{-2}, \eigenbd -1 \}$.
\end{lemma}
Then by Assoaud's lemma, we obtain that
\begin{equation*}
\inf_{\tilde{\Omega}} \sup_{\Omega' \in \calp_1} \ep \opnorm{\tilde{\Omega} - \Omega}^2 \geq \frac{\tau^2}{32} n^{-\frac{2\alpha-1}{2\alpha}}.
\end{equation*}
We can use a similar strategy to construct $\calp'_2$. Note that in this case we need to put $\tau a^{\hf}$ on the first sub-diagonal in $I-A$, instead of the diagonal of $\Sigma$.
Combined with the result from the Le Cam's lemma on the subset $\calp'_2$, we have
\begin{equation}\label{eq: lower nparaspp}
\sup_{\{\Omega, \{f_i\}\} \in \nparaspp} \ep \opnorm{\tilde{\Omega} - \Omega}^2 \geq C (n^{-\frac{2\alpha-1}{2\alpha}}+\frac{\log p}{n}).
\end{equation}

Next, we turn to the risk upper bound of our rank-based local cropping estimators. The risk can be decomposed into the bias terms and the variance term in the same fashion in \prettyref{sec: up op l}. Since the bias terms are deterministic and only due to the bandable structure of the Cholesky factor of the inverse correlation matrix, the upper bounds of two bias terms we derived in \prettyref{lmm: tp omega close} and \prettyref{lmm: bias in block up lop} still hold. For the variance term, a simple extension of Theorem 1 in \citep{mitra2014multivariate} provides the following result.
\begin{lemma}\label{lmm: cor upper lop sample precision}
	For any $\Omega$ such that $\diag(\Omega^{-1}) = I$ and $\Omega \in \paraspp\cup\paraspq$, we have 
	\begin{equation*}
	\ep \big( \max_{m \in [p]} \opnormsq{(\cut{{3k}}{m-k}( \hat\Sigma^\tau ) - \cut{{3k}}{m-k}( \Omega^{-1} ) } \big) \leq C\frac{\log p + k}{n},
	\end{equation*}
	\begin{equation*}
	\ep \big( \max_{m \in [p]} \opnormsq{(\cut{{3k}}{m-k}( \hat\Sigma^\rho ) - \cut{{3k}}{m-k}( \Omega^{-1} ) } \big) \leq C\frac{\log p + k}{n}.
	\end{equation*}
\end{lemma}
Replacing \prettyref{lmm: sample cov max} by \prettyref{lmm: cor upper lop sample precision} and following the rest of the proof in \prettyref{thm: up op 1}, we finally obtain that
\begin{equation} \label{eq: up nparaspp}
\sup_{\{\Omega, \{f_i\}\} \in \nparaspp} (\ep \opnorm{\tilde{\Omega}_k^\tau - \Omega}^2 +\ep \opnorm{\tilde{\Omega}_k^\rho - \Omega}^2 )\leq Cn^{-\frac{2\alpha-1}{2\alpha}} + C\frac{\log p}{n}.
\end{equation}
The lower bound \prettyref{eq: lower nparaspp} and upper bound \prettyref{eq: up nparaspp} together give the optimal rate of convergence \prettyref{eq: optimality nparaspp} in \prettyref{thm: npara}. The optimal rate of convergence \prettyref{eq: optimality nparaspq} in \prettyref{thm: npara} can be proved similarly. Therefore, we establish the minimax framework for nonparanormal distributions.

\subsection{Proof of \prettyref{lmm: np1 subset}}
It is easy to show $\diag(\Omega^{-1}) = I$. Now we need to verify $\Omega \in \calp_\alpha(\eigenbd^2, M\eigenbd)$.
Let $S$ denote $(\diag(\Omega'^{-1}))^{\frac{1}{2}}$.
The Cholesky decomposition of $\Omega$ is:
\begin{align*}
\Omega&=S\Omega' S\\
&= S(I-A')^TD'^{-1}(I-A')S\\
&= (S-A'S)^TD'^{-1}(S-A'S)\\
&= (I-S^{-1}A'S)^T SD'^{-1}S (I-S^{-1}A'S).
\end{align*}
According to \prettyref{lmm: prop of paraspp}, $\eigenbd^{-\hf} \leq \lambda_{\min}(S) \leq \lambda_{\max}(S) \leq \eigenbd^{\hf}$ and $\eigenbd^{-1} \leq \lambda_{\min}(\Omega') \leq \lambda_{\max}(\Omega') \leq \eigenbd$. Therefore, we obtain $\eigenbd^{-2} \leq \lambda_{\min}(\Omega) \leq \lambda_{\max}(\Omega) \leq \eigenbd^2$. Let $A\equiv [a_{ij}]_{p\times p} = S^{-1}A'S$. The desired result $\Omega \in \mathcal{P}_\alpha(\eigenbd^2, M \eigenbd)$ then immediately follows from that $\max_i \sum_{j< i-k} |a_{ij}| < M\eigenbd k^{-\alpha} $ for $k \in [i-1]$.

Since $\diag(\Omega'^{-1})_i^{\hf} > 0$, $f_i$ is a strictly increasing function. Therefore, $\{\Omega, \{f_i\}\} \in \mathcal{P}'_\alpha(\eigenbd^2, M \eigenbd)$.

\subsection{Proof of \prettyref{lmm: 2 assoaud nonpara}}
Set $I_k+A^{*}_k(\theta)^T A^*_k(\theta)$ as $W(\theta)$. One can check
\begin{equation*}
\Omega(\theta)=\begin{bmatrix}
W(\theta)  & -A^{*}_k(\theta)^T\diag(W(\theta))^{\hf} & 0_{k \times (p-2k)}\\
- \diag(W(\theta))^{\hf}A^*_k(\theta) & \diag(W(\theta)) & 0_{k \times (p-2k)}\\
0_{(p-2k) \times k} & 0_{(p-2k) \times k} & I_{p-2k}\\
\end{bmatrix}.
\end{equation*}
Then,
\begin{align*}
\opnormsq{\Omega(\theta')-\Omega(\theta)} &\geq \opnormsq{\diag(W(\theta))^{\hf}A^*_k(\theta) - \diag(W(\theta'))^{\hf}A^*_k(\theta')}\\
&=H(\theta, \theta')(1+(\lopelem)^2)(\lopelem)^2,
\end{align*}
which further implies
\begin{equation*}
\min_{H(\theta,\theta') \geq 1}\frac{\| \Omega(\theta)-\Omega(\theta')\|_2^2}{H(\theta,\theta')} \geq (\lopelem)^2.
\end{equation*}

\renewcommand{\theequation}{J.\arabic{equation}}
\setcounter{equation}{0}
\renewcommand{\thelemma}{J.\arabic{lemma}}
\setcounter{lemma}{0}

\section{Proof of \prettyref{lmm: prob op p} in the analysis of \prettyref{thm: adap op p}}

This proof is similar to that of \prettyref{thm: up op 1}. For any bandwidth $k$, according to \prettyref{eq: up lop 0}, we have
\begin{equation*}
\opnormsq{\tilde{\Omega}_k^{\op}-\Omega} \leq 8 \opnormsq{\tilde{\Omega}^*_k-\Omega^*_k} + 8\opnormsq{\Omega^*_k - \Omega}. 
\end{equation*}
\prettyref{lmm: tp omega close} indicates that
\begin{equation*}
\opnormsq{\Omega^*_k - \Omega} \leq Ck^{-2\alpha + 1}.
\end{equation*}
Equations \prettyref{eq: up lop 1} - \prettyref{eq: up lop 3} together show that
\begin{align*}
\opnormsq{\tilde{\Omega}^*_k-\Omega^*_k}
\leq& 16 \eta^2  \max_{m \in [p]} \big( \opnorm{\cut{{6k}}{m}( \frac{1}{n}\mz\mz^T ) - \cut{{6k}}{m}( \Omega^{-1} ) } \\
&+  \opnorm{\cut{k}{k+1}\big( (\cut{{3k}}{m}( \Omega^{-1} )^{-1}) \big)-\cut{k}{m}(\Omega)}^2 \big).
\end{align*}
Note that \prettyref{lmm: bias in block up lop} indicates
\begin{equation*}
\opnormsq{\cut{{k}}{k+1}\big( (\cut{{3k}}{m}( \Omega^{-1} )^{-1}) \big) - \cut{k}{m}(\Omega) }\leq Ck^{-2\alpha +1}.
\end{equation*}
In addition, the probability version of \prettyref{lmm: sample cov max} (its proof can be found in Lemma 3 of \citep{cai2010optimal}) indicates that for any $C_1>0$, one can find a sufficiently large $C>0$ irrelevant of $\alpha$ such that
\begin{equation*}
\pb \big( \max_{m \in [p]} \opnormsq{(\cut{{6k}}{m}( \frac{1}{n}\mz\mz^T ) - (\cut{{6k}}{m}( \Omega^{-1} ) } \leq C\frac{\log p + k}{n} \big) \geq 1- \exp\big(-C_1(\log p + k)\big).
\end{equation*}
Combining the above arguments, we derive the desired result in \prettyref{lmm: prob op p}.

\renewcommand{\theequation}{K.\arabic{equation}}
\setcounter{equation}{0}
\renewcommand{\thelemma}{K.\arabic{lemma}}
\setcounter{lemma}{0}

\section{Additional Numerical Studies}
In this section, we provide additional numerical studies to demonstrate the performance of rank-based estimator $\tilde\Omega_k^{\tau}$ proposed in Section \ref{sec: rank} for the nonparanormal model. In addition, we check the robustness of the proposed rate-optimal estimators towards model misspecification under the operator norm. In the end, we verify the minimax rates under the operator norm by fixing $n$ and varying $p$, or fixing $p$ and varying $n$.

\subsection{Simulation for the Nonparanormal Model}

We design the following experiment to show the performance of our nonparanormal local cropping estimator $\tilde\Omega_k^{\tau}$. Since our method is the first attempt to estimate precision matrices with bandable Cholesky factor under the nonparanormal model, we only compare ours with some method (normal-score estimator proposed in \cite{liu2009nonparanormal}) designed for a close setting under the nonparanormal model, i.e., estimation of sparse precision matrices. 

We generate centered multivariate normal vector $\textbf{X}$  with the same Cholesky factor structure used in Section \ref{sec: simulation eop} from $\paraspq$. After that, we transform each entry of $\textbf{X}$ with three different functions: identity ($y = x$), cubic ($y = x^3$) and step function ($y = x^3 +1$ when $x\geq 0$ and $y=x^3-1$ when $x<0$) to obtain the random vector $\textbf{Y}$, which follows the nonparanormal distribution. The simulation is done with a range of parameter values for $p$, $n$, $\alpha$ ($n$, $p = 200, 500, 1000$; $\alpha = 1, 1.5, 2$). The bandwidth of our method is chosen as $\floor{n^{1/(2\alpha + 1)}}$ while normal-score estimator is obtained using ``huge" package \cite{zhao2012huge} with default tuning method.

Our method is used to estimate inverse of the correlation matrix under the nonparanormal model. To make the results comparable to those in Section \ref{sec: simulation eop} for estimating precision matrices, we re-scale the estimators to make them have the same diagonals with the underlying precision matrix. Since rank-based methods are invariant under monotonic transformation, our local cropping method generates the same result for all three functions. In comparison, normal-score estimator results in different performance under three transformation functions.  

As indicated in the Table \ref{tab: simulation 3} below, our nonparametric local cropping method shows better performance over normal-score estimator in most cases, especially when $\alpha$ is large. This is as expected since our method takes advantage of the order structure in Cholesky factor. Besides, since our rank-based method only relies on the rank information, the resulting estimation errors are worse than those in the normal model shown in Section \ref{sec: simulation eop}. 

\subsection{Robustness towards Model Misspecification}

To evaluate the robustness of our method towards model misspecification, we test our estimator under the scenarios which sightly violate our bandable Cholesky factor structure. More specifically, we create the same Cholesky factor structure used in Section \ref{sec: simulation eop} from $\paraspq$ as the base, and permute the entries within each row for a few rows of Cholesky factors to reflect the model misspecification. For each $i = 1,2,3$, we choose the $\floor{j2^{-i}p}$th row with equal distance, $j = 1,\dots,2^i-1$ to do the permutation. We also add the result without model misspecification as the reference. Again, the simulation is done with a range of parameter values for $n$, $p$ and $\alpha$. The performance is reported in Table \ref{tab: simulation 4}. It can be seen that when $\alpha=0.5$, the effect of model misspecification is not significant as the decay rate within each row is small, which implies that the permutation would not change the structure too much. When $\alpha$ becomes larger, the overall performance of our estimator is worse compared to that of the base model. This is due to a more severe change of the bandable structure. In addition, we can see that when $\alpha=1$ or $1.5$, the performance becomes even worse as we have more rows permuted. Of note, when $\alpha=0.5$, our estimator under model misspecification still outperforms the banding estimator implemented in Section \ref{sec: simulation eop}. Overall, the trend reflected in this numerical study meets our expectations. Model misspecification will hurt the performance of our estimator since our procedure heavily relies on the bandable structure. However, our method still maintains robustness to some extent when the structure is only slightly violated. 
\subsection{Verification of Minimax Rates under the Operator Norm}
Building upon the numerical studies in both Section \ref{sec: simulation eop} and Section \ref{sec: simulation lop}, we provide an additional sets of simulation for $p=250$,  $\alpha=1, 1.5 ,2$, and $n=500, 1000, 2000, 4000$ with the same bandable Cholesky structures considered in Sections \ref{sec: simulation eop}-\ref{sec: simulation lop}. Below are figures to show the trend of loss under the squared operator norm, with fixed $p$ and varying $n$ (Figure \ref{fig:data1} for $\paraspq$ and Figure \ref{fig:data3} for $\paraspp$), or with fixed $n$ and varying $p$ (Figure \ref{fig:data2} for $\paraspq$ and Figure \ref{fig:data4} for $\paraspp$). We observe in Figures \ref{fig:data1}-\ref{fig:data3} that the error increases linearly as $n^{-\frac{2\alpha}{2\alpha+1}}$ or $n^{-\frac{2\alpha-1}{2\alpha}} $ increases for $\paraspq$ and $\paraspp$ respectively, confirming the first term shown in the optimal rates of convergence. In addition,  Figures \ref{fig:data2}-\ref{fig:data4} are provided to try to confirm the second term $(\log p)/n$ shown in the optimal rates of convergence for $\paraspq$ and $\paraspp$ respectively. It seems that only when $n=500$, the plots in Figure \ref{fig:data2} confirm such a linear relationship for  $\paraspq$. This is probably because (1) that as $n$ increases to 1000 or larger, the second term $(\log p)/n$ is dominated by the first term $n^{-\frac{2\alpha}{2\alpha+1}}$ for $\paraspq$; and (2) that even when $n=500$, the term $(\log p)/n$ is dominated by $n^{-\frac{2\alpha-1}{2\alpha}} $ for $\paraspp$, which has a slower optimal rate than that for $\paraspq$.
\begin{landscape}
	\begin{table}
		\centering
		\caption{The average errors under the operator norm of the nonparanormal local cropping estimator (npn-crop.Q) and normal-score estimator under entry-wise three transformation functions over 100 replications.}
		\label{tab: simulation 3}
		\begin{tabular}{cccccccccccccc}
			\hline
			\\[-1em]
			\multirow{2}{*}{$p$} & \multirow{2}{*}{$n$} & \multicolumn{4}{c}{$\alpha$ = 1}             & \multicolumn{4}{c}{$\alpha$ = 1.5}  & \multicolumn{4}{c}{$\alpha$ = 2}    \\
			\\[-1em]
			\cline{3-14} 
			\\[-1em]
			&                      & npn-crop.Q             & identity           & cubic & step  & npn-crop.Q             & identity  & cubic & step  & npn-crop.Q             & identity  & cubic & step\\
			\\[-1em]
			\hline
			\\[-1em]
			& 200                  & \textbf{5.75} & 6.09          & 6.24  & 6.75 & \textbf{4.46} & 4.62 & 4.80  & 5.39 & \textbf{3.94} & 4.08 & 4.22  & 4.81 \\
			\\[-1em]
			200                  & 500                  & 5.74          & \textbf{5.65} & 5.92  & 6.78 & \textbf{4.47} & 4.48 & 4.64  & 5.41 & \textbf{3.96} & 4.07 & 4.17  & 4.81 \\
			\\[-1em]
			& 1000                 & 5.74          & \textbf{5.64} & 5.89  & 6.77 & \textbf{4.47} & 4.47 & 4.62  & 5.40 & \textbf{3.96} & 4.06 & 4.16  & 4.81 \\
			\hline
			\\[-1em]
			& 200                  & \textbf{5.80} & 6.31          & 6.50  & 6.81 & \textbf{4.47} & 4.93 & 5.02  & 5.36 & \textbf{3.96} & 4.30 & 4.41  & 4.81 \\
			\\[-1em]
			500                  & 500                  & \textbf{5.79} & 5.94          & 6.12  & 6.83 & \textbf{4.48} & 4.49 & 4.76  & 5.42 & \textbf{3.97} & 4.08 & 4.19  & 4.82 \\
			\\[-1em]
			& 1000                 & 5.79          & \textbf{5.70} & 5.96  & 6.83 & \textbf{4.48} & 4.48 & 4.64  & 5.41 & \textbf{3.96} & 4.07 & 4.17  & 4.82 \\
			\hline
			\\[-1em]
			& 200                  & \textbf{5.82} & 6.46          & 6.56  & 6.83 & \textbf{4.48} & 4.97 & 5.12  & 5.37 & \textbf{3.97} & 4.09 & 4.50  & 4.76 \\
			\\[-1em]
			1000                 & 500                  & \textbf{5.81} & 5.96          & 6.28  & 6.84 & \textbf{4.49} & 4.68 & 4.81  & 5.41 & \textbf{3.98} & 4.08 & 4.30  & 4.82 \\
			\\[-1em]
			& 1000                 & 5.81          & \textbf{5.72} & 5.99  & 6.85 & \textbf{4.48} & 4.49 & 4.65  & 5.42 & \textbf{3.97} & 4.07 & 4.18  & 4.82 \\ 
			\hline
		\end{tabular}
	\end{table}
	\begin{table}
		\centering
		\caption{The average errors under the operator norm of the local cropping estimator with ($i=1,2,3$) and without ($i=0$) model misspecification over 100 replications.}
		\label{tab: simulation 4}
		\begin{tabular}{cccccccccccccc}
			\hline
			\\[-1em]
			\multirow{2}{*}{$p$} & \multirow{2}{*}{$n$} & \multicolumn{4}{c}{$\alpha$ = 0.5} & \multicolumn{4}{c}{$\alpha$ = 1} & \multicolumn{4}{c}{$\alpha$ = 1.5} \\ \cline{3-14}
			\\[-1em]
			&  & $i$ = 0 & $i$ = 1 & $i$ = 2 & $i$ = 3 & $i$ = 0 & $i$ = 1 & $i$ = 2 & $i$ = 3 & $i$ = 0 & $i$ = 1 & $i$ = 2 & $i$ = 3 \\ \hline
			\\[-1em]
			& 500 & 4.68 & 4.64 & 4.64 & 4.70 & 1.64 & 1.92 & 2.03 & 2.16 & 1.18 & 1.89 & 1.99 & 2.13 \\
			\\[-1em]
			500 & 1000 & 3.29 & 3.30 & 3.26 & 3.25 & 1.17 & 1.90 & 2.02 & 2.15 & 0.82 & 1.9 & 1.97 & 2.15 \\
			\\[-1em]
			& 2000 & 2.47 & 2.47 & 2.47 & 2.49 & 0.89 & 1.88 & 2.03 & 2.15 & 0.59 & 1.88 & 1.99 & 2.14 \\ \hline
			\\[-1em]
			& 500 & 4.96 & 4.95 & 4.93 & 4.88 & 1.75 & 1.92 & 2.00 & 2.13 & 1.3 & 1.92 & 2.00 & 2.13 \\
			\\[-1em]
			1000 & 1000 & 3.43 & 3.51 & 3.43 & 3.45 & 1.24 & 1.93 & 2.01 & 2.14 & 0.86 & 1.92 & 2.00 & 2.13 \\
			\\[-1em]
			& 2000 & 2.58 & 2.56 & 2.58 & 2.57 & 0.93 & 1.93 & 2.00 & 2.10 & 0.62 & 1.92 & 2.01 & 2.12 \\ \hline
		\end{tabular}
	\end{table}
\end{landscape}

\begin{figure}
	\centering
	\includegraphics[width=1\textwidth]{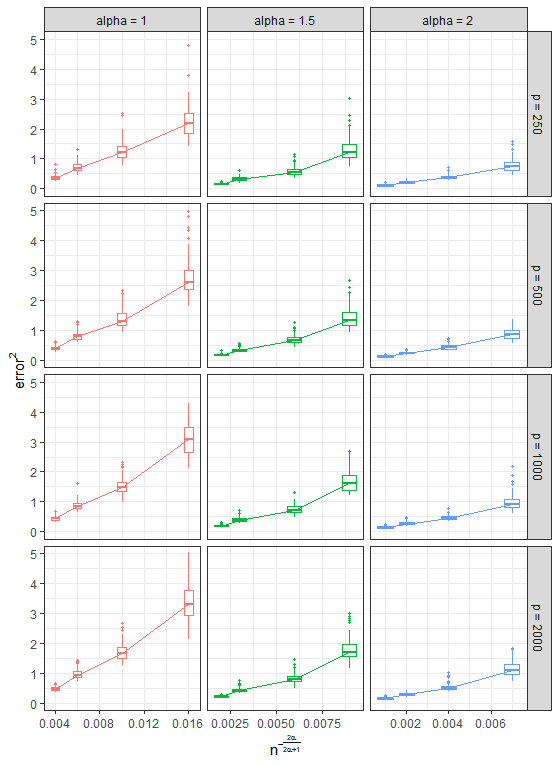}
	\caption{\label{fig:data1} Average error under squared operator norm as $n^{-\frac{2\alpha}{2\alpha+1}}$ increases for $\paraspq$.}
\end{figure}
\begin{figure}
	\centering
	\includegraphics[width=1\textwidth]{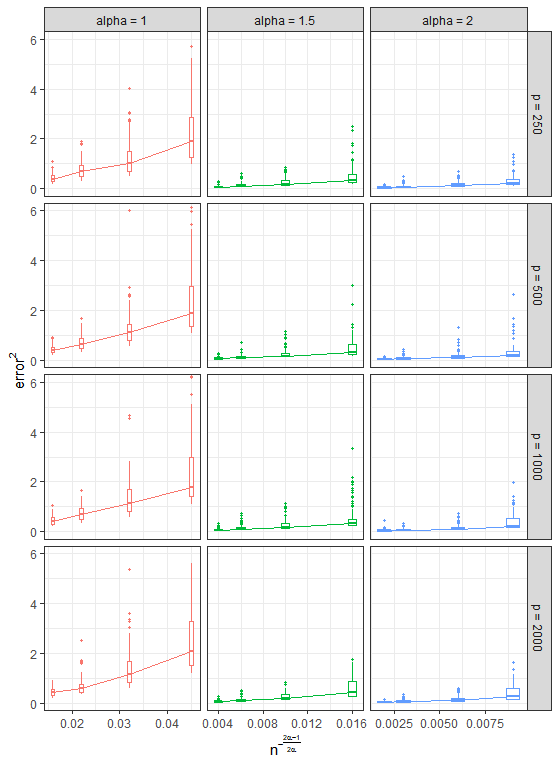}
	\caption{\label{fig:data3} Average error under squared operator norm as $n^{-\frac{2\alpha-1}{2\alpha}}$ increases for $\paraspp$.}
\end{figure}

\begin{figure}
	\centering
	\includegraphics[width=1\textwidth]{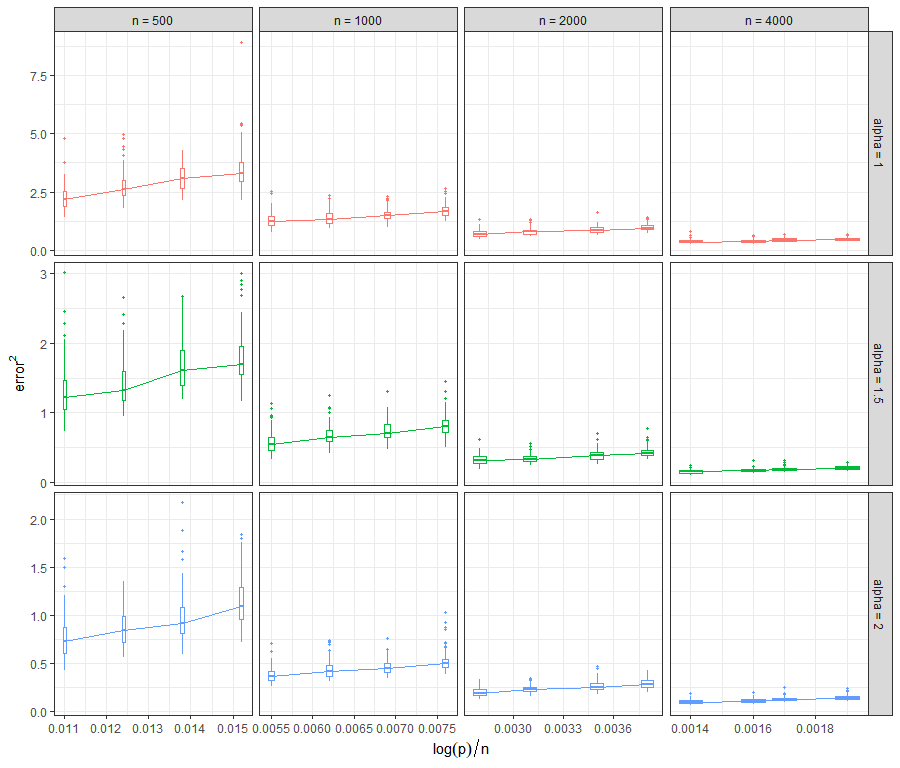}
	\caption{\label{fig:data2} Average error under squared operator norm as $\log(p)/n$ increases for $\paraspq$.}
\end{figure}
\begin{figure}
	\centering
	\includegraphics[width=1\textwidth]{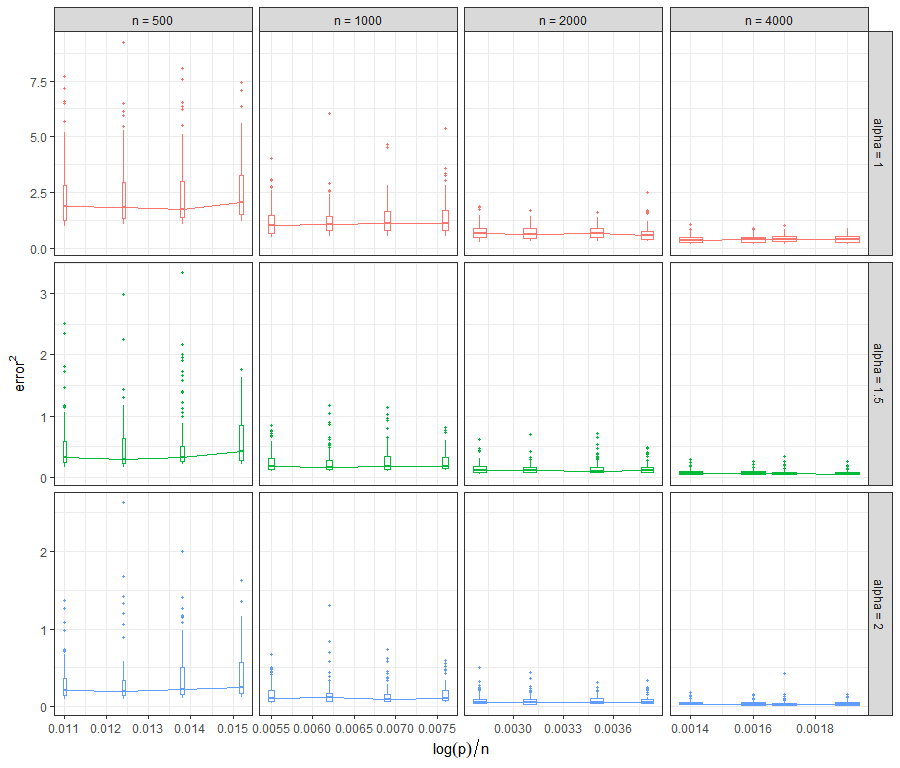}
	\caption{\label{fig:data4} Average error under squared operator norm as $\log(p)/n$ increases for $\paraspp$.}
\end{figure}

\end{document}